\documentclass[oneside]{amsart}

   \usepackage{amsthm}
   \usepackage{amsmath}
   \usepackage{amssymb}
   \usepackage{amsfonts}
   \usepackage{amscd}
   \usepackage{mathrsfs}  
   \usepackage{bm} 
   \usepackage[all]{xy}
   \makeatletter
    
    \@addtoreset{equation}{section}
  \makeatother   
    \newtheorem{theorem}{Theorem}[section]
  \newtheorem{definition}[theorem]{Definition}
  \newtheorem{remark}[theorem]{Remark} 
  \newtheorem{lemma}[theorem]{Lemma}
  \newtheorem{proposition}[theorem]{Proposition}
  \newtheorem{corollary}[theorem]{Corollary}

  \newtheorem{assumption}[theorem]{Assumption}
  \newtheorem{notation}[theorem]{Notation}
  \newtheorem{conjecture}[theorem]{Conjecture}
   \def\L{\mathscr{L}}  
   \def\i{{\tt i}}
   
   \newcommand{\C}{\mathbb{C}}
  \newcommand{\Z}{\mathbb{Z}}
  
  \newcommand{\R}{\mathbb{R}}
  \newcommand{\D}{\mathbb{D}}
  \newcommand{\simeqto}{\xrightarrow{\sim}}
  \newcommand{\into}{\hookrightarrow}
  \newcommand{\onto}{\twoheadrightarrow}
  \newcommand{\id}{\mathrm{id}}
  \newcommand{\End}{\mathrm{End}}
  \newcommand{\gr}{\mathrm{gr}}
  \newcommand{\Hom}{\mathrm{Hom}}
  \newcommand{\HOM}{\mathscr{H}om}
  
  \newcommand{\Sm}{\mathrm{Sm}}
  \newcommand{\Sf}{\mathrm{Sf}}
  
  \newcommand{\bil}[2]{[#1,#2\rangle} 
  \newcommand{\is}{C} 
  \newcommand{\idx}{c} 
  \newcommand{\idxx}{{c'}} 
  \newcommand{\pol}{(V,\bil{\cdot}{\cdot})} 
  \newcommand{\poll}{(V',\bil{\cdot}{\cdot}')} 
  \newcommand{\polp}{(V_\idx,\bil{\cdot}{\cdot}_\idx)_{\idx\in\is}}
  
   \newcommand{\coho}[1]{H^\bullet(#1)}
   \newcommand{\Eff}[1]{\mathrm{Eff}(#1)}

   \newcommand{\rep}{{\sf Rep}}
   \newcommand{\st}{{\sf St}}
   \newcommand{\stp}{{\sf StP}}
   \newcommand{\std}{{\sf Std}}
   \newcommand{\me}{{\sf Me}}
   \newcommand{\mep}{{\sf MeP}}

 \newcommand{\im}{\mathrm{Im}}
 \newcommand{\A}{\mathcal{A}}
 \newcommand{\B}{\mathcal{B}}
 \newcommand{\E}{\mathcal{E}}
 \newcommand{\F}{\mathcal{F}}
 \newcommand{\oo}{\mathcal{O}}
 \newcommand{\Ch}{\mathrm{Ch}}
 \newcommand{\Td}{\mathrm{Td}}
 \newcommand{\ima}{\mathrm{Im}}
 \newcommand{\Deg}{\mathrm{deg}}
 \newcommand{\lperp}{{}^\perp\!}
 \newcommand{\HH}[1]{\mathrm{HH}_\bullet \! \left( #1 \right)}
 \newcommand{\lua}{{}^A\!}
 \newcommand{\lub}{{}^B\!}
 \newcommand{\fra}[2]{\{\gr_\idx D^b(#1)\}_{\idx\in #2}}
 \newcommand{\frs}[1]{(\{\A_i\}_{1\le i \le m}, \{F_\idx\}_{\idx\in #1})}
 \newcommand{\rfrs}{\big(\{(\R_{i+1}\A_\bullet)_j\}_{1\le j \le m},
                              \{(\R_{i+1}F_\bullet)_\idx\}_{\idx \in \is}\big)}
 \newcommand{\rms}{(\R_{i+1}\A_\bullet)}
 \newcommand{\lfrs}{\big(\{(\mathbb{L}_i\A_\bullet)_j\}_{1\le j \le m},
                             \{(\mathbb{L}_i F_\bullet)_\idx\}_{\idx \in \is}\big)}
 \newcommand{\lms}{(\mathbb{L}_i\A_\bullet)}
 \newcommand{\tat}{2\pi{\tt i}}
 \newcommand{\di}{d}
 \newcommand{\e}{\mathrm{e}}
 \newcommand{\amut}[2]{\mathfrak{A}_{#2}\big(\mathrm{RH}(\mathscr{M}_{#1}, \mathscr{Q}_{#1})\big)} 
 \newcommand{\bmut}[2]{\mathfrak{B} \big( \frs \big)}
 
 \newcommand{\amb}{H^\bullet_{\mathrm{amb}}(X)}

\begin{document}
\title{An analogue of Dubrovin's conjecture}
\author{Fumihiko Sanda}
\address{Graduate school of Mathematics, Nagoya University, Furo-cho, Chikusa-ku, Nagoya 464-8602, Japan}
\email{fumhiko.sanda@gmail.com}
\author{Yota Shamoto}
\address{Kavli Institute for the Physics and Mathematics of the Universe (WPI), University of Tokyo, 5-1-5 Kashiwanoha, Kashiwa, Chiba 277-8583, Japan}
\email{yota.shamoto@ipmu.jp}

        \begin{abstract}
We propose an analogue of Dubrovin's conjecture for the case where Fano manifolds have quantum connections of exponential type. It includes the case where the quantum cohomology rings are not necessarily semisimple. The conjecture is described as an isomorphism of two linear algebraic structures, which we call ^^ ^^ mutation systems". Given such a Fano manifold $X$, one of the structures is given by the Stokes structure of the quantum connection of $X$, and the other is given by a semiorthogonal decomposition of the derived category of coherent sheaves on $X$. We also prove the conjecture for a class of smooth Fano complete intersections in a projective space.
        \end{abstract}
\maketitle
        
        \setcounter{tocdepth}{1}
       
          \section{Introduction}
          The purpose of this paper is to propose an analogue of Dubrovin's conjecture \cite{dubgeo} 
          in a more general setting. 
          We shall firstly recall the original conjecture of Dubrovin in \S \ref{intro dub}.
          We also review Gamma conjecture of Galkin-Golyshev-Iritani \cite{ggi}, \cite{gi} in \S \ref{intro gamma} 
          since their result plays a key role in this paper.
          Then we explain the outline of the formulation and the main result 
          in \S \ref{intro mutation system}-\S \ref{intro main theorem}.
          \subsection{Dubrovin's conjecture}\label{intro dub}
          Let $X$ be a Fano manifold.
          B. Dubrovin predicted some relations between 
          the derived category $D^b(X)$ of bounded complexes of coherent sheaves on $X$ 
          and the quantum cohomology ring of $X$. 
          More precisely, 
           he conjectured that            
           $D^b(X)$ has a full exceptional collection 
           if and only if
           the quantum cohomology ring of $X$
           is (generically) semisimple.
           This conjecture is proved for many examples 
           \cite{baysem}, \cite{chaqua3}, \cite{cioont},  \cite{galdub},
           \cite{herupd}, \cite{tabfro}, etc.    
          
          In the case where these two conjecturally equivalent conditions hold, 
          he also predicted a relationship between full exceptional collections of $D^b(X)$ and 
          the quantum connection associated to the quantum cup product. 
          To be more precise, let us roughly recall the definition of the quantum connection.          
          For simplicity, we take the quantum parameter $\tau$ to be $0$.
          
          Let $\mathcal{H}_X:=\coho{X}\otimes \mathcal{O}_{\C_z}$ be the 
          trivial ($\Z/2\Z$-graded) $\mathcal{O}_{\C_z}$-module. 
          Here, $\C_z$ denotes the complex plane whose coordinate function is $z$,
          and $\mathcal{O}_{\C_z}$ denotes the sheaf of algebraic functions on $\C_z$.
          The quantum connection 
          $\nabla:\mathcal{H}_X\to\mathcal{H}_X\otimes\Omega^1_{\C_z}(\log \{0\})\otimes\mathcal{O}_{\C_z}(\{0\})$
          is defined by
          \begin{align*}
              \nabla:={d}-\Big{(}\frac{1}{z}(c_1(X)*_0)-{\mu}\Big{)}\frac{dz}{z},
             \end{align*}          
          where $c_1(X)$ is the first Chern class of $X$, 
          $*_0$ is the quantum cup product with respect to the quantum parameter $\tau=0$,
          and $\mu$ is the grading operator (see Definition \ref{def of qdm} for details).
          
          The quantum connection has a regular singularity at $z=\infty$, and an irregular singularity at $z=0$. 
          If the quantum cohomology ring is semisimple, 
          we have a matrix called a Stokes matrix. 
          We consider the following form of Dubrovin's conjecture:
          \begin{conjecture}\label{conject}
          If the quantum cohomology ring of $X$ is semisimple, 
          then there exists a full exceptional collection $E_1,\dots, E_m$ 
          of $D^b(X)$ such that 
          the Stokes matrix of $(\mathcal{H}_X,\nabla)$ at $z=0$ is equal to the matrix $(\chi(E_i, E_j))_{i, j}$, 
          where $\chi(E,F):=\sum_k(-1)^k\dim \mathrm{Hom}(E,F[k])$ for $E,F\in D^b(X)$.
          \end{conjecture}
          
          Conjecture $\ref{conject}$ has been proved for some $X$ with semisimple quantum cohomology rings
         \cite{morsto},  \cite{guzsto},  \cite{iwasto}, \cite{uedsto2}, \cite{uedsto1}, etc.
%
%
%
%
          \subsection{Gamma conjectures}\label{intro gamma}
          Gamma conjecture II proposed by Galkin-Golyshev-Iritani \cite{ggi} refines Conjecture \ref{conject}.
          The conjecture is described as a coincidence of two kinds of cohomology classes of $X$: 
          Asymptotic classes and Gamma classes. 
                   
          Let $C_X$ be the set of eigenvalues of the linear operator $c_1(X)*_0$ on $\coho{X}$.
          Fix a real number $\theta_\circ$
          (for simplicity of notation, we assume $0\leq \theta_\circ < 2\pi$).  
          We assume that $\theta_\circ$ is generic with respect to $-C_X$ in a certain sense
          (see \S \ref{stofil}).
          Then we have an ordering $\tau_{\theta_\circ}:-C_X\simeqto \{1,\dots, m\}$.  
          If the quantum cohomology ring of $X$ is semisimple, 
          the space of solutions on the sector
          $ \{z\mid |\arg z-\theta_\circ|<\pi/2+\varepsilon \}$ with sufficiently small $\varepsilon>0$          
          have a basis $(y_{i})_{i=1}^m$. 
          The basis $(y_i)_{i=1}^m$ is characterized by their asymptotic growth 
          as $z\to 0$,
          and $y_i$ corresponds to $\exp(c_i/z)$ where $c_i:=\tau_{\theta_\circ}^{-1}(i)$.
          Since the fundamental solution at infinity identifies this space of solutions with $\coho{X}$
          (see Proposition \ref{fundamental solution} or \cite[Proposition 2.3.1]{ggi}), 
          the basis $(y_i)_i$ gives a basis $(A_{ i})_{i=1}^m$ of $\coho{X}$.
          The classes $A_{ i}$ are called asymptotic classes
          (see \cite[\S 4.5]{ggi} for more precise).
          
          The Gamma class \cite{iriint}, \cite{kkp1}, \cite{libche} of $X$ is defined by 
          $$ \hat\Gamma_X:=\prod_{j=1}^{\dim X}\Gamma(1+\delta_j)$$
          where $\delta_1,\dots,\delta_{\dim X}$ are the Chern roots of the tangent bundle of $X$
          and $\Gamma(x)$ is the Gamma function.
          Gamma conjecture II states that            
           there exists a full exceptional collection 
          $E_1,\dots, E_m$ of $D^b(X)$ such that 
          $$A_i=\hat\Gamma_X \mathrm{Ch}(E_i)$$ for all $i$
          (under the semisimplicity of the quantum cohomology ring 
          and the existence of at least one full exceptional collection of $D^b(X)$).  
                    
          This conjecture refines Conjecture \ref{conject} in the following sense.
          Let $[\cdot,\cdot)_X:\coho{X}\otimes\coho{X}\to\C$ be 
          a non-symmetric linear map defined by 
          \begin{align*}
              [\alpha,\beta)_X:=\frac{1}{(2\pi)^{\dim X}}\int_Xe^{\pi\i\mu}e^{-\pi\i\rho_{\!{}_X}}\alpha \cup \beta
             \end{align*}           
          where $\rho_{\!{}_X}$ is the cup product $c_1(X)\cup$.
          Then, on the one hand, the Stokes matrix coincides with $([A_i,A_j)_X)_{i, j}$. 
          On the other hand, we have
          $\chi(E_i,E_j)=\big{[}\hat{\Gamma}_X\mathrm{Ch}(E_i),\hat{\Gamma}_X\mathrm{Ch}(E_j)\big{)}_X$.

          \subsection{Mutation systems}\label{intro mutation system}
          The goal of this paper is to give an analogue of Conjecture \ref{conject} in the case 
          where the quantum cohomology ring is not necessarily semisimple.
          To do this, we introduce a notion of mutation systems.  
          For a finite dimensional vector space $V$ over a field ${\bm k}$, 
          we call a pairing $[\cdot,\cdot\rangle: V\otimes V\to {\bm k}$ of $V$
          non-degenerate if
          the induced map $V\ni v\mapsto (w\to [v, w\rangle)\in V^\vee=\mathrm{Hom}(V,\bm{k})$  
          is isomorphic. 
          We often assume non-degenerateness of the pairings without a mention. 
          
          A mutation system is a tuple of a
          (finite dimensional) vector space $V$ with a pairing $[\cdot,\cdot\rangle$,
          a family of  vector spaces with pairings $(V_c,\bil{\cdot}{\cdot}_c)_{c\in C}$ indexed 
          by $C$, a bijection $\tau:C\to\{1,\dots, m\}$,
          and an isomorphism of vector spaces $f:\bigoplus_cV_c\simeqto V$
          with some conditions (see Definition \ref{def of mutation system} for details).
          
          Mutation systems admit a kind of mutation, 
          that is, we have a series of functors between categories of mutation systems 
          with braid relations.
          
          In the case where $\dim V_c=1$ for all $c\in C$, 
          similar structure has been investigated in various contexts 
          \cite{bonsym}, \cite{vanona}, \cite{ggi}, \cite{kuzexc}, \cite{noghel}.
          \subsection{A-mutation systems}

          We can construct a mutation system from the quantum connection of $X$ 
          under an assumption (Assumption \ref{assumption of exponential type}).
          The construction has the following steps. 
          \begin{enumerate}
          \item\label{step 1}
          Apply the Riemann-Hilbert correspondence \cite{mallac}, \cite{sabint} 
          for the quantum connection (around $z=0$), 
          we get a local system on $S^1$ with a filtration called a Stokes filtration.
          We use the assumption here.
          \item\label{step 2}
          A reformulation of Stokes data by Hertling and Sabbah \cite{herexa}
          gives a mutation system (we also use the Poincar\'e pairing on $\coho{X}$).
          Here, we need to fix a real number $\theta_\circ$ with a genericity condition.
          We give this reformulation in \S 2.
          \item\label{step 3}  By using the fundamental solution at $z=\infty$, 
          the vector space $V$ with a pairing $\bil{\cdot}{\cdot}$ underlying the mutation system constructed in (\ref{step 2})
          is identified with $(H^\bullet(X),[\cdot,\cdot)_X)$. 
          \end{enumerate}
          The resulting mutation system is called an A-mutation system.
          Each step is closely related to the construction of asymptotic classes in the semisimple case.
          Although the resulting mutation system depends on the choice of $\theta_\circ$, 
          they are all equivalent by mutations.
          
          \subsection{B-mutation systems}
          We can construct a mutation system for a semiorthogonal decomposition 
          of $D^b(X)$ (see, for example, \cite[Definition 2.3]{kuzhoc} for the definition of semiorthogonal decompositions).
          Put  
          $\mathrm{HH}_k(X):=\bigoplus_{p-q=k}H^q(X,\Omega_X^p)$
          and $\mathrm{HH}_\bullet(X):=\bigoplus_{k\in\Z}\mathrm{HH}_k(X)$. 
          We have a pairing $[\cdot,\cdot\rangle_X$ defined by 
          $$[\alpha,\beta\rangle_X:=\frac{1}{(2\pi\i)^{\dim X}}\int_X e^{\pi\i \rho_{\!_X}}W(\alpha)\cup\beta,$$
          where $W(\alpha)=\i^{p+q}\alpha$ for $\alpha\in H^q(X,\Omega_X^p)$
          (cf. \cite{calmuk1}, \cite{pollef}, \cite{ramrel}, \cite{rammuk}, \cite{shkhir}).          
          Let $D^b(X)=\langle \mathcal{A}_1,\dots,\mathcal{A}_m \rangle$ be a semiorthogonal decomposition.
          Then, by using a theorem of Kuznetsov \cite{kuzhoc}, \cite{kuzbas},  
          we can define the subspace $\mathrm{HH}_\bullet(\mathcal{A}_i)$
          of $\mathrm{HH}_\bullet(X)$,
          which we call the Hochschild homology
          of $\mathcal{A}_i$ (cf. \cite{calmuk1}, \cite{kuzhoc}, \cite{polche}, \cite{shkhir}). 
          They satisfies $\mathrm{HH}_\bullet(X)=\bigoplus_i\mathrm{HH}_\bullet (\mathcal{A}_i)$. 
          Roughly speaking, this decomposition defines 
          a mutation system, which we call a B-mutation system. 
          The braid group action on the set of semiorthogonal decompositions of $D^b(X)$  
          is compatible with the mutation on the B-mutation systems.
          \subsection{The formulation of the analogue and the main theorem}\label{intro main theorem}
          Using the Gamma class, 
          we define an isomorphism $\Gamma:\mathrm{HH}_\bullet(X)\simeqto H^\bullet(X)$
          with $[\alpha,\beta\rangle_X=[\Gamma(\alpha),\Gamma(\beta))_X$.
          We define an analogue of Dubrovin's conjecture 
          by the existence of a semiorthogonal decomposition 
          such that 
          the B-mutation system for the semiorthogonal decomposition
          is isomorphic to 
          the A-mutation system
          via $\Gamma$. 
          We call the analogue ^^ ^^ Dubrovin type conjecture''.
          For more precise, see Definition \ref{dubconj}.
          In the case where the quantum cohomology ring is semisimple, 
          Gamma conjecture II implies this conjecture.
           
          The main result of this paper is to give a class of examples 
          such that this Dubrovin type conjecture holds.
          More precisely, we have the following.
          \begin{theorem}[{Theorem \ref{propertyd}}]
          Let $X$ be a smooth Fano complete intersection in a projective space. 
          If Fano index of $X$ is larger than $1$, then $X$ satisfies Dubrovin type conjecture.
          \end{theorem}
          
          \subsection{Plan of the paper}
          In \S 2, we introduce
          the notion of mutation systems, 
          and define the mutation on mutation systems.
          We also relate it with Stokes filtered local systems, 
          and recall some general facts on the Riemann-Hilbert correspondence 
          for meromorphic connections on a germ of complex plane at the origin.
          In \S 3, we give a definition of A-mutation systems.
          In \S 4,  we give a definition of B-mutation systems.
          In \S 5, we formulate an analogue of Dubrovin's conjecture, which we call ^^ ^^  Dubrovin type conjecture''. 
          In \S 6, we show some properties, which are used to prove the main theorem, 
          of quantum connections of Fano manifolds. 
          In \S 7, we show the main theorem, that is, we give a class of examples 
          which satisfies Dubrovin type conjecture.

          \section{Preliminary}
          \subsection{Stokes filtrations and Stokes data}\label{stofil}
          Recall the definition of Stokes filtrations on local systems over $S^1=\{t\in \C\mid |t|=1\}$
          in the sense of \cite{herexa} (see also \cite{sabint}).
          Let   $\bm{k}$ be a field. 
          Let ${\sf Loc}_{\bm{k}}:={\sf Loc}_{\bm{k}}(S^1)$ denote the category of 
          (finite rank) $\bm{k}$-local systems on $S^1$. 
          We identify $S^1$ with $\mathbb{R}/2\pi\mathbb{Z}$ in the standard way and denote them simply by $S^1=\mathbb{R}/2\pi\mathbb{Z}$.
          For $\theta\in S^1=\mathbb{R}/2\pi\mathbb{Z}$, 
                 let $\leqslant_\theta$ be the partial order on $\mathbb{C}$ defined by the following relation$:$
                 \begin{align}\label{order-c}
                  c\leqslant_\theta c'
                  \Leftrightarrow 
                  \Re (e^{-\mathrm{i}\theta}c) < \Re (e^{-\mathrm{i}\theta}c')\ \ 
                  \text{or}\ \ 
                   c=c'.
                 \end{align}
                 We also set $c<_\theta c'$ if and only if $c\leqslant_\theta c'$ and $c\neq c'$.
                 \begin{definition}[{\cite[2.a]{herexa}}]\label{def stokes}
                  Let $\L$ be a ${\bm k}$-local system on $S^1$.
                  A family $\L_\bullet$ of subsheaves $\L_{\leqslant c}\subset \L\ ({c\in \C})$ 
                  is called a Stokes filtration 
                  if the following conditions are satisfied$:$                           
           \begin{enumerate}
           \item\label{first condition of stokes structure}
                 For each $\theta$, 
                 the germs form an exhaustive increasing filtration of $\L_\theta$
                 with respect to the order defined by the equation $(\ref{order-c})$ above.
           \item
                 Set $\L_{< c,\theta}:=\sum_{c'<c}\L_{\leqslant c', \theta}$.
                 It defines a subsheaf $\L_{<c}$ in $\L$.
                 The second condition is  that the sheaves $\gr_c \L:=\L_{\leqslant c}/\L_{<c}$ 
                 are in ${\sf Loc}_{\bm k}(S^1)$.
           \item Set $\gr\L:=\bigoplus_c \gr_c\L$. 
                 It has a natural filtration 
                 given by 
                 $(\gr\L_{\leqslant c})_{\theta}:=\bigoplus_{c'\leqslant_\theta c}\gr_{c'}\L_\theta$.
                 The last condition is that near any point $\theta\in S^1$, 
                 there are local isomorphisms  $\eta:\L \simeqto \gr \L$ such that
                 $\eta (\L_{\leqslant c})\subset \gr\L_{\leqslant c}$ for all $c\in\C$, 
                 and that the induced graded morphism is identity.          
                 \end{enumerate}
              We call the pair $(\L,\L_\bullet)$ a Stokes filtered local system.
              Let $(\L^i,\L^i_\bullet)$ be two Stokes filtered local systems $(i=0,1)$. 
              A morphism of local systems $\lambda:\L^0\to\L^1$ is called a morphism of Stokes filtered local systems
              if it preserves Stokes filtrations,
              i.e. $\lambda{\L^0_{\leqslant c}} \subset \L^1_{\leqslant c}$ for all $c\in\C$. 
             \end{definition}
           For example, we define a Stokes filtration on ${\bm k}_{S^1}$
           by ${\bm k}_{S^1<0,\theta}=0$ and ${\bm k}_{S^1\leqslant 0,\theta}={\bm k}_{S^1}$
           for all $\theta \in S^1$. 
           For two Stokes structures $(\L^i,\L^i_\bullet)$ $(i=0,1)$, 
           the tensor product $\L^0\otimes \L^1$ is equipped with a natural Stokes filtration 
           by 
           \begin{align}\label{tensor stokes}
              (\L^0\otimes \L^1)_{\leqslant c,\theta}
              :=\sum_{c^0+c^1\leqslant_\theta c} \L^0_{\leqslant c^0,\theta}\otimes \L^1_{\leqslant c^1,\theta}.
           \end{align}
           
           Let $\iota$ be a involution on $S^1=\mathbb{R}/2\pi\mathbb{Z}$ defined by $\iota(\theta)=\theta+\pi$.
           If $(\L,\L_\bullet)$ is a Stokes structure, $\iota^{-1}\L$ has a natural Stokes filtration 
           by 
           \begin{align}\label{involution}
           (\iota^{-1}\L)_{\leqslant c,\theta}:=\L_{\leqslant -c,\theta+\pi} \ (c\in \C,\ \theta \in S^1).
           \end{align}
           The graded quotient $\gr_c(\iota^{-1}\L)$ is naturally identified with $\iota^{-1}\gr_{-c}\L$
           (\cite[\S 3.a]{herexa}).
           The dual local system $\L^\vee=\HOM(\L,{\bm k}_{S^1})$ also has a natural filtration defined by 
           \begin{align}\label{vee}
            (\L^\vee)_{\leqslant c}:=(\L_{<-c})^\perp
           \end{align}
           where $(\L_{<c})^\perp$ consists of local morphisms $\L\to {\bm k}_{S^1}$
           which send $\L_{<c}$ to $0$. 
           The graded quotient $\gr_c(\L^\vee)$ is naturally identified with 
           $(\gr_{-c}\L)^\vee$. 
           Hence $\D\L=\iota^{-1}\L^\vee$ has natural Stokes filtration by 
           \begin{align}\label{duality for Stokes filtration}
            (\D\L)_{\leqslant c}=(\L_{<c,\theta+\pi})^\perp\ (c\in\C, \theta\in S^1),
           \end{align}
           whose graded quotient $\gr_c(\D\L)$ is naturally identified with $\D(\gr_c\L)$.
                        
          The category of Stokes filtered local systems on $S^1$ defined above is denoted by ${\sf St}_{\bm{k}}(S^1)$.
          In this paper we only consider Stokes filtered local systems on $S^1$. 
          We abbreviate $S^1$ and simply denote it by ${\sf St}_{\bm{k}}$.
          For each $(\L,\L_\bullet)\in \st_{\bm{k}}$, the set $\{c\in \C\mid  \gr_c \L\neq  0\}$ 
          is called exponents of $(\L,\L_\bullet)$.
          For a finite subset $C$ in $\C$, 
          ${\sf St}_{\bm{k}}^C$ denote the full subcategory of Stokes filtered local systems 
          whose exponents are contained in $C$.
          We remark that $\D$ defines a contravariant functor from ${\sf St}_{\bm{k}}^C$ to itself.
          
          For two distinct points $c\neq c'$ in $\C$, the Stokes direction
          of the pair is the set of points $\theta$
          in $S^1$ such that $\Re (e^{-i\theta}(c-c'))=0$.
          For a finite set $C\subset \C$, an open interval $I\subset \R$ is called $C$-good 
          if the image in $\R/{2\pi\Z}$ of $I$
          contains exactly one Stokes direction for each pair $c\neq c'$ in $C$.     
          We put $I_{\theta_\circ}:=]\theta_\circ-\frac{\pi}{2}-\varepsilon,\theta_\circ+\frac{\pi}{2}+\varepsilon[$
          for fixed $\theta_\circ\in\R$ where $\varepsilon$ is a sufficiently small positive number .  
          If $I_{\theta_\circ}$ is $C$-good for sufficiently small $\varepsilon >0$, $\theta_\circ$ is called $C$-generic.
          It is equivalent to the condition that $\theta_\circ+\pi/2$ is not the Stokes direction of 
          any pair $c\neq c'$ in $C$.
          
          \begin{proposition}[{\cite[Proposition 2.2]{herexa}}]\label{split}
          \
          \begin{enumerate}
           \item On any $C$-good open interval $I\subset \R$, there exists a unique splitting 
                    $\eta_{[I]}:\L\mid_{[I]}\simeqto \gr\L\mid_{[I]}$ compatible with the Stokes filtrations
                    where $[I]$ is the image of $I$ in $\R/2\pi\Z$.
           \item Let $\lambda:(\L,\L_\bullet)\to(\L',\L'_\bullet)$ be a morphism in ${\sf St}_{\bm k}^C$.
                    Then, for any $C$-good open interval $I\subset \R$, the restriction $\lambda|_{[I]}$ 
                    is graded with respect to the splittings in $(1)$.  
                    In other words, the induced morphism $\lambda_{c, c'}:\gr_c\L\to \gr_{c'}\L$ is zero for 
                    any pair $c\neq c'$ in $C$. 
          \end{enumerate}
          \end{proposition}
          
         Let $\rep_\Z({\bm k})$ 
         be the category of finite dimensional representations of the fundamental group $\pi_1(S^1)\simeq \Z$ over ${\bm k}$.
         This category is considered as the category of pairs $(V,T)$ of 
         a finite dimensional $\bm{k}$-vector space $V$
         and an automorphism $T$ on $V$. 
         Let $\is$ be a finite set and $\tau:\is\simeqto \{1,2,\dots, m\}$ be a bijection.
         \begin{definition}\label{stokes data}
          Let $(V,T)$ be an object in $\rep_\Z({\bm k})$.
          Stokes data on $(V,T)$ of type $(\is,\tau)$ are
          a family of objects 
          $(V_\idx,T_\idx)_{\idx\in\is}$ in $\rep_\Z({\bm {k}})$, 
          isomorphisms $f:\bigoplus_{\idx\in \is}V_\idx\to V$, and
          $f^*:V\simeqto \bigoplus_{\idx\in \is}V_\idx$
          of {\em vector spaces} with the following properties$:$
          \begin{enumerate}
           \item Let $f_\idx:V_\idx \to V$ be the composition of $f$ and the natural inclusion 
           $V_\idx\into \bigoplus_{\idxx\in\is}V_{\idxx}$. 
           Let $f^*_\idx:V \to V_\idx$ be the composition of the projection 
          $\bigoplus_{\idxx} V_{\idxx}\onto V_\idx$ and $f^*$. 
          Then, $f^*_{\idxx} \circ f_\idx=0$ if $\tau(\idxx) < \tau(\idx )$,
          and $f^*_\idx \circ f_\idx = \id_{V_\idx}$ 
          for all $\idx \in \is$.
          \item Set $f^!_\idx:=T_\idx^{-1}\circ f^*_\idx \circ T$ for $\idx\in \is$.
          Then $f^!_{\idxx}\circ f_{\idx}=0$ for $\tau(\idx)<\tau(\idxx)$,
          and $f^!_{\idx} \circ f_\idx=\id_{V_\idx}$ for all $\idx\in \is$.
          \end{enumerate}
          Let $\big{(}(V^{(a)},T^{(a)}),(V^{(a)}_\idx,T^{(a)}_\idx)_{\idx\in\is},f^{(a)},f^{*(a)}\big{)}$  $(a=0,1)$
          be two objects in $\rep_\Z({\bm k})$ with Stokes data of type $(\is,\tau)$.
          A morphism $g:(V^{0},T^{(0)})\to (V^{(1)},T^{(1)})$ is compatible with the Stokes data when 
          the induced maps $(f^{(1)})^{-1}\circ g \circ f^{(0)}$ and $f^{*(1)}\circ g\circ (f^{*(0)})^{-1}$
          are graded.
         \end{definition}
         \begin{remark}\label{2.4}
          We do not need to assume that $f^*$ is an isomorphism
          since it is deduced from the condition $(1)$ and the condition that $f$ is an isomorphism. 
          Similarly, we can show that $f^!:= \prod_cf_c^! :V\to\bigoplus_c V_c$ is an isomorphism.
          \end{remark}
          Let ${\sf Std}_{\bm k}(\is,\tau)$ denote the category of representations in 
          $\rep_\Z({\bm k})$ 
          with Stokes data of type $(\is,\tau)$ defined above
          (a morphism in this category is defined as a morphism in $\rep_\Z({\bm k})$ compatible with the Stokes data).          
          Fix  an $m$-point set $C\subset \C$ and $C$-generic point $\theta_\circ \in \R$.
          Then we have a unique bijection $\tau_{\theta_\circ}:C\simeqto \{1,2,\dots, m\}$
          with the following property:
          when we put $c_i:=\tau^{-1}_{\theta_\circ}(i)$ for $i=1,2,\dots, m$, 
          we have
          $c_1<_{\theta_{\circ}+\pi/2}c_2<_{\theta_{\circ}+\pi/2}\cdots <_{\theta_{\circ}+\pi/2}c_m$.
          
          In the rest of \S \ref{stofil}, we shall construct the functor 
          $\mathfrak{A}_{\theta_\circ}:{\sf St}_{\bm k}^C\to{\sf Std}_{\bm k}(C,\tau_{\theta_\circ})$
          and show that $\mathfrak{A}_{\theta_\circ}$ gives an equivalence of categories.
          
          To prepare for the construction of the functor,
          we recall the equivalence ${\sf Loc}_{\bm k}\simeqto {\sf Rep}_{\bm k}$ of categories.
          Let $p:\R\to S^1:=\R/2\pi\Z$ be the quotient map. This gives an universal covering of $S^1$.
          For $\L\in {\sf Loc}_{\bm k}$, set $V_\L:=\Gamma(\R, p^{-1}\L)$ 
          and define the map $T_\L:V_\L\to V_\L$ by $T_\L v(\theta):=v(\theta+2\pi)$.
          The correspondence $\L\mapsto (V_\L,T_\L)$ is functorial and we can construct the quasi-inverse 
          functor explicitly, this is an equivalence of categories. 
          
           By Proposition \ref{split}, (1), we have a unique splitting
          $\eta_{[I_{\theta_\circ}]}:\L|_{[I_{\theta_\circ}]}\simeqto \gr\L|_{[I_{\theta_\circ}]}$.
          We put 
          $\eta_{\theta_\circ}:=(p^{-1}\eta_{[I_{\theta_\circ}]})|_{I_{\theta_\circ}}:(p^{-1}\L)|_{I_{\theta_\circ}}
          \simeqto (p^{-1}\gr\L)|_{I_{\theta_\circ}}$, 
          which is an isomorphism of sheaves on $I_{\theta_\circ}\subset \R$.

          For an object $(\L,\L_\bullet)\in {\sf St}_{\bm k}^C$,
          we have $(V_\L,T_\L)$ and $(V_{c},T_c):=(V_{\gr_{c}\L},T_{\gr_c\L})$ for $c\in C$.
          Let $r_{\gr\L}^{\theta}:\bigoplus_cV_c\simeqto \Gamma(I_{\theta}, p^{-1}\gr\L)$
          and $r^\theta_\L:V_\L\simeqto\Gamma(I_{\theta},p^{-1}\L)$ 
          be the restriction maps for $\theta\in\R$.
          We define $f_{\theta_\circ}:\bigoplus_{c\in C}V_c\simeqto V$
          as a composition: 
          $f_{\theta_\circ}:=
          (r_\L^{\theta_\circ})^{-1}
          \circ\Gamma(I_{\theta_\circ},\eta_{\theta_\circ}^{-1})
          \circ r^{\theta_\circ}_{\gr\L}$.
          Similarly, we define  $f^*_{\theta_\circ}:V\simeqto \bigoplus_{c\in C}V_c$
          as a composition: 
          $f^*_{\theta_\circ}:=
          (r_{\gr\L}^{\theta_\circ-\pi})^{-1}
          \circ \Gamma(I_{\theta_\circ-\pi},\eta_{\theta_\circ-\pi})
          \circ r^{\theta_\circ-\pi}_\L$.
          Then, 
          $f^!_{\theta_\circ}:=\bigoplus_c T^{-1}_c\circ f^*_{\theta_\circ}\circ T$
          is described as a composition: 
          $f^!_{\theta_\circ}:=
          (r_{\gr\L}^{\theta_\circ+\pi})^{-1}
          \circ \Gamma(I_{\theta_\circ+\pi},\eta_{\theta_\circ+\pi})
          \circ r^{\theta_\circ+\pi}_\L$. 
          \begin{proposition}\label{object}
           The data $\big{(}(V_c,T_c)_{c\in C},f_{\theta_\circ},f^*_{\theta_\circ}\big{)}$
           are Stokes data on $(V_\L,T_\L)$ of type $(C,\tau_{\theta_\circ})$.
           Moreover, the correspondence 
           $(\L,\L_\bullet)\mapsto \big{(}(V_\L,T_\L),(V_c,T_c)_{c\in C},f_{\theta_\circ},f^*_{\theta_\circ}\big{)}$, 
           and 
           $\lambda
           \mapsto \Gamma(\R, p^{-1}\lambda)$
           $($where $\lambda$ is a morphism in ${\sf St}^C_{\bm k}$$)$
           gives a fully faithful functor $\mathfrak{A}_{\theta_\circ}:
           {\sf St}^C_{\bm k}\to {\sf Std}_{\bm k}(C,\tau_{\theta_\circ})$.
          \end{proposition} 
          \begin{proof}
          Let 
          $r^{\theta,\theta'}: V_{\gr \L}\to \Gamma (I_\theta\cap I_{\theta'},\gr\L)$            
          be the restriction map for $\theta,\theta'\in \R$.
          The composition 
          $f^*_{\theta_\circ}\circ f_{\theta_\circ}$
          can be described as the composition:
          $$ (r^{\theta_\circ,\theta_\circ-\pi})^{-1} \circ
          \Gamma(I_{\theta_\circ}\cap I_{\theta_\circ-\pi},\eta_{\theta_\circ-\pi})\circ
          \Gamma(I_{\theta_\circ}\cap I_{\theta_\circ-\pi},\eta_{\theta_\circ}^{-1})\circ
           r^{\theta_\circ,\theta_\circ-\pi}.$$
           The fact that  $\eta_{\theta_\circ-\pi}$ and $\eta_{\theta_\circ}$ are splitting 
          implies  
          $f^*_{\theta_\circ, c}\circ f_{\theta_\circ, c}=\id_{V_c}$.
          Since $\eta_{\theta_\circ-\pi}$ and $\eta_{\theta_\circ}$
          are compatible with filtration, 
          we have $f^*_{\theta_\circ, c'}\circ f_{\theta_\circ, c}=0$
          if $c <_{\theta_\circ-\pi/2} c'$.
          This proves the condition $(1)$ of the Definition \ref{stokes data} 
          because $c <_{\theta_\circ-\pi/2} c'$ is equivalent to 
          $\tau_{\theta_\circ}(c')<\tau_{\theta_\circ}(c)$.          
          Similarly, 
          the composition $f^!_{\theta_\circ}\circ f_{\theta_\circ}$ is described as 
          the following composition: 
          $$ (r^{\theta_\circ,\theta_\circ+\pi})^{-1} \circ
          \Gamma(I_{\theta_\circ}\cap I_{\theta_\circ+\pi},\eta_{\theta_\circ+\pi})\circ
          \Gamma(I_{\theta_\circ}\cap I_{\theta_\circ+\pi},\eta_{\theta_\circ}^{-1})\circ
           r^{\theta_\circ,\theta_\circ+\pi}.$$ 
          Since  $c<_{\theta_\circ+\pi/2} c'$ is equivalent to 
          $\tau_{\theta_\circ}(c)<\tau_{\theta_\circ}(c')$, 
          we have the condition $(2)$ of Definition \ref{stokes data}.
          
          Let $(\L,\L_\bullet)$ and $(\L',\L'_\bullet)$ be objects in ${\sf St}^C_{\bm k}$.
          Let $\lambda:\L\to\L'$ be a morphism of local systems.
          It induces a morphism $g:=\Gamma(\R, p^{-1}\lambda):(V_\L,T_\L)\to (V_{\L'},T_{\L'})$.
          Proposition \ref{split}, (2) implies that $\lambda$ is compatible with Stokes filtration 
          if and only if $g$ is compatible with Stokes data constructed above.
          This shows that the correspondence gives a fully faithful functor
          $\mathfrak{A}_{\theta_\circ}:{\sf St}^C_{\bm k}\to {\sf Std}_{\bm k}(C,\tau_{\theta_\circ})$.          
          \end{proof}        
          We shall prove 
          the essential surjectivity of  $\mathfrak{A}_{\theta_\circ}$. 
          We recall a classification result of Stokes filtered local systems with a fixed graded Stokes structure.
          For a Stokes filtered local system $(\L,\L_\bullet)$, 
          $\mathscr{E}nd(\L)$ is equipped with a natural Stokes filtration.
          We put $\mathcal{A}ut^{<0}(\L):=\id_\L+\mathscr{E}nd(\L)_{<0}\subset \mathscr{E}nd(\L)_{\leqslant0}$.
          A local section of $\mathcal{A}ut^{<0}(\L)$
          is a local endomorphism $\lambda:\L\to \L$ of the local system which is compatible 
          with the Stokes filtration such that the graded morphism $\gr(\lambda)$ is identity.
          For a set of local systems $(\mathscr{G}_c)_{c\in C}$ indexed by $C$, 
          put $\mathscr{G}:=\bigoplus_{c\in C}\mathscr{G}_c$ and 
          define a Stokes filtration on $\mathscr{G}$ by 
          $\mathscr{G}_{\leqslant c,\theta}:=\bigoplus_{c'\leqslant_\theta c}\mathscr{G}_{c'}$.
          \begin{lemma}[{\cite[Proposition 1.42]{sabint}}]\label{graded classification}
           The set of isomorphism classes of Stokes filtered local systems $(\L,\L_\bullet)$  with
           an isomorphism from 
           the graded part $\gr\L$ to $\mathscr{G}$
           is identified with $H^1(S^1;\mathcal{A}ut^{<0}(\mathscr{G}))$. 
          \end{lemma}
          \begin{proof}
          The proof is standard and shown in more general contexts.
          Here, we  only give the construction of the Stokes filtered local system
          from a class $\alpha \in H^1(S^1;\mathcal{A}ut^{<0}(\mathscr{G}))$. 
          Let $(I_k)_{k=1}^\ell$ be a covering on $S^1$ by open intervals 
          such that $I_k\cap I_{k'}=\emptyset$ if $|k-k'|\neq 1, \ell-1$.
          Then $\alpha$ is represented by a \v{C}ech cocycle
          $(\alpha_k)_{k=1}^\ell$ where $\alpha_k\in H^0(I_k\cap I_{k+1};\mathcal{A}ut^{<0}(\mathscr{G}))$ 
          for $k=1,2,\dots,\ell-1$
          and $\alpha_\ell\in H^0(I_\ell\cap I_{1};\mathcal{A}ut^{<0}(\mathscr{G}))$.
          Then, by gluing, there exists a unique (up to isomorphisms) local system $\L$ with
          isomorphisms $\beta_k: \L|_{I_k}\simeqto \mathscr{G}|_{I_k}$ such that
          $\alpha_k=\beta_{k+1}(\beta_k)^{-1}$ for $k=1,2,\dots, \ell-1$ and $\alpha_\ell=\beta_1\circ (\beta_\ell)^{-1}$.
          The Stokes filtration on $\L$ is defined via $\beta_k$ and it is well defined 
          and its graded part $\gr\L$ is isomorphic to $\mathscr{G}$ since $\alpha_k$
          is a local section of $\mathcal{A}ut^{<0}(\mathscr{G})$.
          \end{proof}
          \begin{theorem}
           The functor $\mathfrak{A}_{\theta_\circ}:{\sf St}^C_{\bm k}\to {\sf Std}_{\bm k}(C,\tau_{\theta_\circ})$
           constructed in Proposition $\ref{object}$ is 
           essentially surjective, and hence gives an equivalence of categories.
          \end{theorem}
          \begin{proof}
          Let $\big{(}(V,T),(V_c,T_c)_{c\in C}, f, f^*\big{)}$ be an object in ${\sf Std}_{\bm k}(C,\tau_{\theta_\circ})$.
          Let $\mathscr{G}_c$ be the local system corresponding to $(V_c,T_c)$.
          Set $\mathscr{G}:=\bigoplus_{c\in C}\mathscr{G}_c$. 
          Define a local endomorphism 
          $\lambda^*:\mathscr{G}|_{p(I_{\theta_\circ-\pi}\cap I_{\theta_\circ})}
          \to \mathscr{G}|_{p(I_{\theta_\circ-\pi}\cap I_{\theta_\circ})}$
          by the following composition:
          $$\Gamma(I_{\theta_\circ-\pi}\cap I_{\theta_\circ}, p^{-1}\mathscr{G}) 
          \xleftarrow{\sim}  \bigoplus_c V_c 
          \xrightarrow{f^*\circ f} \bigoplus_c V_c
          \simeqto \Gamma(I_{\theta_\circ-\pi}\cap I_{\theta_\circ}, p^{-1}\mathscr{G})$$
          where the first and the third isomorphisms are the restriction maps 
          regarding $V_c$ as $\Gamma(\R, p^{-1}\mathscr{G}_c)$.
          By Definition \ref{stokes data} (1), $\lambda^*$ is a local section of 
          $\mathcal{A}ut^{<0}(\mathscr{G})$ on $p(I_{\theta_\circ-\pi}\cap I_{\theta})$.
          Similarly, we can define a local endomorphism 
          $\lambda^!:\mathscr{G}|_{p(I_{\theta_\circ-\pi}\cap I_{\theta})}
          \to \mathscr{G}|_{p(I_{\theta_\circ-\pi}\cap I_{\theta})}$
          by using $f^!\circ f$ and it defines a local section of 
          $\mathcal{A}ut^{<0}(\mathscr{G})$ on $p(I_{\theta_\circ-\pi}\cap I_{\theta})$.
          The pair $(\lambda^*,\lambda^!)$ defines a class in 
          $H^1(S^1;\mathcal{A}ut^{<0}(\mathscr{G}))$
          as \v{C}ech cocycle 
          for the covering $S^1=p(I_{\theta_\circ})\cup p(I_{\theta_{\circ}+\pi})$.
          By Lemma \ref{graded classification}, it defines a Stokes filtered local system 
          $(\L,\L_\bullet)$. The fact that 
          $\mathfrak{A}_{\theta_\circ}(\L,\L_\bullet)\simeq \big{(}(V,T),(V_c,T_c)_{c\in C}, f, f^*\big{)}$
          is obvious by the construction of $\mathfrak{A}_{\theta_\circ}$ and the proof of Lemma \ref{graded classification}.
          \end{proof}
          \begin{remark}
          By this theorem, the notion of Stokes data introduced here is equivalent to 
          the notion of Stokes data introduced in \cite{herexa}.
          \end{remark}
          \subsection{Mutations on Stokes data}\label{mutations on std}
              \begin{definition} Let $\big{(}(V,T),(V_\idx,T_\idx)_{\idx\in \is}, f, f^*\big{)}$ be an object in 
              ${\sf Std}_{\bm k}(\is,\tau)$.
               For each $i\in\{1,2,\dots, m\}$, we define the endomorphisms $R_i,R_i^*,L_i,L_i^!$ on $V$
               as follows:
               \begin{align*}
                R_i&:=\id_V-f_{\tau^{-1}(i)}\circ f_{\tau^{-1}(i)}^*, \\
                R_i^*&:=\id_V-T\circ f_{\tau^{-1}(i)}\circ T_{\tau^{-1}(i)}^{-1}\circ f_{\tau^{-1}(i)}^*, \\
                L_i&:=\id_V-f_{\tau^{-1}(i)}\circ f_{\tau^{-1}(i)}^!, \\
                L_i^!&:=\id_V-T^{-1} \circ f_{\tau^{-1}(i)}\circ T_{\tau^{-1}(i)} \circ f_{\tau^{-1}(i)}^!.
               \end{align*} 
              \end{definition}
              We remark that $T^{-1}\circ R_i^*\circ T=L_i$ and $T\circ L_i^!\circ T^{-1}=R_i$.
              We also have $R^*_i\circ R_i=R_i\circ R_i=R_i$, and $L^!_i\circ L_i=L_i\circ L_i =L_i$
              by easy computation.             
              \begin{proposition}
              Let $\big{(}(V,T),(V_\idx,T_\idx)_{\idx\in \is}, f, f^*\big{)}$ be an object in 
              ${\sf Std}_{\bm k}(\is,\tau)$.
              For $i\in \{2,\dots, m\}$, we define 
              \begin{align*}
               (\R_i f)_\idx&:=\begin{cases}
                                R_{i}\circ f_\idx & (\tau(\idx)=i-1) \\
                                f_\idx & \text{otherwise}
                               \end{cases}, \\
                (\R_i f^*)_\idx&:=\begin{cases}
                                        f_\idx^*\circ R^*_i  & (\tau(\idx)=i-1)\\
                                       f_\idx^* & \text{otherwise}.
                                      \end{cases}
              \end{align*}
             Then the tuple $\big{(}(V,T),(V_\idx,T_\idx)_{\idx \in \is}, \R_i f, \R_i f^*\big{)}$  
             is an object of ${\sf Std}_{\bm k}(\is, (i,\ i-1)\circ \tau)$.     
             Similarly, for $i\in\{1,\dots, m-1\}$, we define       
               \begin{align*}
               (\mathbb{L}_i f)_\idx&:=\begin{cases}
                                L_{i}\circ f_\idx & (\tau (\idx)=i+1) \\
                                f_\idx & \text{otherwise}
                               \end{cases}, \\
                (\mathbb{L}_i f^*)_\idx&:=\begin{cases}
                                        f_\idx^*\circ T\circ L_i^{!}\circ T^{-1}  & (\tau(\idx)=i+1)\\
                                       f_\idx^* & \text{otherwise}.
                                      \end{cases}
              \end{align*}              
              Then the tuple $\big{(}(V,T),(V_\idx,T_\idx)_{\idx \in \is}, \mathbb{L}_i f, \mathbb{L}_i f^*\big{)}$  
              is an object of $\std_{\bm k}(\is, (i,\ i+1)\circ \tau)$.                
              \end{proposition}
              \begin{proof}
              We show the first half of the proposition. The second half is shown similarly. 
              Let $\idx, \idxx$ be  two distinct elements of $C$.
              We only consider the case where $\tau(\idx)=i-1$ and $\tau(\idxx)=i$
              since the discussion is easier in the other cases.
              
              We first show that $(\R_i f)_c=R_i\circ f_c$ is injective. 
              Let $u$ be any vector in $V_c$ and assume $R_if_cu=0$.
              It implies $f_cu=f_{c'}\circ f^*_{c'}\circ f_c u$.
              Next, we show that the intersection of the image $(\R_i f)_c(V_c)$ and 
              $f_{c'}V_{c'}$ is $(0)$.
              Let $v$ be a vector in $V_c$ such that $\R_i f v\in V_{c'}$. 
              It implies that $f_cv-f_{\idxx} f_{\idxx}^*f v \in V_{c'}$. 
              Hence $f_c v\in V_\idxx$, and it implies $v=0$.
              The first and the second assertions show that $\R_i f$ defines an isomorphism
              of vector spaces.      
              
              Finally, we show that $(\R_i f, \R_i f^*)$ satisfies the conditions (1) and (2) in Definition \ref{stokes data}.
              We only need to consider the case where $\tau(\idx)=i-1$ and $\tau(\idxx)=i$.
              We have
              \begin{align}\label{A1}
              (\R_i f^*)_{c}\circ (\R_i f)_{c}&=(f^*_c-f^*_cT f_{c'} T^{-1}_{c'}f^*_{c'})(f_c-f_{c'}f_{c'}^*f_c)\\ 
               \label{A2}                       &=f^*_cf_c-f^*_cTf_{c'}T^{-1}_{c'}f^*_{c'}f_c
                                                     +f^*_cTf_{c'}T^{-1}_{c'}f^*_{c'}f_c\\
               \label{A3}                       &=\id_{V_c}. 
              \end{align}
              Here,  we used $f_c^* f_{c'}=0$ from (\ref{A1}) to (\ref{A2})
              and $f_c^* f_c=\id\ (f_{c'}^*f_{c'}=\id)$ from (\ref{A2}) to (\ref{A3}).
              We also have 
              \begin{align*}              
              (\R_i f^*)_{c'}\circ (\R_i f)_{c}&=f_{c'}^*\circ(f_c-f_{c'} f^*_{c'} f_c)\\
                                                    &=0.
              \end{align*}
              by using $f_{c'}^* f_{c'}=\id$. They show the condition (1).
              The condition (2) is shown similarly, which proves the proposition (see Remark \ref{2.4}).
              \end{proof}
              We put 
              $\sigma_i:\std_{\bm k}(C,\tau)\to \std_{\bm k}(C,(i\ i+1)\circ \tau)$ 
              to be the functor 
              $\big{(}(V,T),(V_\idx,T_\idx)_{\idx\in \is}, f, f^*\big{)}\mapsto 
              \big{(}(V,T),(V_\idx,T_\idx)_{\idx\in \is}, \R_{i+1}f, \R_{i+1}f^*\big{)}$
              and put 
              $\sigma_i^{-1}:\std_{\bm k}(C,\tau)\to \std_{\bm k}(C,(i\ i+1)\circ \tau)$
              to be the functor 
              $\big{(}(V,T),(V_\idx,T_\idx)_{\idx\in \is}, f, f^*\big{)}\mapsto
               \big{(}(V,T),(V_\idx,T_\idx)_{\idx\in \is}, \mathbb{L}_i f,\mathbb{L}_i f^*\big{)}$.
              These functors act trivially on the morphisms. 
              It can easily be checked that $\sigma_i, \sigma_i^{-1}$ actually define functors.
              \begin{proposition}\label{braid relation}
              The functors $\{\sigma_i,\sigma_i^{-1}\}_{i=1}^{m-1}$
              satisfy the following braid relations.
               \begin{enumerate}
              \item \label{braid-1}
              $\sigma_i\circ\sigma_i^{-1}=\mathrm{id},$ $\sigma_i\circ\sigma_i^{-1}=\mathrm{id}$,
                            \item \label{braid-2}
              $\sigma_i\circ\sigma_{i+1}\circ\sigma_i=\sigma_{i+1}\circ\sigma_i\circ\sigma_{i+1} $, and
              \item \label{braid-3}
              $\sigma_i\circ\sigma_j=\sigma_j\circ\sigma_i\ \ \  (\ |i-j|\geq 2\ )$.
             \end{enumerate}
             In particular, by the relation $(\ref{braid-1})$, $\sigma_i$ and $\sigma_i^{-1}$ are equivalences of categories.
              \end{proposition}
              \begin{proof}
              The component  
              $(\mathbb{L}_{i}\mathbb{R}_{i+1}f)_c$ is the composition
              $L_{i+1}R_{i+1}f_c$ if $\tau(c)=i$,
              otherwise $f_c$. Easy calculation shows that $L_{i+1} R_{i+1} f_c=f_c$. 
              Similarly, we have $\mathbb{L}_{i}\mathbb{R}_{i+1}f^*=f^*$.
              This implies $\sigma_i^{-1}\circ\sigma_i=\id$. The relation $\sigma_i\circ\sigma_i^{-1}=\id$
              is shown similarly. Hence we get $(\ref{braid-1})$.
              
              We shall prove (\ref{braid-2}). 
              Let $\big{(}(V,T),(V_\idx,T_\idx)_{\idx\in \is}, f, f^*\big{)}$ be an object of $\std_{\bm k}(C,\tau)$.
              Then we have the following: 
              \begin{align*}
               (\R_{i+1} \R_{i+2} \R_{i+1} f)_c=
                 \begin{cases}
                 R_{i+2} R_{i+1} f_c & \tau(c)=i      \\   
                 R_{i+1} f_c & \tau(c)=i+1 \\     
                  f_c& \text{otherwise},              
                 \end{cases}\\
                (\R_{i+2} \R_{i+1} \R_{i+2} f)_c=
                 \begin{cases}
                 R'_{i+2} R_{i+2} f_c & \tau(c)=i      \\   
                 R_{i+1} f_c & \tau(c)=i+1 \\     
                  f_c& \text{otherwise},
                 \end{cases}                                                
              \end{align*}
              where $R'_{i+2}:=\id_V- R_{i+2}f_{\tau^{-1}(i+1)}f_{\tau^{-1}(i+1)}^*R_{i+2}^*$.
              Put $R'_{i+1}:=R_{i+2}$. 
              Then we have $R^{\prime}_{i+2}\circ R^{\prime}_{i+1}=R_{i+2}\circ R_{i+1}$, 
              which implies $\R_{i+1} \R_{i+2} \R_{i+1} f=\R_{i+2} \R_{i+1} \R_{i+2} f$.
              Indeed, we have 
              \begin{align*}
              R^{\prime}_{i+2}\circ R^{\prime}_{i+1}
              &=(\id-R_{i+2}\circ(\id-R_{i+1})\circ R^*_{i+2})\circ R_{i+2}\\
              &=R_{i+2}-R_{i+2}\circ(\id-R_{i+1})\circ R_{i+2}\\
              &=R_{i+2}\circ R_{i+1}\circ R_{i+2},             
              \end{align*}
              and 
              \begin{align*}
              R_{i+2}\circ R_{i+1}\circ(\id-R_{i+2})
              &=R_{i+2}\circ(\id-f_{\tau^{-1}(i+1)}\circ f^*_{\tau^{-1}(i+1)})\circ (f_{\tau^{-1}(i+2)}\circ      
              f^*_{\tau^{-1}(i+2)})\\
              &=R_{i+2}\circ(f_{\tau^{-1}(i+2)}\circ f^*_{\tau^{-1}(i+2)})\\
              &=R_{i+2}\circ(\id-R_{i+2})=0.
              \end{align*}
               The relation $\R_{i+2}\R_{i+1}\R_{i+2}f^*=\R_{i+1}\R_{i+2}\R_{i+1}f^*$ is shown similarly.  
               The  relation $(\ref{braid-3})$ is obvious by the definition.        
              \end{proof}
              Let $\mathrm{Br}_m$ denote the braid group of $m$-strands
              with the standard generator $\sigma_1,\dots,\sigma_{m-1}$.
              For an element $\sigma\in\mathrm{Br}_m$,  
              let $\bar{\sigma}$ denote the image of $\sigma$
              via the quotient map $\mathrm{Br}_m\to S_m$ to the symmetric group of degree $m$. 
              In particular,  $s_i=\overline{\sigma_i}$ are the permutations $s_i=(i\ i+1)$, 
               $(i=1,2,\dots, m-1)$.            
              \begin{definition}\label{braid action}
               For $\sigma\in \mathrm{Br}_m$, define a functor 
               $\mathbb{M}_\sigma:{\sf Std}_{\bm k}(\is,\tau)\simeqto 
               {\sf Std}_{\bm k}(\is,\bar{\sigma}\circ\tau)$
               as the composition of mutations defined above.       
              \end{definition}
%
%
%
%
%
%
%
%
%
%
%
%
%
          \subsection{Stokes factors and Mutations}
          Fix a finite subset $C\subset \C$ and a bijection $\tau:C\simeqto\{1,\dots, m\}$.
          Let  $\big{(}(V,T),(V_\idx,T_\idx)_{\idx\in \is}, f, f^*\big{)}$ be an object of  $\std(C,\tau)$.
          \begin{definition}
          We define a map $(\cdot)_R:S_m\to\mathrm{Br}_m$ by $(s)_R:=\sigma_{i_1}\cdots\sigma_{i_k}$
          for a reduced expression $s=s_{i_1}\cdots s_{i_k}$. This map is well defined
          $($see, for example, \cite[Chapter 4, \S1.5, Proposition 5]{boulie2}, \cite[Theorem 2]{matgen}$)$.
          \end{definition}
          For $s\in S_m$, set $I_i(s):=\left\{j\in\{1,\dots, m\}\mid i<j, s(i)>s(j)\right\}$,
          and $I(s):=\{(i, j)\mid j\in I_i(s)\}$.
          \begin{notation}
          Let $A=\{a_1,a_2,\dots, a_k\}$ be a finite ordered set with $a_1<a_2<\cdots<a_k$.
          Let $(h_a)_{a\in A}$ be a sequence of endomorphisms of $V$
          indexed by $A$. 
          Then, we use the following notation$:$ 
          \begin{align*}
           \prod_{a\in A}^{\leftarrow}h_a:=h_{a_k}\circ h_{a_{k-1}}\circ\cdots\circ h_{a_1}, 
           & \
           \prod_{a\in A}^{\rightarrow}h_a:=h_{a_1}\circ h_{a_{2}}\circ\cdots\circ h_{a_k}.           
          \end{align*}
          \end{notation}
          Recall that for $\sigma\in \mathrm{Br}_m$, 
          we have $\mathbb{M}_\sigma:{\sf Std}_{\bm k}(\is,\tau)\simeqto 
               {\sf Std}_{\bm k}(\is,\bar{\sigma}\circ\tau)$ (Definition \ref{braid action}).
          Define $\sigma f$ and $\sigma f^*$  
          by $\mathbb{M}_\sigma\big{(}(V,T),(V_\idx,T_\idx)_{\idx\in \is}, f, f^*\big{)}=
           \big{(}(V,T),(V_\idx,T_\idx)_{\idx\in \is}, \sigma f, \sigma f^*\big{)}$.
           \begin{lemma}\label{braidact}
           For $s\in S_m$, we have the following formulas: 
           \begin{align}\label{fc}
           \big{(}(s)_R{f}\big{)}_c&=\left( \prod_{i\in I_{\tau(c)}(s)}^{\leftarrow}R_i\right)\circ f_c,\\
           \label{fcc}
           \big{(}(s)_R{f^*}\big{)}_c&=f^*_c\circ\left( \prod_{i\in I_{\tau(c)}(s)}^{\rightarrow}R_i^* \right).
           \end{align}
           \end{lemma}
           \begin{proof}
           For $s\in S_m$, 
           let $\ell(s)$
           be the length of $s$. 
           We shall show the lemma by induction on $\ell(s)$.
           If $\ell(s)=0$, $s$ is identity and the lemma is obvious.
           If $\ell(s)>0$, then there exists some $s_j$ such that $s=s'\cdot s_j$ with $\ell(s)=\ell(s')+1$.
           Set 
           \begin{align*} 
           &f'_c:=\begin{cases}
                      f_c & (\tau(c)\neq j)\\
                      R_{j+1}\circ f_c& (\tau(c)=j),
                   \end{cases} \\
           &R_i':=\begin{cases}
                    R_i& (i\neq j, j+1)\\
                    R_{j+1}& (i=j)\\
                    \id-R_{j+1}\circ f_{c_j}\circ f_{c_j}^*\circ R_{j+1}^* &(i=j+1)
                    \end{cases}           
           \end{align*} where $c_j=\tau^{-1}(j)$.
           By induction hypothesis,
           we have 
           \begin{align*}
           \big{(}(s)_R{f}\big{)}_c=\left(\prod_{i\in I_{s_j\circ\tau(c)}(s')}^{\leftarrow}R'_i\right)\circ f'_c.
           \end{align*}
           Since $\ell(s)=\ell(s')+1$, we have $j+1\notin I_j(s')$,
           and 
           \begin{align*}
            I_k(s)=\begin{cases}
                       s_j(I_{s_j(k)}(s')) & k\neq j\\
                       s_j(I_{s_j(k)}(s'))\cup\{j+1\} & k=j.
                       \end{cases}
           \end{align*}
           If $j+1\notin I_{s_j\circ \tau(c)}(s')$, then 
           \begin{align*}
           \prod_{i\in I_{s_j\circ \tau(c)}(s')}^{\leftarrow}R_i'&=
           \prod^{\leftarrow}_{i\in s_j(I_{s_j\circ\tau(c)}(s'))}R_i,          
           \end{align*}
           hence we have (\ref{fc}).
           If $j+1\in I_{s_j\circ\tau(c)}(s')$, then $j\in I_{s_j\circ\tau(c)}(s')$ since $s'(j)<s'(j+1)$ by
           $j+1\notin I_j(s')$.
           By the proof of Proposition \ref{braid relation}, we have $R'_{j+1}\circ R'_j=R_{j+1}\circ R_j$.
           This implies (\ref{fc}). 
           The equation (\ref{fcc}) is shown similarly.
           \end{proof}
           As a special case,  we have the following.
           \begin{proposition}\label{2.16}
           Set $\Delta:=(w_0)_R\in \mathrm{Br}_m$ where $w_0\in S_m$ is the longest element,
           i.e., 
           $$w_0=
                        \begin{pmatrix}
                        1 & 2&\cdots & m\\
                        m& m-1&\cdots& 1
                        \end{pmatrix}.
           $$ 
           Then we have $\Delta f =(f^*)^{-1}$, and $\Delta f^*=(\bigoplus_cT_c)\circ f^{-1}\circ T^{-1}$.
           \end{proposition}
           \begin{proof}
            Set $\Delta f^!:=(\bigoplus_cT_c)^{-1}\circ \Delta f^* \circ T$.
            We show $\Delta f \circ f^*\circ f=f$, and $f^!\circ f\circ \Delta f^!= f^!$.
            We use the notation $c_i=\tau^{-1}(i)\ (i=1,\dots , m)$.
            Then the restriction $\Delta f\circ f^* \circ f\mid_{V_{c_i}}$ is 
            \begin{align*}
              &\sum_{i\leq j} R_m\circ \cdots \circ R_{j+1}\circ f_{c_j} \circ f^*_{c_j}\circ f_{c_i}\\
              &=R_m\circ \cdots \circ R_{i+1} \circ f_{c_i}+
                \sum_{i< j} R_m\circ \cdots \circ R_{j+1}\circ f_{c_j} \circ f^*_{c_j}\circ f_{c_i}\\
              &=R_m\circ \cdots\circ  R_{i+1} \circ f_{c_i}+\sum_{i<j\leq m}R_m\circ \cdots R_{j+1}\circ f_{c_i}
                  -\sum_{i<j\leq m}R_m\circ \cdots \circ R_j\circ f_{c_i}\\
              &=f_{c_i}.
            \end{align*}
            We also see that the composition $f^!_{c_i}\circ f\circ \Delta f^!$ is 
            \begin{align*}             
              &\sum_{i\leq j}f^!_{c_i}\circ f_{c_j}\circ f^!_{c_j}\circ L_{j+1}\circ\cdots L_m \\
              &=f^!_{c_i}\circ L_{i+1}\circ \cdots \circ L_m+\sum_{i<j\leq m}f^!_{c_i}\circ L_{j+1}\circ \cdots \circ L_m 
                  -\sum_{i<j\leq m}f^!_{c_i}\circ L_{j}\circ \cdots \circ L_m\\
              &=f^!_{c_i}.
            \end{align*}
            This implies the proposition.
          \end{proof}
           For an endomorphism $g\in \End(\bigoplus_c V_c)$ and $(c, c')\in C\times C$, 
          let $g_{c c'}$ denote the component $V_c\to V_{c'}$.
          Using this notation, we define 
           \begin{align}
           &\End^\diamond\left(\bigoplus_{c\in C}V_c\right):=\left\{g\in \End\left(\bigoplus_{c\in C}V_c\right)
                             \middle| g_{cc}=\id_{V_c}\text{ for all } c\in C\right\}, \text{ and }
                             \label{endo 1}
                             \\
           &\mathrm{Sf}_{\tau}(s):=\left\{g\in \End^\diamond\left(\bigoplus_{c\in C}V_c\right)
                             \middle| g_{cc'} =0 
                             \text{ for }(\tau(c),\tau(c'))\notin I(s)\text{ with }\ c\neq c'\right\}, 
           \end{align}
           where $s$ is a element of $S_m$, and $\tau $ is an isomorphism $\tau: C\simeqto \{1,\dots, m\}$.
           \begin{remark}
           If $(i, j)\in I(s)$, then $i<j$.
           If $(i, j), (j, k)\in I(s)$, then $(i, k)\in I(s)$.
           These properties imply $\mathrm{Sf}_\tau(s)$ is a group.
           \end{remark}
           \begin{lemma}
            Take $s\in S_m$ and put $\sigma:=(s)_R$, then $(\sigma f)^{-1}\circ f\in \Sf_\tau(s)$.
           \end{lemma}
           \begin{proof}
           Since $\Sf_\tau(s)$ is a group, it is enough to show $f^{-1}\circ \sigma f\in \Sf_\tau (s)$.
           We have  
           \begin{align*}\sigma f\mid_{V_c}
            &=\left(\prod_{i\in I_{\tau(c)}(s)}^{\leftarrow}R_i\right)\circ f_c \\
            &=\left(\prod^{\leftarrow}_{(\tau(c),\tau(c')\in I(s))}(\id_{V} -f_{c'}\circ f_{c'}^*)\right)\circ f_c.           
            \end{align*}
           This implies $f^{-1}\circ \sigma f\in \Sf_\tau (s)$.
           \end{proof}
%
%
%
%
%
%
%
%
%
%
%
%
           \subsection{Stokes factors and mutations for Stokes filtered local systems}\label{factor and local system}
           Let $C$ be a set of $m$-points in $\C$.
          Let $\theta_\circ$ be a $C$-generic real number. 
          Recall that 
          we have the bijection $\tau_{\theta_\circ}:C\simeqto\{1,2,\dots, m\}$ 
          such that $\tau_{\theta_\circ}(c)<\tau_{\theta_\circ}(c')\Leftrightarrow c<_{{\theta_\circ}+\pi/2}c'$.
          We put $c_i:=\tau_{\theta_\circ}^{-1}(i)$.
          We also have an equivalence of categories 
          $$\mathfrak{A}_{\theta_\circ}:{\sf St}^C_{\bm k}\to {\sf Std}_{\bm k}(C,\tau_{\theta_\circ}). $$
          We also recall that we have a functor 
          $$\mathbb{M}_\sigma:{\sf Std}_{\bm k}(C,\tau_{\theta_\circ})
          \simeqto{\sf Std}_{\bm k}(C,\bar{\sigma}\circ\tau_{\theta_\circ})$$
          for $\sigma\in\mathrm{Br}_m$.
          The purpose of \S\ref{factor and local system} is to show the following: 
          \begin{theorem}\label{mutation gives stokesfactor}
          Let $\theta_\circ,\theta_\circ'$ be $C$-generic real numbers. 
          Then there exists an element $\sigma\in \mathrm{Br}_m$ such that 
          $\mathbb{M}_\sigma\circ\mathfrak{A}_{\theta_\circ}\simeq \mathfrak{A}_{\theta_\circ'}$.
          Moreover, 
          if $\theta_\circ \ge \theta_\circ' \ge \theta_\circ-\pi$, 
          then we can take $\sigma=(\tau_{\theta_\circ'}\circ\tau_{\theta_\circ}^{-1})_R$.
          \end{theorem}
          Take an object $(\L,\L_\bullet)\in{\sf St}^C_{\bm k}$ and 
          put $\big{(}(V,T),(V_c,T_c)_{c\in C},f_\theta, f^*_\theta\big{)}:=\mathfrak{A}_\theta(\L,\L_\bullet)$
          for a $C$-generic real number $\theta\in\R$.
          \begin{definition}           
           For a real number $\theta\in \R$ , we define the following:
           \begin{align*}
           \mathrm{R}(\theta)&:=\{(c, c')\in C\times C\mid \Im (e^{-\i\theta}(c-c'))=0, \text{and } \Re (e^{-\i\theta}(c-c'))>0\},\\
           \mathrm{Sf}(\theta)&:=\left\{g\in \End^\diamond\left(\bigoplus_{c\in C}V_c\right) \middle|
                                              g_{cc'}=0 \text{ for }
                                           (c, c')\notin 
                                            \mathrm{R}(\theta)
                                            \text{ with } c\neq c'\right\}, \\
           \mathrm{Sm}(\theta)&:=\left\{ g\in \End^\diamond\left(\bigoplus_{c\in C}V_c\right)\middle|                                              
                                                  g_{cc'}=0 \text{ for } c'<_{\theta+\pi/2} c  \right\},
           \end{align*}
           where $\End^\diamond\left(\bigoplus_cV_c\right)$ is defined in $(\ref{endo 1})$.
           For a fixed $c\in C$, we also define $\mathrm{R}_c(\theta)$ as 
           the subset of elements $c'$ of $C$ such that $(c, c')$ is in $\mathrm{R}(\theta)$.
           \end{definition}
           For $C$-generic $\theta,\theta'\in \R$, set $s_{\theta,\theta'}:=\tau_{\theta'}\circ \tau_\theta^{-1}\in S_m$.
           For $\theta\in \R$, set $s_\theta:=s_{\theta+\varepsilon,\theta-\varepsilon}$ for $0<\varepsilon\ll 1$.
           \begin{lemma}
            If $\theta \ge \theta' \ge \theta-\pi$, then we have 
            \begin{align}
            I_{\tau_\theta(c)}(s_{\theta,\theta'})
            =\tau_\theta\left(\bigsqcup_{\theta>\theta''>\theta'}\mathrm{R}_c(\theta'')\right).
            \end{align}
           \end{lemma}
           \begin{proof}
            The condition 
            $c'\in \tau_\theta^{-1}(I_{\tau_\theta(c)}((s_{\theta,\theta'})))$
            is equivalent to 
            $c<_{\theta+\pi/2}c', c'<_{\theta'+\pi/2}c$.
            Consider the following function 
            $\mathrm{Im}:\varphi\mapsto \Im(e^{-\i\varphi}(c-c')),\ \varphi\in [\theta',\theta].$
            Then $c'\in \tau_\theta^{-1}(I_{\tau_\theta(c)}((s_{\theta,\theta'})))$ is 
            equivalent to 
            $\mathrm{Im}(\theta)<0<\mathrm{Im}(\theta')$.
            Since $\theta \ge \theta' \ge \theta-\pi$, 
            this is equivalent to 
            the existence and uniqueness of $\theta''\in ]\theta',\theta[$
            such that $\mathrm{Im}(\theta'')=0$, and $\mathrm{Im}(\theta''+\varepsilon')<0<\mathrm{Im}(\theta''-\varepsilon')$
            for a sufficiently small positive number $\varepsilon'$.
            This is equivalent to $c'\in \mathrm{R}_c(\theta'')$.
           \end{proof}
           \begin{corollary}
           We have 
           $\mathrm{Sf}(\theta)=\Sf_{\tau_{\theta+\varepsilon}}(s_\theta)$ for $(0<\varepsilon\ll 1)$.
           If $\theta$ is $C$-generic, then we have
           $\Sm(\theta)=\Sf_{\tau_\theta}(w_0)$.
           \end{corollary}
          Fix $C$-generic $\theta_\circ\in \R$.
          Take $\theta_i\ (i=1,\dots ,\ell)$ so that 
          $\theta_\circ>\theta_1>\theta_2>\cdots>\theta_\ell>\theta_\circ-\pi$, 
          $\{\theta_1,\dots,\theta_\ell\}=\{\theta\in \R\mid \theta_\circ-\pi<\theta<\theta_\circ$ and
          $\mathrm{R}(\theta)\neq \emptyset\}.$
          Then we have
          \[\{(c, c')\in C\times C\mid c<_{\theta_\circ+\pi/2}c'\}=\bigsqcup_i\mathrm{R}(\theta_i).\]
          \begin{lemma}[{\cite[Lemma 2]{bjr}}]
          For all $g\in \Sm(\theta_\circ)$, 
          there exists a unique element $(g_i)_i\in \prod_{1\leq i\leq \ell}\Sf(\theta_i)$ 
          such that 
          \begin{align}\label{gdc}
           g=g_\ell\circ g_{\ell-1}\circ\cdots \circ g_1, 
          \end{align}
          i.e., $\Sf(\theta_\ell)\times\cdots\times \Sf(\theta_1)\simeqto\Sm(\theta_\circ)$.
          \end{lemma}
          \begin{proof}
          For a pair $(i, j)$, $1\leq i<j \leq m$, let $k(i, j)$ be the number such that
          $(c_i, c_j)\in \mathrm{R}(\theta_{k(i, j)})$.
          Let $g_{i j}$ be the $(c_i, c_j)$-component of $g_{k(i, j)}$.
          Then the $(c_i, c_j)$-component of the right hand side of (\ref{gdc}) is 
          \begin{align}\label{gsum}
          \sum_{i=i_0<i_1<\cdots<i_a=j}g_{i_{a-1}i_a}\cdots g_{i_0i_1}
          \end{align}
          where $i=i_0<i_1<\cdots<i_a=j$ runs so that 
          $\theta_{k(i_0,i_1)}<\theta_{k(i_1,i_2)}<\cdots<\theta_{k(i_{a-1},i_a)}$.
          Since (\ref{gsum}) is the sum of $g_{i j}$ and products of $g_{i'j'}$ with $j'-i'<j-i$, 
          we can uniquely determine $g_{i j}$ by induction on $(j-i)$ for given $g$.   
          \end{proof}
          Set $\Delta_{\theta_i}:=(s_{\theta_i})_R\in\mathrm{Br}_m$,
          $(f_0, f^*_0):=(f_{\theta_\circ},f^*_{\theta_\circ})$, 
          and $(f_i, f_i^*)=(\Delta_{\theta_i}f_{i-1}, \Delta_{\theta_i}f^*_{i-1})$.
          We remark that $f_i^{-1}\circ f_{i-1}=
          (\Delta_{\theta_i}f_{i-1})^{-1}f_{i-1}
           \in \Sf_{\tau_{\theta_i+\varepsilon}}(s_{\theta_i})=\Sf(\theta_i)$.
           \begin{lemma}\label{braid decomp}
         $ \Delta_{\theta_\ell}\cdot\Delta_{\theta_{\ell-1}}\cdot \cdots \cdot \Delta_{\theta_1}=\Delta$.
          \end{lemma}
          \begin{proof}
          We have
           $s_{\theta_\ell}\cdots s_{\theta_1}
           =\tau_{\theta_\circ-\pi}\circ\tau_{\theta_\circ}^{-1}=w_0$, and 
           $\sum_{i=1}^\ell\ell(s_{\theta_i})=\sum_{i=1}^\ell \# \mathrm{R}(\theta_i)=\ell(w_0)$.
           This implies the lemma.
          \end{proof}
          \begin{lemma}
           Set $\Sf_{\theta_i}:=f^{-1}_{\theta_i-\varepsilon}\circ f_{\theta_i+\varepsilon}$. 
          Then $\Sf_{\theta_i}$ is an element of $\Sf(\theta_i)$.
          \end{lemma}
          \begin{proof}
           Since $f_{\theta_i-\varepsilon}$ and $f_{\theta_i+\varepsilon}$
           preserve the filtration on $I_{\theta_i-\varepsilon}\cap I_{\theta_i+\varepsilon}$,
           $(\Sf_{\theta_i})_{cc'}=0$ if there exists a $\varphi\in I_{\theta_i-\varepsilon}\cap I_{\theta_i+\varepsilon}$
           such that $c<_\varphi c'$.
          Therefore, if $(\Sf_{\theta_i})_{cc'}\neq0$ and $c \neq c'$, then
          $\theta_i-\varepsilon<\arg (c-c') < \theta_i+\varepsilon$ 
           which implies $(c, c') \in \mathrm{R}(\theta_i).$ 
          Since $f_{\theta_i+\varepsilon}$ and $f_{\theta_i-\varepsilon}$ are splitting, 
          $(\Sf_{\theta_i})_{cc}=\id$, which implies the lemma.
          \end{proof}
          \begin{lemma}\label{2.26}
          We have $\Delta f_{\theta_\circ}=f_{\theta_\circ-\pi}$ and $\Delta f^*_{\theta_\circ}=f^*_{\theta_\circ-\pi}$.
          \end{lemma}
          \begin{proof}
           By definition, we have $(f^*_{\theta_\circ})^{-1}=f_{\theta_\circ-\pi}$. 
           Using Proposition \ref{2.16}, we have the following: 
           \begin{align*}
           (\Delta f_{\theta_\circ} , \Delta f_{\theta_\circ}^*)
           &=\left((f_{\theta_\circ}^*)^{-1}, \left(\bigoplus_c T_c\right)\circ f^{-1}_{\theta_\circ}\circ T^{-1}\right)  \\
           &=(f_{\theta_\circ-\pi},f_{\theta_\circ-2\pi}^{-1}) \\
           &=(f_{\theta_\circ-\pi},f^*_{\theta_\circ-\pi}).
           \end{align*} 
          \end{proof}
          \begin{lemma}
           We have $f_i=f_{\theta_i-\varepsilon}$ and $f_i^*=f^*_{\theta_i-\varepsilon}$.
          \end{lemma}
          \begin{proof}For $\sigma \in \mathrm{Br}_m$,  set $\widetilde{\sigma f}:=(\sigma f)^{-1}\circ f$.
          Then we have
           \begin{align*} 
            \Sf_{\theta_\ell}\cdots\Sf_{\theta_1}
             &=f_{\theta_\circ-\pi}^{-1}\circ f_{\theta_\circ} \\
             &=\widetilde{\Delta f_{\theta_\circ}} \\
             &=\widetilde{\Delta_{\theta_\ell}\cdots\Delta_{\theta_1}f_0} \\
             &=\widetilde{\Delta_{\theta_\ell} f_{\ell-1}}\circ \cdots \circ \widetilde{\Delta_{\theta_1}f_0}.
           \end{align*}
           This implies $\Sf_{\theta_i}=\widetilde{\Delta_{\theta_i}f_{i-1}}$. 
           Hence we have $f_i=f_{\theta_i-\varepsilon}$.
           Set $\theta_\circ':=\theta_\circ -\pi$, $\theta_i':=\theta_i-\pi$, 
           $(f_0',f_0'^*):=(f_{\theta_\circ'},f_{\theta_\circ'}^*)$, 
           and $(f_i',f_i'^*):=(\Delta_{\theta_i'}f_{i-1}',\Delta_{\theta_i'}f_{i-1}'^*)$. 
           By the first part of this lemma, we have $f_i'=f_{\theta'_i-\varepsilon}$.
           By Lemma \ref{2.26}, we see that
           \[(\Delta_{\theta_\ell} \cdots \Delta_{\theta_{i+1}}f_i,
           \Delta_{\theta_\ell} \cdots \Delta_{\theta_{i+1}}f_i^*)=(f_{\theta_\circ'}, f_{\theta_\circ'}^*),\]
           which implies $f_i'=\Delta_{\theta_i'}\cdots \Delta_{\theta_1'}\cdot\Delta_{\theta_\ell}\cdots            
           \Delta_{\theta_{i+1}}f_i=\Delta f_i.$
           Combined with Proposition \ref{2.16}, we have $f_{\theta'_i-\varepsilon}=(f_i^*)^{-1},$
           which proves the lemma.            
          \end{proof}
          \begin{corollary}\label{factorcomputation}
          If $\theta, \theta'\in\R$ are $C$-generic and $\theta \ge \theta' \ge \theta-\pi$, 
          then
          $$\hspace{1.0in}
          f_{\theta',c}
          =\left(\prod^{\leftarrow}_{i\in I_{\tau_\theta(c)}(s_{\theta,\theta'})}R_i\right)
          \circ f_{\theta, c},\ \ 
          f^*_{\theta',c}=f^*_{\theta,c}\circ \prod_{i\in I_{\tau_\theta(c)}(s_{\theta,\theta'})}^\rightarrow                    
          R^*_i.
          \hspace{1.2in}$$ 
          \end{corollary}
          \begin{proof}[Proof of Theorem $\ref{mutation gives stokesfactor}$]
          In the case $\theta \ge \theta' \ge \theta-\pi$,  Theorem $\ref{mutation gives stokesfactor}$ 
          is a direct consequence of Corollary \ref{factorcomputation}. 
          The general case can be reduced to this case.
          \end{proof}
          \subsection{Pairings on Stokes filtered local systems and Mutation systems}
                   For $(V,T)\in {\sf Rep}_{\Z}(\bm{k})$, a pairing 
         $[\cdot, \cdot \rangle: V\otimes V\to \bm{k}$ is called compatible with $T$
         if
         \begin{align}\label{compatiblepairing} 
         [v, w \rangle=[Tw, v\rangle \text{ for all } v, w\in V.
         \end{align}
         If the map $[\cdot, \cdot \rangle$ is non-degenerate, 
         then the monodromy $T$ is determined by 
         the compatibility condition.
         Here, non-degenerate means that 
         the induced map $v\mapsto(w\mapsto [v, w\rangle)$
         is an isomorphism. The condition is equivalent to that 
         the map $v\mapsto(w\mapsto [w, v\rangle)$
         is an isomorphism.   
         A pair $(V , [\cdot,\cdot\rangle)$ of a vector space $V$ and 
         a non-degenerate pairing $ [\cdot,\cdot\rangle$ is called polarized vector space.      
         We often assume that the pairing is non-degenerate without a mention.
         For two polarized vector spaces $\pol$ and $\poll$, 
         a linear map $f:V\to V'$ is called a morphism of  polarized vector spaces if
         it is compatible with the pairings:  $[v, w\rangle=[f v, f w\rangle'$.
         \begin{definition}
         Let $\big{(}(V,T),(V_\idx,T_\idx)_{\idx\in\is}, f, f^*\big{)}$ be an object of $\std(C,\tau)$.
         A pairing $[\cdot ,\cdot \rangle$ on $V$ compatible with $T$ is called
         compatible with the Stokes data 
         if the following conditions hold:
         \begin{itemize}
          \item The induced map 
                   $[v, w\rangle_\idx:=[f_\idx v, f_\idx w\rangle\ (v, w\in V_\idx)$ 
                   is non-degenerate on 
                   $V_\idx$  and compatible with $T_\idx$ for all $\idx \in \is$.
          \item For every $\idx \in \is$, the map
                   $f^*_\idx:V\to V_\idx$ is left adjoint to $f_\idx$ in  the sense that 
                   $[v, f_\idx v_\idx \rangle=[f^*_\idx v, v_\idx \rangle_\idx$ for all $ v\in V, v_\idx \in V_\idx$. 
         \end{itemize}         
         \end{definition}
         A representation $(V,T)$ equipped with Stokes data and 
         a compatible pairing is equivalent to the following structure,
         which we call ^^ ^^ mutation systems".
         \begin{definition}\label{def of mutation system}
          A mutation system is a tuple 
          $\big{(}\pol,\polp,\tau, f\big{)}$
          consisting of 
          \begin{enumerate}
           \item[$1.$] a polarized vector space $\pol$,
           \item[$2.$] a family of polarized vector spaces $\polp$ indexed by a finite set $\is$,
           \item[$3.$] a bijection  $\tau: \is\simeqto \{1,2,\dots, m\}$ 
           $($the pair $(C,\tau)$ is called type of the mutation system$)$, and
           \item[$4.$] an isomorphism $f:\bigoplus_{c\in\is}V_c \simeqto V$ of vector spaces
          \end{enumerate}
          such that $(\tau, f)$ gives a semiorthogonal decomposition of $\pol$ 
          with respect to $\polp$ in the sense that
          \begin{enumerate}
               \item[$(a)$] for every $\idx\in\is$, the restriction $f_\idx:=f|_{V_\idx}$ 
                            is a morphism of polarized vector spaces, and
               \item[$(b)$] if $v\in V_\idx$, $w\in V_{\idxx}$ and $\tau(\idx)>\tau(\idxx)$, 
                            then $[f_c v, f_{\idxx} w\rangle=0$.
          \end{enumerate}
          We call the underlying pair $(\tau, f)$ the splitting data of the mutation system. 
          The category of mutation systems  with fixed type $(\is,\tau)$
          $($whose morphisms  are the morphisms of underlying Stokes data 
          compatible with the pairings$)$ is denoted by ${\sf Mut}_{\bm k}(\is,\tau)$.               
         \end{definition} 
         We remark that we can reconstruct the maps $f^*_c$ (resp. $f^!_c$) for $c\in C$ by the condition 
         $[v, f_c w\rangle = [f^*_c v, w\rangle_c$ 
         (resp. $[f_c w, v\rangle=[w, f^!_c v\rangle_c$) for all $v\in V$, $w\in V_c$.
         We also have $[v, R_i w \rangle=[R^*_i v, w\rangle$, and $[L_i v, w\rangle=[v, L^!_i w\rangle$. 
         The functor $\mathbb{M}_\sigma:{\sf Std}_{\bm k}(\is,\tau)\simeqto {\sf Std}_{\bm k}(\is,\bar{\sigma}\circ\tau)$
         defined in Definition \ref{braid action}
         for $\sigma\in \mathrm{Br}_m$  
         can be extended to the functor 
         $\mathbb{M}_\sigma:{\sf Mut}_{\bm k}(C,\tau)\simeqto{\sf Mut}_{\bm k}(C,\bar{\sigma}\circ \tau)$.  
          
          Let $\L$ be a local system on $S^1$.
        A sesquilinear pairing on $\L$ is
        a morphism $h:\iota^{-1}\L\otimes\L\to {\bm{k}}_{S^1}$
        of local systems.
        It induces two morphisms  
        $\ell_h, \D\ell_h:\L\to\D\L$. 
        Here, $\ell_h$ is defined by $\ell_h t(s):=h(s, t)$ where $t\in\L, s\in \iota^{-1}\L$, 
        and $\D\ell_h$ is its dual.
        It is called non-degenerate if $\ell_h$ is an isomorphism. 
        It is called symmetric
        if $\iota^{-1} h\circ\mathrm{ex}= h$
        where 
        $\text{ex}: \iota^{-1}\L\otimes\L\simeqto\L\otimes\iota^{-1}\L$
        is the exchanging operator.
        This is equivalent to the condition $\D\ell_h=\ell_h$.
        
        \begin{lemma}\label{lem pp}
         Let $\L$ be a local system on $S^1$ and $(V,T)=(V_\L,T_\L)$ be the corresponding object in ${\sf Rep}(S^1)$.
         Then there is a natural one-to-one correspondence 
         between the set of non-degenerate symmetric sesquilinear pairings on $\L$
         and the set of pairings on $V$ compatible with $T$.         
        \end{lemma}
        \begin{proof}
        Let $h$ be a non-degenerate symmetric sesquilinear pairing on $\L$.
        Then the pairing $[\cdot,\cdot\rangle_h$ on $V$ is given by the formula
        \begin{align}\label{p p}
         [s, t\rangle_h:=p^{-1} h(\tau_{\L}s, t) \ \ \ \ \ \ \ & (s, t \in V_{\L}=\Gamma(\R, p^{-1}\L)). 
        \end{align}
        Here, $\tau_{\L}: V_\L\simeqto V_{\iota^{-1}\L}$ is 
        given by $\tau_{\L}s(\theta)=s(\theta-\pi)$ for $\theta\in\R$.
        Symmetry of $h$ implies compatibility of $[\cdot,\cdot\rangle_h$ with $T$
        and non-degeneracy of $h$ implies that of $[\cdot,\cdot\rangle_h$.
        \end{proof}
         \begin{definition}
           Let $(\L,\L_{\bullet})$ be a Stokes filtered local system.
           We call a sesquilinear pairing 
           $h:\iota^{-1}\L\otimes \L\to {\bm k}_{S^1}$ 
           compatible with the Stokes filtration $\L_\bullet$ 
           if the induced morphism $\ell_h$ is a morphism of Stokes filtered local systems.                    
           The category of Stokes filtered local systems with exponents $C$ 
           with symmetric non-degenerate sesquilinear pairings compatible with the Stokes filtration 
           is denoted by ${\stp}^C_{\bm k}$. 
           The morphism in this category is the morphism in $\st_{\bm{k}}^C$ such that the pairing is preserved.          
           \end{definition}
          Remark that $h$ is compatible with the Stokes filtration if and only if $\ell_h:\L\to\D\L$
          is a morphism of Stokes filtered local systems. 
          \begin{lemma}
          We have a functorial isomorphism of Stokes data
          \begin{align*}
            \Phi_{\theta_\circ}(\D\L,\D\L_\bullet)\simeq 
            \left((V^\vee,(T^\vee)^{-1}),(V_c^\vee,(T_c^\vee)^{-1})_{c\in C},
            (f^*_{\theta_\circ})^\vee, \left(\bigoplus_{c\in C}T_c^\vee\right)^{-1} f^\vee_{\theta_\circ}T^\vee\right).
           \end{align*}
         Via this isomorphism, 
         $\Phi_{\theta_\circ}(\ell_h)$ is identified with the map $V\simeqto V^\vee; v\mapsto [\bullet, v\rangle_h$.
              \end{lemma}
               \begin{proof}
               The morphism $\tau_{\L}$ (in the proof of Lemma \ref{lem pp}) gives an 
               isomorphism
               $\tau^\vee_\L:V_{\D\L}\simeqto V^\vee_\L$.
               The morphism $\Phi_{\theta_\circ}(\ell_h):V_\L\to V_{\D\L}$ is by definition 
               identified with $V\simeqto V^\vee; v\mapsto [\bullet, v\rangle_h$.
               We also have $\tau^\vee_{\gr_c\L}:V_{\gr_c\D\L}\simeqto V_c^\vee$.
               Via these isomorphisms, the pair $((f^\vee_{\theta_\circ-\pi})^{-1},(f^{* \vee}_{\theta_\circ-\pi})^{-1})$
               underlies the Stokes data $\Phi_{\theta_\circ}(\D\L,\D\L_\bullet).$
               Since $f_{\theta_\circ-\pi}=(f^*_{\theta_\circ})^{-1}$, and 
               $f^*_{\theta_\circ-\pi}=T^{-1}f_{\theta_\circ}^{-1} (\bigoplus_c T_c)$, 
               we have the conclusion.
               \end{proof}
          \begin{lemma} 
           Let $\big{(}(V,T),(V_c,T_c)_{c\in C}, f, f^*\big{)}$ be an object of 
           $\std(C,\tau)$.
           Then the tuple 
           \begin{align}\label{duality for stokes data}           
           \left((V^\vee,(T^\vee)^{-1}),(V_c^\vee,(T_c^\vee)^{-1})_{c\in C},
            (f^*)^\vee, \left(\bigoplus_{c\in C}T_c^\vee\right)^{-1} f^\vee T^\vee \right)
            \end{align} 
            is also an object in $\std(C,\tau)$.
            A pairing $[\cdot,\cdot\rangle$ on $V$ compatible with $T$ is compatible with the Stokes data 
            if and only if the induced map $V\simeqto V^\vee; v\mapsto [\bullet, v\rangle$ 
            gives an isomorphism of Stokes data between $\big{(}(V,T),(V_c,T_c)_{c\in C}, f, f^*\big{)}$
            and $(\ref{duality for stokes data})$.
            \end{lemma}
            \begin{proof}
            Put $g:=(f^*)^\vee$, and $g^*:=\left(\bigoplus_c T_c^\vee\right)^{-1} f^\vee T^\vee$.
            We also put $g^!:=(\bigoplus_c T_c^\vee) g^*(T^\vee)^{-1}$.
            Then we have $g^*g=\big{(}(\bigoplus_cT_c) f^! f (\bigoplus_cT_c)^{-1}\big{)}^\vee,$
            and $g^! g=(f^*f)^\vee$. This implies the first part of the lemma.
            Let $\ell:V\to V^\vee$ be the map $v\mapsto [\bullet, v\rangle$.
            We also define $\gr_c(\ell): V_c\to V_c^\vee$ by $v_c\mapsto [f_c\bullet,f_c v_c\rangle$.
            The induced morphism $\bigoplus_cV_c\to \bigoplus_cV^\vee_c$ is denoted by $\gr(\ell)$.
            The composition 
            $g^{-1}\ell f$ is given by 
            $w\mapsto [(f^*)^{-1}\bullet, fw\rangle$.
            This map is graded if and only if it is equal to $\gr(\ell)$. 
            This condition is equivalent to the compatibility condition of the pairing with the Stokes data.
            By (\ref{compatiblepairing}),  we have $T^\vee \ell=\ell^\vee$.
            Hence the composition 
            $g^*\ell (f^*)^{-1}$ is given by $\big{(}g^{-1}\ell f (\bigoplus_cT_c^{-1})\big{)}^\vee$.
            Therefore, $g^*\ell (f^*)^{-1}$ is graded if and only if $g^{-1}\ell f$ is graded.
            This completes the proof.      
          \end{proof}
          By these lemmas, the equivalence of the categories 
          $\mathfrak{A}_{\theta_\circ}:\st^C_{\bm k}
          \simeqto \std_{\bm k}(C,\tau_{\theta_\circ})$
          gives an equivalence of categories  
          ${\stp}^C_{\bm k}\simeqto {\sf Mut}_{\bm k}(C,\tau_{\theta_\circ})$, 
          which is also denoted by $\mathfrak{A}_{\theta_\circ}$.
          We remark that the pairing $[\cdot,\cdot\rangle_c\ (c\in C)$
          underlying $\mathfrak{A}_{\theta_\circ}(\L,\L_\bullet, h)$
          is canonically identified with $[\cdot,\cdot\rangle_{\gr_c(h)}$.  
          In particular, it does not depend on the choice of $\theta_\circ$.
          \begin{remark}\label{z2grading}
           We also consider a $\Z/2\Z$-graded version of these structures. 
           If $V$ is a $\Z/2\Z$-graded vector space, the compatibility condition of 
           pairing $[\cdot,\cdot\rangle$ with $T$ is defined by 
           $$ [Tv, w\rangle=(-1)^{\deg v}[w, v\rangle,$$
           where $v, w$ are homogeneous elements.  
           The symmetry of pairings on local systems is replaced by the graded-symmetry. 
           For $\Z/2\Z$-graded local system $\L=\L^0\oplus \L^1$, $h$ is graded symmetric 
           if the following equalities hold$:$
           $h(\L^i,\L^j)=0$ for $i\neq j$, and $\iota^{-1} h\circ\mathrm{ex}= (-1)^ih$ on $\iota^{-1}\L^i\otimes\L^i$ 
           for $i=0,1$. 
           The functor  
           $\mathfrak{A}_{\theta_{\circ}}: {\stp}^C_{\bm k}
           \simeqto {\sf Mut}_{\bm k}(C,\tau_{\theta_\circ})$ enhanced to 
           these categories is also denoted by the same notation.       
          \end{remark}
  

\subsection{A generalization of the construction of Stokes data}
We generalize the construction of Stokes data from Stokes filtered local systems.
This construction is only used
in the proof of Lemma \ref{van cyc at center} and Theorem \ref{propertyd}.

Let $(\L, \L_\bullet) \in \st_{\bm{k}}^\is$
be a Stokes filtered local system with exponents $\is \subset \C$.
Fix $\is$-generic $\theta_\circ$.
A tuple of real numbers $\theta_\bullet=\{\theta_\idx\}_{\idx\in\is}$ 
is called $\is$-generic if $\theta_\idx$ are $\is$-generic for all $\idx \in \is$.
Set \[L^-_{\theta_c}:=\left\{\idx-r e^{{\tt i}\theta_\idx}\middle| 0\le r\right\} \subset \C\].

\begin{definition}
Let $\theta_\bullet$ be a $\is$-generic tuple of real numbers
with $\theta_\circ \ge\theta_\idx>\theta_\circ-2\pi$ for all $\idx \in \is$.
We define a binary relation $<_{\theta_\bullet+\pi/2}$ as follows:
\[\idx<_{\theta_\bullet+\pi/2}\!\idxx \Leftrightarrow 
\begin{cases} 
  \theta_\idx>\theta_\idxx \  \text{and} \ \ L^-_{\theta_\idx}\cap L^-_{\theta_\idxx}=
  \emptyset \ \text{or}  \\
  \theta_\idx=\theta_\idxx \  \text{and} \ \ \idx<_{\theta_\idx+\pi/2} \idxx.
\end{cases}
\]
\end{definition}

The next lemma gives another description of $<_{\theta_\bullet+\pi/2}$.

\begin{lemma}\label{ordered lemma}
\[\idx<_{\theta_\bullet+\pi/2}\!\idxx \Leftrightarrow
\begin{cases}
\idx<_{\theta_\idx+\pi/2}\!\idxx \text{or} \ \idx<_{\theta_\idxx+\pi/2}\!\idxx
&\quad (\theta_\idx \ge\theta_\idxx>\theta_\idx-\pi). \\
\idxx\!<_{\theta_\idx+\pi/2}\!\idx\, \text{or} \ \idxx\!<_{\theta_\idxx+\pi/2}\!\idx
&\quad (\theta_\idx-\pi \ge\theta_\idxx).
\end{cases}\]
\end{lemma}
\begin{proof}
Set
\[L_{\theta_c}:=\left\{\idx-re^{{\tt i}\theta_\idx}\middle| r\in \R \right\},\quad
   L^+_{\theta_\idx}:=L_{\theta_\idx}\setminus L^-_{\theta_\idx}.\]
If $\theta_\idx=\theta_\idxx$ or $\theta_\idx=\theta_\idxx \pm \pi,$
then the lemma easily follows.
Hence we consider the following two cases:
  \begin{enumerate}
   \item $\theta_\idxx \in ] \theta_\idx-\pi, \theta_\idx [ \ \pmod{2\pi}.$
   \item $\theta_\idxx \in ] \theta_\idx, \theta_\idx+\pi [ \ \pmod{2\pi}.$
  \end{enumerate}
We first consider the case (1).
In this case, we see that
$\idx<_{\theta_\idx+\pi/2}\!\idxx$
if and only if 
$L_{\theta_\idx}\cap L_{\theta_\idxx} \subset L^+_{\theta_\idxx}$
and
$\idx<_{\theta_\idxx+\pi/2}\!\idxx$
if and only if 
$L_{\theta_\idx}\cap L_{\theta_\idxx} \subset L^+_{\theta_\idx}$.
Therefore $\idx<_{\theta_\idx+\pi/2}\!\idxx$ or $\idx<_{\theta_\idxx+\pi/2}\!\idxx$ if and only if 
$L^-_{\theta_\idx}\cap L^-_{\theta_\idxx}=\emptyset$.
The lemma easily follows from this.
The proof of the case (2) is similar.
\end{proof}

In general, the binary relation $<_{\theta_\bullet+\pi/2}$ is not a partial ordering.

\begin{definition}
A tuple of real numbers $\theta_\bullet$ is said to be ordered if 
$\theta_\bullet$ is $\is$-generic,
$\theta_\circ\ge\theta_\idx>\theta_\circ-2\pi$ for all $\idx \in \is$,
and $<_{\theta_\bullet+\pi/2}$ gives a total ordering.
\end{definition}

For ordered $\theta_\bullet$,
we have the isomorphism of ordered sets $\is \simeqto \{1,2, \dots, m\},$
where the order of $\is$ is given by $<_{\theta_\bullet+\pi/2}$.
This isomorphism is denoted by $\tau_{\theta_\bullet}$. 
Set
\[f_{\theta_\bullet}:=\coprod_{\idx\in\is} f_{\theta_\idx,\idx}, \quad 
   f^*_{\theta_\bullet}:=\prod_{\idx\in\is} f^*_{\theta_\idx,\idx},\]
where the symbol $\coprod$ means the direct sum.

\begin{proposition}\label{orderedstokes}
If $\theta_\bullet$ is ordered, then
$\left((V_\idx,T_\idx)_{\idx\in\is}, f_{\theta_\bullet}, f^*_{\theta_\bullet}\right)$
are Stokes data on $(V, T)=(V_\L, T_\L)$ of type $(\is, \tau_{\theta_\bullet})$.
\end{proposition}
\begin{proof}
It is sufficient to show that $f^*_{\theta_\idx,\idx}\circ f_{\theta_\idxx, \idxx}=0$ and 
$f^!_{\theta_\idxx,\idxx}\circ f_{\theta_\idx, \idx}=0$ for $\idx<_{\theta_\bullet+\pi/2} \idxx$.
We first consider the case $\theta_\idx\ge\theta_\idxx>\theta_\idx-\pi$.
In this case, we see that
\[\theta_\idx-\pi/2, \ \theta_\idxx-\pi/2\in I_{\theta_\idx-\pi}\cap I_{\theta_\idxx}.\]
Combined with Lemma \ref{ordered lemma},
we have $f^*_{\theta_\idx, \idx}\circ f_{\theta_\idxx, \idxx}=0$
(see also the proof of Proposition \ref{object}).
By rotating $\pi$, we also see that 
\[\theta_\idx+\pi/2, \ \theta_\idxx+\pi/2\in I_{\theta_\idx}\cap I_{\theta_\idxx+\pi},\] 
which implies $f^!_{\theta_\idxx,\idxx}\circ f_{\theta_\idx, \idx}=0$.
The case $\theta_\idx-\pi\ge\theta_\idxx$ is similar.
Note that in this case we have
\[\theta_\idx-3\pi/2, \ \theta_\idxx+\pi/2\in I_{\theta_\idx-\pi}\cap I_{\theta_\idxx}. \qedhere \]
\end{proof}

Similar to Proposition \ref{object}, this correspondence gives a fully faithful functor
\[ \mathfrak{A}_{\theta_\bullet} : \st_{\bm{k}}^\is \to \std_{\bm{k}} ( \is, \tau_{\theta_\bullet}). \]
The following is a direct consequence of the above proposition. 

\begin{corollary}\label{orderedmutation} 
If $(\L, \L_\bullet, h) \in \stp_{\bm{k}}^\is$, 
then the corresponding pairing $\bil{\cdot}{\cdot}$ on $(V_\L, T_\L)$ is compatible with the Stokes data
$\mathfrak{A}_{\theta_\bullet}(\L, \L_\bullet).$ 
\end{corollary}

By the above corollary, we have a functor from
$\stp_{\bm{k}}^\is$ to ${\sf Mut}_{\bm{k}}(\is, \tau_{\theta_\bullet}).$
To simplify notation, this functor is also denoted by $\mathfrak{A}_{\theta_\bullet}.$ 

\begin{remark}
We also consider a $\Z/2\Z$-graded version $($see Remark $\ref{z2grading} )$.
This generalized functor is also denoted by $\mathfrak{A}_{\theta_\bullet}.$
\end{remark}
          
        \subsection{A reminder for Riemann-Hilbert correspondence}\label{reminder for RH}
        We recall some fundamental results on Riemann-Hilbert correspondence for meromorphic connections 
        on the germ $(\mathbb{C},0)$ of a complex plane at zero.
        A comprehensive reference is \cite{sabint}.
        Let $\mathscr{M}$ be a finite dimensional 
        $\mathscr{O}_{\C,0}(*\{0\})$-vector space
        together with a $\mathbb{C}$-linear map 
        $\nabla:\mathscr{M}\to\mathscr{M}\otimes \Omega_{\mathbb{C},0}^1$
        satisfying the Leibniz rule: $\nabla (as)=a\nabla s+s\otimes da$, 
        $(a\in \mathscr{O}_{\mathbb{C},0}(*\{0\}),s\in \mathscr{M})$.
        We call such a pair $(\mathscr{M},\nabla)$ a meromorphic connection on $(\mathbb{C},0)$.
        We often abbreviate $\nabla$.
       %
        For $c\in \mathbb{C}$, we define a meromorphic connection 
        $\mathscr{E}^{c/z}:=(\mathscr{O}_{\C,0}(*\{0\}),d+d(c/z))$.
         \begin{definition}\label{meromorphic connection of exponential type}
        A meromorphic connection $\mathscr{M}$ is called of exponential type with exponents 
        $C\subset \mathbb{C}$
        if there is an isomorphism of formal meromorphic connections:
         \begin{align}\label{exp type iso}
           \mathscr{M}\otimes\mathbb{C}((z))\simeq
           \left(\bigoplus_{c\in C}\mathscr{E}^{-c/z}\otimes \mathscr{R}_c\right)\otimes \mathbb{C}((z)),
         \end{align}
         where 
         $\mathscr{R}_c$ is a regular singular meromorphic connection for each $c$.
        The category of meromorphic connections on $(\mathbb{C},0)$ of exponential type
        is denoted by $\me$.
        For a finite subset $C\subset \mathbb{C}$, $\me_C$ denotes the full subcategory 
        such that the
        exponents of objects are contained in $C$.
        \end{definition}
        Let $\varpi:\mathrm{Bl}^{\mathbb{R}}_{0}(\mathbb{C})\to\mathbb{C}$ be the real blowing up of 
        the complex plane $\mathbb{C}$ at zero. We have natural inclusions 
        $j_{\mathbb{C}^*}:\mathbb{C}^*\hookrightarrow\mathrm{Bl}^{\mathbb{R}}_{0}(\mathbb{C})$
        and 
        $i_\partial:S^1\simeq \partial\mathrm{Bl}^{\mathbb{R}}_{0}(\mathbb{C})
         \hookrightarrow \mathrm{Bl}^{\mathbb{R}}_{0}(\mathbb{C})$
         where $\partial\mathrm{Bl}^{\mathbb{R}}_{0}(\mathbb{C})$ is the boundary.
        We put  $\widetilde{\mathscr{O}}:=(j_{\mathbb{C}^*})_*\mathscr{O}_{\mathbb{C}^*}$
        where $\mathscr{O}_{\mathbb{C}^*}$ is the sheaf of holomorphic functions on $\mathbb{C}^*$.
        We define subsheaves $\mathscr{A}^{\rm mod}$ and $\mathscr{A}^{\rm rd}$ of $\widetilde{\mathscr{O}}$ 
        as follows:
        For an open subset $U$ in $\mathrm{Bl}^{\mathbb{R}}_{0}(\mathbb{C})$,
        put $U^*:=U\cap \mathbb{C}^*$.
        A section $f$ of $\widetilde {\mathscr{O}}(U)=\mathscr{O}(U^*)$ is 
        a section of $\mathscr{A}^{\rm mod}(U)$ 
        if and only if for any compact subset $K\subset U$, there exist a constant $C_K$ and a non-negative integer 
        $N_K$ such that $|f(z)|\leq C_K|z|^{-N_K}$ for all $z\in U^*\cap K$.
        A section $g\in \mathscr{O}(U^*)$ is a section of 
        $\mathscr{A}^{\rm rd}(U)$ if and only if for any compact subset $K\subset U$,
        and for any non-negative integer $N$, there is a constant $C_{K,N}$ such that 
        $|g(z)|\leq C_{K,N}|z|^N$ for all $z\in U^*\cap K$.
        For $c\in \C$, we also have the subsheaves $e^{c/z}\mathscr{A}^{\rm mod}$, and $e^{c/z}\mathscr{A}^{\rm rd}$
        where $e^{c/z}$ is considered as a section of $\widetilde {\mathscr{O}}$.
        
        Using these sheaves as coefficients, we define various de Rham complexes as follows:
        \begin{itemize}
        \item  
        $\widetilde{\mathrm{DR}}(\mathscr{M})
          :=\{\widetilde{\mathscr{O}}\otimes \varpi^{-1}\mathscr{M}\to 
              \widetilde{\mathscr{O}}\otimes \varpi^{-1}(\Omega^1\otimes\mathscr{M})\}$,
        \item 
        $\mathrm{DR}_{\leqslant c}(\mathscr{M})
          :=\{e^{c/z}\mathscr{A}^{\rm mod}\otimes \varpi^{-1}\mathscr{M}\to 
              e^{c/z}\mathscr{A}^{\rm mod}\otimes \varpi^{-1}(\Omega^1\otimes\mathscr{M})\}$,
        \item
          $\mathrm{DR}_{<c}(\mathscr{M})
          :=\{e^{c/z}\mathscr{A}^{\rm rd}\otimes \varpi^{-1}\mathscr{M}\to 
              e^{c/z}\mathscr{A}^{\rm rd}\otimes \varpi^{-1}(\Omega^1\otimes\mathscr{M})\}, $      
        \end{itemize}
        where $\widetilde{\mathrm{DR}}(\mathscr{M})$ has cohomology in degree $0$ at most.

        Then the pair 
        $$\mathrm{RH}(\mathscr{M})
        :=(\mathscr{H}^0\widetilde{\mathrm{DR}}(\mathscr{M}),
        \mathscr{H}^0 \mathrm{DR}_{\leqslant \bullet}(\mathscr{M}))$$
        is considered as a Stokes filtered local system on $S^1$ via $i_\partial^{-1}$,
        where $\mathscr{H}^0$ is the cohomology of degree zero.
        \begin{theorem}[\cite{mallac}, {\cite[Theorem 5.7]{sabint}}]
        The Riemann-Hilbert functor 
        $\mathrm{RH}:\me_\is \to {\sf St}_{\bm \C}^C$
        is an equivalence of categories.
        \end{theorem}
        \begin{remark}\label{remark RH taiou}
        Let $(\L,\L_\bullet)$ be the Stokes filtered local system corresponding to 
        $(\mathscr{M},\nabla)\in \me_C$ via $\mathrm{RH}$. 
        Then the rank of $\gr_{c_i}\L$ is equal to the rank of $\mathscr{R}_i$ in 
        $(\ref{exp type iso})$.
        It is also equal to the dimension of $V_{c_i}$ of the corresponding Stokes data.
        \end{remark}
        Let $\iota:(\C,0)\to (\C,0)$ be the involution $z\mapsto -z$.
        Set $\mathscr{M}^\vee:=\HOM\big{(}\mathscr{M},\mathscr{O}(*\{0\})\big{)}$.
        Then we define $\mathbb{D}(\mathscr{M}):=\iota^{-1}\mathscr{M}^\vee$.
        
        \begin{proposition}[{\cite[Proposition 5.15]{sabint},}]
        The Riemann-Hilbert functor is compatible with duality, i.e. 
        $\mathrm{RH}(\mathbb{D}\mathscr{M})\simeq \mathbb{D}\mathrm{RH}(\mathscr{M})$.
        \end{proposition}
        \begin{proof}
         The compatibility with dual $\mathscr{M}\mapsto \mathscr{M}^\vee$ 
         is shown in \cite[Proposition $5.15$]{sabint}.
         The compatibility with the involution $\iota^{-1}$ is shown similarly.
        \end{proof}
        \begin{definition}
         We define the category $\mep$ of meromorphic connections of exponential type with pairings
         as follows$:$
         \begin{enumerate}
          \item[$1.$] An object in $\mep$ is a pair $(\mathscr{M},\mathscr{Q})$ of a meromorphic connection
          $\mathscr{M}\in \me$ and an isomorphism 
           \begin{align}
            \mathscr{Q}:\mathscr{M}\xrightarrow{\sim}\mathbb{D}\mathscr{M}
           \end{align}
           of meromorphic connections such that $\D\mathscr{Q}=\mathscr{Q}$.
          \item[$2.$] Let $(\mathscr{M},\mathscr{Q})$ and $(\mathscr{M}',\mathscr{Q}')$ 
          be objects in $\mep$.
           A morphism from $(\mathscr{M},\mathscr{Q})$ to
           $(\mathscr{M}',\mathscr{Q}')$ 
           is a morphism $\lambda\in \mathrm{Hom}_\me(\mathscr{M},\mathscr{M}')$ such that 
           $\mathbb{D}\lambda\circ{\mathscr{Q}'}\circ\lambda=\mathscr{Q}$.            
         \end{enumerate}
        \end{definition}
        For a finite set $C\subset\mathbb{C}$, we also define $\mep_C$ as a full subcategory
        of $\mep$ whose exponents are contained in $C$.
        \begin{corollary}
         The Riemann-Hilbert functor 
          $\mathrm{RH}:\mep_C\to\stp^C_\C$
          is well defined and gives an equivalence of categories.
        \end{corollary}
        \begin{remark}\label{zisuu}
         As in Remark $\ref{z2grading}$, we consider $\Z/2\Z$-graded meromorphic connections.
         The connections are assumed to be grade-preserving.
         The only difference with non-graded case is the pairing. 
         The pairing is defined to be graded-symmetric so that 
         the equivalence $\mathrm{RH}:\mep_C\to\stp^C_\C$ is generalized to 
         the $\Z/2\Z$-graded case. The generalized functor is also denoted by $\mathrm{RH}$.
        \end{remark}
                 \section{A-mutation  systems}\label{sectionamut}
          \subsection{Quantum connections of exponential type}
           Let $X$ be a Fano manifold, that is, $X$ is a smooth projective variety 
            whose anti-canonical bundle $\omega_X^{-1}:=\det TX$ is ample.
            Let $H^\bullet(X)$ 
            denote the Betti cohomology group of $X$ over $\mathbb{C}$.
            For $\alpha_1,\alpha_2,\dots,\alpha_n \in H^\bullet(X)$, 
            let $\langle \alpha_1,\alpha_2,\dots,\alpha_n \rangle^X_{0, n, d}$ denote 
            the genus-zero $n$-points Gromov-Witten invariant of degree $d\in H_2(X,\Z)$. 
            The quantum cup product $\alpha_1*_\tau \alpha_2$ of two classes 
            $\alpha_1,\alpha_2\in \coho{X}$ with parameter $\tau\in \coho{X}$ is given by 
            \begin{equation}
             (\alpha_1*_\tau \alpha_2,\alpha_3)_X
              =\sum_{d\in \Eff{X}}\sum_{n=0}^\infty
               \langle \alpha_1,\alpha_2,\alpha_3,\tau,\dots,\tau \rangle_{0, 3+n, d}^X
            \end{equation}
            where $(\alpha,\beta)_X:=\int_X\alpha\cup\beta$ is the Poincar\'e pairing and 
            $\Eff{X}\subset H_2(X;\Z)$ is the set of effective curve classes.
            It is not known if the quantum products $*_\tau$ converge in general, 
            however the quantum cup product for $\tau\in H^2(X)$ makes sense
            since $X$ is a Fano manifold.
            \begin{definition}[\cite{dubgeo2}]\label{def of qdm}
             Consider the trivial vector bundle 
             $\mathcal{H}_X:=H^\bullet(X)\otimes \mathcal{O}_{\C}$ over $\C_z$.
             Define the meromorphic flat connection 
             $\nabla:\mathcal{H}_X\to\mathcal{H}_X\otimes\Omega^1_\C(\log \{0\})\otimes\mathcal{O}_\C(\{0\})$
             called a quantum connection by
             \begin{align}\label{quantum connection}
              \nabla:={d}-\Big{(}\frac{1}{z}(c_1(X)*_0)-{\mu}\Big{)}\frac{dz}{z}, 
             \end{align}
             where $z$ denotes the coordinate on $\C$,
             $\mu\in \End(\coho{X})$ is the grading operator defined by 
             $\mu|_{H^{p}(X)}:=({p-\dim_\C X})/{2}\cdot \id_{H^{p}(X)}$,
             and $c_1(X)\in H^2(X;\Z)$ is the first Chern class of $X$.
             We also set the meromorphic connection
             $\mathscr{M}_X:=\mathcal{H}_X\otimes \mathscr{O}_{\C,0}$.
             We define a $\Z/2\Z$-graded symmetric sesquilinear pairing 
               $Q_X:\iota^*\mathcal{H}_X\otimes \mathcal{H}_X\to \mathcal{O}_\C$ 
               by $Q_X(s, t)(z):=(s(-z),t(z))_X$ where $s, t \in \mathcal{H}_X$.            
               The induced pairing on $\mathscr{M}_X$ is denoted by $\mathscr{Q}_X$.
             \end{definition}
            We consider the following:
            \begin{assumption}[{\cite[Conjecture 3.4]{kkp1}}]\label{assumption of exponential type}
             The meromorphic connection $\mathscr{M}_X$ is of exponential type.
            \end{assumption}       
            Under this assumption, 
            we have the following:
            \begin{corollary}[{Corollary of Assumption \ref{assumption of exponential type}}]\label{cor of exp type}
            Let $C_X$ be the set of eigenvalues of $c_1(X)*_0$.
             The set of exponents of $\mathscr{M}_X$ coincide with 
             $-C_X$. 
             In other words, we have an isomorphism 
             $$ \mathscr{M}_X\otimes \C((z))\simeq 
             \left(\bigoplus_{c\in C_X} \mathscr{E}^{c/z}\otimes \mathscr{R}_c\right)
             \otimes \C((z))$$
             where $\mathscr{R}_c$ is regular singular. 
             The rank of the regular singular part $\mathscr{R}_c$ 
             is 
             the dimension of the eigenspace of $c_1(X)*_0$ associated with $c$. 
            \end{corollary}
            \begin{proof}
             For the lattice $\mathcal{H}_X$, we have $(z^2\nabla_{\partial_z})|_{z=0}=-c_1(X)*_0$  
             identifying the fiber of $\mathcal{H}_{X}$ at $z=0$ with $H^\bullet(X)$.
             By Exercise 5.9 in \cite[II]{sabiso}, we have a decomposition  
             $\mathscr{M}_X\otimes \C((z))\simeq \bigoplus_{c\in C_X}\mathscr{M}_c$  
             compatible with the lattice and the generalized eigenvalue decomposition of $c_1(X)*_0$.
             On the other hand, by Assumption \ref{assumption of exponential type}, 
             we have an isomorphism 
             $\mathscr{M}_X\otimes \C((z))\simeqto 
             \left( \bigoplus_{i=1}^{m} \mathscr{E}^{c_i/z}\otimes\mathscr{R}_i\right)
             \otimes\C((z))$
             for some distinct complex numbers $c_1,\dots, c_m$.
             Consider the induced morphism 
             $\phi:\mathscr{M}_c\to  (\mathscr{E}^{c_i/z}\otimes\mathscr{R}_i)\otimes \C((z))$
             for some $c\in C_X$ and $i=1,\dots, m$. 
             We claim that $\phi=0$ if $c\neq c_i$ (This claim implies the corollary).
             Take a frame $\bm{v}=(v_k)_k$ of $\mathscr{M}_c$ so that $v_k\in \mathcal{H}_X\otimes \C[[z]].$
             Then we have 
             $$\nabla_{z\partial_z} \bm{v}=\bm{v}\big{(}-z^{-1}(c\cdot \mathrm{Id}+N)+A(z)\big{)}$$
             where $\mathrm{Id}$ is the identity matrix, 
             $N$ is a nilpotent constant matrix, and $A(z)$ is a matrix with entries in $\C[[z]]$.
             We also take a frame $\bm{w}$ in $(\mathscr{E}^{c_i/z}\otimes\mathscr{R}_i)\otimes\C((z))$ 
             so that $\nabla_{z\partial_z}\bm{w}=\bm{w}(-z^{-1}c_i\cdot \mathrm{Id}+A'(z))$ where 
             $A'(z)$ is a matrix with entries in $\C[[z]]$.
             If we take a matrix $B=B(z)$ with entries in $\C((z))$ so that $\phi(\bm{v})=\bm{w}B$, 
             the flatness condition of $\phi$ is written as follows:
             \begin{equation}\label{order compare}
              z\partial_z B+(A'B-BA)=z^{-1}B\big{(}(c_i-c)\cdot \mathrm{Id}+N\big{)}.
             \end{equation} 
             Since we assume that $c\neq c'$ and $N$ is nilpotent, $((c_i-c)\cdot \mathrm{Id}+N)$ 
             is an invertible constant matrix. 
             By comparing orders of entries between both sides of (\ref{order compare}),  
             we conclude that $B=0$.
            \end{proof}
%
%
%
%
%
%
%
%
            \subsection{Fundamental solutions, pairings, and A-mutation systems}
            Although the following proposition is proved for even degrees of the cohomology, 
            the same proof can be applied for the $\Z/2\Z$-graded case.
              \begin{proposition}[{\cite[Lemma 2.4, and Lemma 2.5]{dubpai}, see also \cite[Proposition 2.3.1]{ggi}}]\label{fundamental solution}
              There exists a unique holomorphic function 
              $S:\mathbb{P}^1\setminus \{0\}\to \mathrm{End}(H^\bullet(X))$ 
              with $S(\infty)=\id_{H^\bullet(X)}$
              such that 
              \begin{align*}
                 \nabla(S(z)z^{-\mu}z^{\rho_{\!{}_X}}\alpha)=0& \text{ for all } \alpha \in \coho{X}, \\
                 T(z):=z^\mu S(z)z^{-\mu} &\text{ is regular at } z=\infty \text{ and } T(\infty)=\id_{\coho{X}},
              \end{align*} 
              where ${\rho_{\!{}_X}}=(c_1(X)\cup)\in \End(\coho{X})$ 
              and we define $z^{-\mu}:=\exp (-\mu\log z)$, 
              $z^{\rho_{\!{}_X}}:=\exp ({\rho_{\!{}_X}}\log z)$.
              
              Moreover, we have
              \begin{align*}
              \hspace{1.2in}
               (S(-z)\alpha,S(z)\beta)_X=(\alpha,\beta)_X \text{ for all } \alpha,\beta \in \coho{X}.
               \hspace{1.6in}
              \end{align*}
                     \end{proposition}              
              Let $\theta_\circ$ be a real number generic with respect to $-C_X$.
              Under Assumption \ref{assumption of exponential type}, we have the mutation system 
              $\mathfrak{A}_{\theta_\circ} \mathrm{RH}(\mathscr{M}_X,\mathscr{Q}_X)$.
              Denote the underlying polarized vector space by
              $(V_{\mathscr{M}_X},[\cdot,\cdot\rangle_{\mathscr{Q}_X})$,
              which is independent of a choice of $\theta_\circ$.
              
              Let $\mathrm{Bl}^\R_0(\C)$ be the real blowing up of $\C$ at $0$.
              Let $\widetilde{\mathrm{Bl}^\R_0(\C)}$ be its universal covering.
              The universal covering $\widetilde{\C^*_z}$ of ${\C^*_z}$ can be considered as an 
              open subset of $\widetilde{\mathrm{Bl}^\R_0(\C)}$. 
              Its complement is identified with $\R$, which is universal covering of $S^1$.
              By construction, the vector space $V_{\mathscr{M}_X}$ can be 
              identified with the space of flat section of $\tilde{p}^{*}\mathcal{H}_X$ 
              where $\tilde{p}:\widetilde{\C^*_z}\to\C_z$ denotes the composition of the universal covering and the inclusion.
              Namely, we can canonically identify $V_{\mathscr{M}_X}$ with
              \begin{align}\label{flat section}
              \left\{ s:\widetilde{\C^*_z}\to H^\bullet(X)\middle | \nabla s=0 \right\}. 
              \end{align}              
             
             The first half of Proposition \ref{fundamental solution} gives  an isomorphism 
             $\coho{X}\simeq V_{\mathscr{M}_X}$             
             by 
             \begin{align*}
             \Phi(\alpha):=(2\pi)^{-\dim X/2}S(z) z^{-\mu} z^{\rho_{\!{}_X}} \alpha,
             \end{align*}
             where $V_{\mathscr{M}_X}$ is identified with (\ref{flat section}).
             By the second half of Proposition  \ref{fundamental solution},
             if we put 
             \begin{align*}
              [\alpha,\beta)_X:=\frac{1}{(2\pi)^{\dim X}}\int_Xe^{\pi\i\mu}e^{-\pi\i\rho_{{}\!_X}}\alpha \cup \beta
             \end{align*}             
             for $\alpha,\beta\in\coho{X}$, the isomorphism $\Phi$ is compatible with the pairings
             $ [\cdot,\cdot)_X$  and $[\cdot,\cdot\rangle_{\mathscr{Q}_X}$.    
             \begin{definition}\label{3.5}
              The mutation system on 
              $(\coho{X},[\cdot,\cdot)_X)$ defined via the isomorphism 
              $$\Phi:(\coho{X},[\cdot,\cdot)_X)\simeqto(V_{\mathscr{M}_X},[\cdot,\cdot\rangle_{\mathscr{Q}_X})$$
              described above 
              is called an A-mutation system of $X$, which is also denoted by 
              $\mathfrak{A}_{\theta_\circ} \mathrm{RH}(\mathscr{M}_X,\mathscr{Q}_X)$.
              The pair $(\tau_{\theta_\circ},{}^{A}\!f_{\theta_\circ})$ denotes the underlying splitting data.
             \end{definition}
             We can describe this mutation system more concretely as follows:             
             \begin{lemma}\label{criterion c}
             Fix a hermitian metric $\| \cdot\|$ on $\coho{X}$.
             Let $c$ be a complex number in $-C_X$. 
             A class $\alpha\in\coho{X}$ is in $\mathrm{Im} {}^A\! f_{\theta_\circ, \idx}$
             if and only if there exists a  non-negative integer $N$ 
             such that 
             \begin{align}\label{(3.5)}
             \| e^{-c/z}\Phi(\alpha)(z)\|\leq O(|z|^{-N})& 
             \end{align}  
             for $\Im \log z\in I_{\theta_\circ}.$
             \end{lemma}    
             \begin{proof}
             In general, for $(\L,\L_\bullet)\in\st^C$ and $C$-generic $\theta_\circ\in\R$, 
             we have 
             $$\mathrm{Im} f_{\theta_{\circ, c}}=\left\{s\in V_\L \middle| s_\theta \in (p^{-1}\L_{\leqslant c})_\theta \text{ for } 
             \theta\in I_{\theta_\circ} \right\}.$$
             In the case $(\L,\L_\bullet)=\mathrm{RH}(\mathscr{M}_X)$, 
             the germ 
             $(p^{-1}\L_{\leqslant c})_\theta$ 
             can be identified with the space of 
             $\alpha \in \coho{X}$ 
             with the following properties: 
             There exists an  open neighborhood $U$ of $\theta$ in $\widetilde{\mathrm{Bl}^\R_0(\C)}$ 
             (remark that $\theta \in \R\subset \widetilde{\mathrm{Bl}^\R_0(\C)}$)
             such that 
             for any compact subset $K$ in $U$ there exists a positive integer 
             $N$ with 
             $ \| e^{-c/z}\Phi(\alpha)( z)\|\leq O(|z|^{-N})$ for $\log z\in K\cap \widetilde{\C^*}$.
             Hence the condition (\ref{(3.5)}) implies $\alpha \in \im \lua f_{\theta_\circ, \idx}.$
             On the other hand, by replacing $I_{\theta_\circ}$ with a bigger $(-\is_X)$-good open interval,
             we see that $\alpha \in \im \lua f_{\theta_\circ, \idx}$ implies the condition (\ref{(3.5)}).
             \end{proof}
             \begin{remark}\label{jigen}
             By Remark $\ref{remark RH taiou}$ and Corollary $\ref{cor of exp type}$,
             the dimension of $\mathrm{Im}f_{\theta_\circ,c}$ is equal to the dimension 
             of the eigenspace of $-c_1(X)*_0$ associated with $c$.
             \end{remark}
\section{B-mutation systems}\label{sectionbmut}

\subsection{Hochschild homology for smooth projective varieties}
We recall some definitions and properties of Hochschild homology of smooth projective varieties.
We mainly follow the formulation of \cite[\S 5]{huybook} (see also \cite{lunlef}).

Let $X, Y, Z$ be smooth projective varieties defined over $\C$.
The dimension of $X, Y, Z$ are denoted by $d_X, d_Y, d_Z$ respectively. 
We denote by $D^b(X)$ the triangulated category of bounded complexes of coherent sheaves on $X$.
The shift functor is denoted by $[1]$. 
For a morphism $f:X\to Y$, we denote by $f_*$ the right derived direct image functor and $f^*$ the left derived inverse image functor.
Moreover the left derived tensor product is denoted by $\otimes$.

Let $\E \in  D^b(X\times Y)$ and $\F \in D^b(Y\times Z)$ be bounded complexes.
We define the exact functor
\[\Phi_\E:D^b(X)\to D^b(Y)\] by 
\[\Phi_\E(-):=(\pi_2)_*(\pi^*_1(-)\otimes\E),\]
where $\pi_i$ is the projection from $X\times Y$ to the $i$-th factor.
For the diagonal sheaf 
\[\oo_\Delta:=\Delta_*\oo_X \in D^b(X\times X),\]
we have $\Phi_{\oo_\Delta}\cong \id,$ where $\Delta:X\to X \times X $ is the diagonal embedding.
We define the composition of kernels $\E$ and $\F$ by
\[\F \circ \E:=(\pi_{1,3})_*(\pi^*_{1,2}\E \otimes \pi^*_{2,3}\F) \in D^b(X \times Z),\]
where $\pi_{i,j}$ is the projection from $X\times Y\times Z$ to the $i^{\text{th}}\times j^{\text{th}}$ factor.
Then we have
$\Phi_{\F \circ \E}\cong \Phi_\F\circ \Phi_\E$.


We define the Hochschild homology of $X$ as follows:
\[\mathrm{HH}_k(X):=\bigoplus_{p-q=k}H^q(X, \Omega^p_X),
   \quad \HH{X}:=\bigoplus_{k \in \Z} \mathrm{HH}_k(X)\]
This is a $\Z$-graded vector space.
\begin{remark}[see, e.g., \cite{calmuk2}, \cite{hkr}, \cite{swahoc}, \cite{yekcon}]
More precisely, the Hochschild homology is defined by
\[\mathrm{HH}_k(X):=\Hom_{D^b(X)} (\oo_X[k], \Delta^* \oo_{\Delta})\]
and we have the Hochschild-Kostant-Rosenberg isomorphism 
\[\mathrm{I}_\mathrm{HKR} : \bigoplus_{k \in \Z} \mathrm{HH}_k(X) \simeqto
   \bigoplus_{k \in \Z} \bigoplus_{p-q=k} H^q(X, \Omega^p_X).\] 
We use the right hand side as a definition of the Hochschild homology.
\end{remark}
For a homogeneous element $\alpha \in \HH{X}$,
the degree of $\alpha$ is denoted by $\deg \alpha.$
We identify $\HH{X}$ with $\coho{X}$ via the Hodge decomposition (as  $\Z/2\Z$-graded vector spaces).
For a morphism $f:X\to Y$, we denote by $f_*$ the Gysin map.
\begin{remark}
By definition, $f_*$ satisfies the following:
\[\frac{1}{(\tat)^{\di_Y}} \int_Y f_*\alpha \cup \beta =
   \frac{1}{(\tat)^{\di_X}} \int_X \alpha \cup f^* \beta. \]
Here $\alpha \in \HH{X}$ and $\beta \in \HH{Y}.$ 
\end{remark} 
For $\E \in D^b(X\times Y),$ a morphism 
$\phi_{\E}:\HH{X}\to \HH{Y}$ is defined as follows:
\[\phi_{\E}(-):=(\pi_2)_*(\pi^*_1(-)\cup\upsilon(\E)),\]
where 
\[\upsilon(\E):=\Ch(\E)\sqrt{\Td_{X\times Y}}\]
is the Mukai vector.
Then we have 
\[\phi_{\oo_\Delta}=\id, \quad \phi_{\F}\circ \phi_{\E}=\phi_{\F \circ \E}.\]
We note that $\phi_{\E}$ preserves the $\Z$-grading.

\begin{remark}
By definition, Chern characters $\Ch$ and Todd classes $\Td$ of bounded complexes of coherent sheaves on $X$ are elements of
$\im \big(\bigoplus^{d_X}_{k=0} H^{2k}\big(X ; \Z(k)\big) \to \coho{X}\big),$
where $\Z(k)$ is the $k$-th Tate twist of $\Z$ $($see \cite[\S 3.4]{ggi}$)$.
\end{remark} 

An exact functor $F:D^b(X)\to D^b(Y)$ is called a Fourier-Mukai functor if there exists 
$\E \in D^b(X\times Y)$ such that $F \cong \Phi_\E.$
The complex $\E$ is called a Fourier-Mukai kernel of $F.$ 
For a Fourier-Mukai functor $F \cong \Phi_\E$, we define $\phi_{F}$ by $\phi_{\E}.$
By the next lemma, we see $\phi_{F}$ is independent of a choice of a Fourier-Mukai kernel $\E$.

\begin{lemma}[{\cite[Corollary 4.4]{stenon}}]\label{cs}
Let $\E_1,\E_2 \in D^b(X\times Y).$ If $\Phi_{\E_1}\cong \Phi_{\E_2},$ then $\phi_{\E_1}=\phi_{\E_2}.$
\end{lemma}
\begin{proof}
 By \cite[Corollary 4.4]{stenon}, we have $[\E_1]=[\E_2]$ in the $K$-group, which implies the lemma.
\end{proof}

Let $\E_1\to \E_2\to\E_3\to\E_1[1]$
be an exact triangle in $D^b(X\times Y)$.
Since Chern characters are additive, we have
\[\phi_{\E_2}=\phi_{\E_1}+\phi_{\E_3}.\]
By considering the case $\E_2\cong 0$, we have
$\phi_{\E[1]}=-\phi_{\E}$.

We define the left and right adjoint kernels $\E^*, \E^!\in D^b(Y\times X)$ by 
  \[\E^*:=(\sigma_{XY})^*(\E^\vee \otimes \pi^*_2 \omega_Y[d_Y]),\quad
     \E^!:=(\sigma_{XY})^*(\E^\vee \otimes \pi^*_1 \omega_X[d_X]),\] 
where $\sigma_{XY}:Y \times X \to X \times Y$ is the natural isomorphism,
$\E^\vee$ is the dual $R\HOM(\E, \oo_{X\times Y})$, 
and $\omega_X, \omega_Y$ are the canonical bundles. 
Then we see that 
$\Phi_{\E^*}$ is the left adjoint of $\Phi_\E$ and $\Phi_{\E^!}$ is the right adjoint of $\Phi_\E$.
Since the operations $*$ and $!$ preserve exact triangles, it follows that 
 \begin{align}\label{additive}
   \phi_{\E^*_2}=\phi_{\E^*_1}+\phi_{\E^*_3}, \quad    
   \phi_{\E^!_2}=\phi_{\E^!_1}+\phi_{\E^!_3}
 \end{align}
for an exact triangle $\E_1\to \E_2\to\E_3\to\E_1[1]$.


\subsection{Hochschild homology for admissible subcategories}
We define Hochschild homology for admissible subcategories and construct objects of $\rep_\Z(\C)$.
Similar construction has already been considered by many people
(e.g., \cite{calmuk1}, \cite{kuzhoc}, \cite{polche}, \cite{shkhir}).

For a functor $F$, we denote by $F^*$ (resp. $F^!$) the left (resp. right) adjoint functor.
Let $\A$ be a full triangulated subcategory of $D^b(X)$ and $i_\A:\A \into D^b(X)$ be the inclusion functor. 
$\A$ is called admissible if $i_\A$ has left and right adjoint functors.
Note that an admissible subcategory $\A$ is saturated, and hence all fully faithful functors to triangulated categories of finite type are also admissible (see \cite[\S 2]{bonrep}).
For a functor $F$ to $\A$, the functor $i_\A \circ F$ is denoted by $\widetilde{F}.$ 
For an admissible subcategory $\A$, we define left and right orthogonal to $\A$ by  
    \begin{align*}
         \lperp\A&:=\{A'\in D^b(X)\mid\forall A \in \A \ \ \Hom (A',A)=0\}, \\
         \A^\perp&:=\{A'\in D^b(X)\mid\forall A \in \A \ \ \Hom (A,A')=0 \}.
    \end{align*}
Then $\lperp\A$ and $\A^\perp$ are also admissible (see \cite[Proposition 3.6]{bonrep}).
Moreover we have semiorthogonal decompositions
\[D^b(X)\cong  \langle \A, \lperp\A\rangle \cong \langle \A^\perp,\A\rangle.\]
See Definition \ref{4.12} for the definition of semiorthogonal decompositions.
We denote by $L_\A$ (resp. $R_\A$) the projection functor from $D^b(X)$ to $\A$ with respect to the semiorthogonal decomposition
$\langle \A,\lperp\A\rangle$ (resp. $\langle \A^\perp,\A\rangle$).

\begin{lemma}
  $i^*_\A \cong L_\A$ and $i^!_\A \cong R_\A.$
\end{lemma}
\begin{proof}
     For $E \in D^b(X)$, we have an exact triangle 
     \[ \widetilde{R}_{\lperp \A} E \to E \to \widetilde{L}_\A E \to \widetilde{R}_{\lperp \A} E[1].\] 
     Since $\Hom_{D^b(X)}(\widetilde{R}_{\lperp \A} E, A)=0$ for $A \in \A$,
     by applying $\Hom_{D^b(X)}(-, A)$ to the above exact triangle, 
     we have a functorial isomorphism $\Hom_{\A}(L_\A E, A)\cong \Hom_{D^b(X)}(E, A).$
     This implies $i_\A^*\cong L_\A$.
     The proof of $i^!_\A \cong R_\A$ is similar.
\end{proof}

To define $\phi_{\widetilde{L}_\A}$ and $\phi_{\widetilde{R}_\A}$,
we need the following lemma:

\begin{lemma}[\cite{kuzhoc}]
  $\widetilde{L}_\A$ and $\widetilde{R}_\A$ are Fourier-Mukai functors.
\end{lemma}
\begin{proof}
Apply \cite[Theorem 3.7]{kuzhoc} for semiorthogonal decompositions
$\langle \A, \lperp\A\rangle$ and $\langle \A^\perp, \A\rangle$. 
\end{proof}

Using the above lemma, we define the Hochschild homology of $\A$ as follows:
\begin{definition}
 We define the Hochschild homology $\HH{\A}$ of $\A$ by $\ima \phi_{\widetilde{R}_\A} \subset \HH{X}.$
\end{definition}


Let $\B\subset D^b(Y)$ be an admissible subcategory. 
A functor $F:\A\to \B$ is called a Fourier-Mukai functor
if there exists $\E \in D^b(X\times Y)$ such that  $\Phi_\E|_\A \cong \widetilde{F}.$
The complex $\E$ is called a Fourier-Mukai kernel of $F$.
To define $\phi_F$ for a Fourier-Mukai functor $F$, we need some lemmas.

\begin{lemma}\label{inclusion}
  Let $\B \subset D^b(Y)$ be an admissible subcategory
  and $\E \in D^b(X\times Y)$ be a bounded complex. 
  If $\Phi_\E(\A) \subset \B$, then $\phi_{\E}(\HH{\A}) \subset \HH{\B}$.
\end{lemma}
\begin{proof}
  By assumption, we have 
  \[\Phi_\E \circ \widetilde{R}_\A \cong \widetilde{R}_\B \circ \Phi_\E \circ \widetilde{R}_\A.\]
  Hence we have 
  \[ \ima (\phi_{\E} \circ \phi_{\widetilde{R}_\A}) \subset \ima \phi_{\widetilde{R}_\B}, \]
which implies the lemma.
\end{proof}

\begin{lemma}\label{imeq}
  If $\Phi_{\E_1}|_\A\cong \Phi_{\E_2}|_\A,$ then $\phi_{\E_1}|_{\HH{\A}}=\phi_{\E_2}|_{\HH{\A}}.$
\end{lemma}
\begin{proof}
   By assumption, we have $\Phi_{\E_1} \circ \widetilde{R}_\A \cong \Phi_{\E_2} \circ \widetilde{R}_\A.$
   This implies the lemma.
\end{proof}

For a Fourier-Mukai functor $F:\A \to \B$ with a Fourier-Mukai kernel $\E,$ 
we define $\phi_F$ by $\phi_{\E}|_{\HH{\A}}.$
By Lemma \ref{inclusion}, we can consider $\phi_F$ as a morphism from $\HH{\A}$ to $\HH{\B}.$ 
By Lemma \ref{imeq}, this is independent of a choice of a Fourier-Mukai kernel.
Note that $\id_\A$ and $i_\A$ are  Fourier-Mukai functors with a Fourier-Mukai kernel $\oo_{\Delta}.$
Hence we have $\phi_{\id_\A}=\id_{\HH{\A}}$
and $\phi_{i_\A}$ is the natural embedding of $\HH{\A}.$
To define $\HH{\A}$, we used the specific projection $\widetilde{R}_\A$.
The next lemma implies that
the definition of $\HH{\A}$ is independent of a choice of a projection.

\begin{lemma}\label{proj}
Let $P:D^b(X)\to\A$ be a Fourier-Mukai functor which satisfies $P|_\A\cong \id_\A.$
Then $\ima \phi_P=\HH{\A}.$
\end{lemma}
\begin{proof}
By assumption, we have $P \circ \widetilde{R}_\A \cong R_\A$ and $ R_\A \circ \widetilde{P} \cong P.$
Hence we have 
$\im \phi_{R_\A} \subset \im \phi_P$ and $ \im \phi_P \subset \im \phi_{R_\A}.$ 
\end{proof}

We construct an object of $\rep_\Z(\C)$ from $\A$ and the Serre functor.
We recall that for $E \in D^b(X),$ the functor$-\otimes E$
is a Fourier-Mukai functor with a Fourier-Mukai kernel $\Delta_*E$.
Hence the Serre functor $S_X\cong -\otimes \omega_X[d_X]$ of $D^b(X)$ is a Fourier-Mukai functor. 

\begin{lemma}\label{serre}
$R_\A \circ S_X \circ i_\A$ is the Serre functor $S_\A$ of $\A$.
Especially, $S_\A$ is a Fourier-Mukai functor.
\end{lemma}
\begin{proof}
Let $A_1, A_2 \in \A$. 
This lemma is proved by the following functorial isomorphisms:
\[\Hom_\A(A_1, A_2) \cong \Hom_{D^b(X)}(A_2, S_X(A_1))^\vee \cong
   \Hom_\A (A_2, R_\A\circ S_X (A_1))^\vee. \qedhere\]
\end{proof}

Let $F:\A \to \B$ be a Fourier-Mukai functor with a Fourier-Mukai kernel $\E$.
Then we have $F \cong R_\B \circ \Phi_\E \circ i_\A$, which implies the left adjoint
$F^*\cong L_\A \circ \Phi_{\E^*} \circ i_\B$ is a Fourier-Mukai functor. 
Hence, if $F$ is an isomorphism, the quasi-inverse $F^{-1}\cong F^*$ is a Fourier-Mukai functor and
$\phi_F$ is an isomorphism with the inverse $\phi_{F^{-1}}.$
Set
\[T_\A:=(-1)^\Deg \phi_{S^{-1}_\A},\]
where $(-1)^\Deg$ is the sign operator defined by $(-1)^k \cdot \id$ on $\mathrm{HH}_k(X).$
The above argument shows that the pair $(\HH{\A}, T_\A)$ is an element of 
$\rep_\Z(\C)$.
Moreover, for an isomorphism $F$, we have 
$F\circ S_\A\cong S_\B\circ F$ (see, e.g., \cite[Lemma 1.30]{huybook}), which implies 
\[\phi_F\circ T_\A=T_\B\circ \phi_F.\]
Thus $\phi_F$ is a morphism in $\rep_\Z(\C)$.
For simplicity of notation, we write $T_X$ instead of $T_{D^b(X)}$.
Then we have $(\HH{X}, T_X) \in \rep_\Z(\C).$


\subsection{Stokes data from semiorthogonal decompositions}
In this section we construct Stokes data from (framed) semiorthogonal decompositions.
We first recall the definition of semiorthogonal decompositions.
\begin{definition}\label{4.12}
Let $\mathcal{T}$ be a triangulated category.
A sequence of full triangulated subcategories $\mathcal{A}_1, \mathcal{A}_2, \dots, \mathcal{A}_m$ is called a semiorthogonal decomposition of $\mathcal{T}$
if $\mathcal{A}_i \subset \mathcal{A}_j^\perp$ for $i<j$ and for every $T \in \mathcal{T}$ there exists the following sequence of exact triangles:
\[
  \xymatrix@C=.5em
  {0 \ar@{=}[r] &T_m \ar[rrrr] &&&& T_{m-1}\ar[rrrr] \ar[lld]
   &&&& T_{m-2} \ar[rr] \ar[lld] && \cdots \ar[rr] && T_1\ar[rrrr]
   &&&&  T_0 \ar[lld] \ar@{=}[r] & T  \\
    &&& A_m \ar@{.>}[llu] &&&& A_{m-1} \ar@{.>}[llu]
   &&&&&&&& A_1 \ar@{.>}[llu]}.
\]
Here $A_i$ are objects of $\mathcal{A}_i$. 
\end{definition}
Let $\is$ be a finite set and $\tau:\is\simeqto\{1,2,\dots,m\}$ be a bijection. 
Let $\fra{X}{\is}$ be a family of admissible subcategories in $D^b(X).$ 

\begin{definition}
A pair
$\frs{\is}$ is called a framed semiorthogonal decomposition of type $(\is, \tau)$
with the frame $\fra{X}{\is}$
if $\{\A_i\}_{1\le i \le m}$ is a semiorthogonal decomposition
$D^b(X)=\langle\A_1,\A_2,\dots,\A_m\rangle$
and $\{F_\idx\}_{\idx\in\is}$ is a tuple of Fourier-Mukai isomorphisms
$F_c:\gr_\idx D^b(X)\simeqto \A_{\tau(c)}.$ 
\end{definition}

We denote by $P_i$ the projection functor from $D^b(X)$ to $\A_i$ with respect to the semiorthogonal decomposition $\{\A_i\}_{1\le i \le m}$.  
Then we have
\[P_i\circ \widetilde{P}_j \cong \begin{cases} P_i \quad &(i=j) \\ 0 \quad &(i\neq j) \end{cases}. \]


The next theorem of Kuznetsov is essential for our construction.
\begin{theorem}[{\cite[Theorem 3.7]{kuzhoc}}]\label{kuznetsov}
Every $P_i$ is a Fourier-Mukai functor.
Moreover, we have the following sequence of exact triangles in $D^b(X \times X)$:
\[
  \xymatrix@C=.5em
  {0 \ar@{=}[r] &\mathcal{D}_m \ar[rrrr] &&&& \mathcal{D}_{m-1}\ar[rrrr] \ar[lld]
   &&&&\mathcal{D}_{m-2} \ar[rr] \ar[lld] && \cdots \ar[rr] &&\mathcal{D}_1\ar[rrrr]
   &&&& \mathcal{D}_0 \ar[lld] \ar@{=}[r] &\oo_\Delta  \\
    &&&\mathcal{P}_m \ar@{.>}[llu] &&&&\mathcal{P}_{m-1} \ar@{.>}[llu]
   &&&&&&&& \mathcal{P}_1 \ar@{.>}[llu],}
\]
where $\mathcal{P}_i$ is a Fourier-Mukai kernel of $P_i$.
\end{theorem}
Since $P_i$ is a projection to $\A_i$, using Lemma \ref{proj}, we have 
$\ima \phi_{P_i}=\HH{\A_i}$.
Hence, as a corollary of Theorem \ref{kuznetsov}, we have the following:

\begin{corollary}\label{sum}
$\HH{X}=\bigoplus^m_{i=1}\HH{\A_i}$.
\end{corollary}

Let $(\{\A_i\}_{1\le i \le m}, \{F_\idx\}_{\idx\in\is})$ be a framed semiorthogonal decomposition
of type $(\is, \tau)$ with a frame $\fra{X}{\is}$.
Then
\[\phi_{F_c} : \left( \HH{\gr_\idx D^b (X)}, T_{\gr_\idx D^b(X)} \right)
  \simeqto \left(\HH{\A_{\tau(c)}}, T_{\A_{\tau(c)}}\right)\]
is an isomorphism in $\rep_\Z(\C).$
Set
\begin{alignat*}{4}
\lub f_c&:=\phi_{i_{\A_{\tau(c)}}} \circ \phi_{F_c} &:\HH{\gr_\idx D^b(X)}
\hookrightarrow \HH{X},
\quad &\lub f&:=\coprod_{\idx \in \is} \lub f_c, \\
\lub f_c^*&:=\phi_{F_c}^{-1}\circ \phi_{L_{\A_{\tau(c)}}} &: \HH{X} \twoheadrightarrow
\HH{\gr_\idx D^b(X)},
\quad &\lub f^*&:=\prod_{\idx \in \is} \lub f^*_c, \\
\lub f_c^!&:=\phi_{F_c}^{-1} \circ \phi_{R_{\A_{\tau(c)}}} &: \HH{X} \twoheadrightarrow
\HH{\gr_\idx D^b(X)},
\quad &\lub f^!&:=\prod_{\idx \in \is} \lub f^!_c.
\end{alignat*}
Since $R_{\A_i}\cong i^!_{\A_i}\cong S_{\A_i}\circ i^*_{\A_i} \circ S^{-1}_X,$ 
we have $\lub f^!_c=T^{-1}_c\circ \lub f_c^*\circ T_X.$

\begin{theorem}
$\left(\left(\HH{\gr_\idx D^b(X)}, T_{\gr_\idx D^b(X)}\right)_{\idx \in \is},
\lub f, \lub f^*\right)$
defines Stokes data on $\left( \HH{X}, T_X \right)$ of type $(\is, \tau).$
\end{theorem}
\begin{proof}
By Corollary \ref{sum}, $\lub f$ is an isomorphism.
By the definition of semiorthogonal decomposition, we have $\A_j \subset {}^\perp\!\A_i$ for $i<j$.
Hence we have
\[L_{\A_i}\circ i_{\A_j} \cong
\begin{cases}
\id_{\A_j} & (i=j) \\
0 & (i<j).
\end{cases}
\]
This implies  
\[ \lub f^*_c \circ \lub f_{c'}=
\begin{cases}
\id_{V_{c'}} & (c=c')\\
0 & (\tau(c)<\tau(c')).
\end{cases}
\]  
Similarly, we have 
\[ \lub f^!_c \circ \lub f_{c'}=
\begin{cases}
\id_{V_{c'}} & (c=c')\\
0 & (\tau(c)>\tau(c')),
\end{cases}
\] 
which proves the theorem.                            
\end{proof}

\begin{definition}
Let $\frs{\is}$ be a framed semiorthogonal decomposition of type $(\is, \tau)$ with a frame $\fra{X}{\is}$.
We define Stokes data
\[\mathfrak{B} \big( \frs{\is} \big) \]
on $\big(\HH{X}, T_X \big)$ of type $(\is, \tau)$ by the tuple 
\[\left( \left( \HH{\gr_\idx D^b (X)}, T_{\gr_\idx D^b(X)} \right)_{\idx \in \is},
\lub f, \lub f^*\right).\]
\end{definition}


\subsection{Mutations of framed semiorthogonal decompositions}
In this section, we define mutations of framed semiorthogonal decompositions.
Essentially, this construction is due to \cite{bonrep} (see also \cite{kuzhoc}).
  
Let $\frs{\is}$ be a framed semiorthogonal decomposition of type $(\is, \tau)$ with a frame $\fra{X}{\is}$.
\begin{definition}[{cf. \cite[\S 2.4]{kuzhoc}}]
For each $i\in\{1,2,\dots, m-1\}$, we define  
\begin{align*}
\rms_j&:=\begin{cases}
                                \A_j &(j\neq i, i+1) \\
                                \A_{i+1} & (j=i)        \\
                                \lperp \langle \A_1, \A_2,\dots, \A_{i-1}, \A_{i+1} \rangle 
                                \cap \langle \A_{i+2},\A_{i+3},\dots,\A_m\rangle^\perp & (j=i+1)
                            \end{cases} \\
(\R_{i+1} F_\bullet)_\idx&:= \begin{cases}
                             F_c & (\tau(\idx) \neq i) \\
                             R_{\lperp \A_{i+1}} \circ F_\idx & (\tau(\idx)=i)
                             \end{cases} 
\end{align*}
The pair $\rfrs$ is called the right mutation of $\frs{\is}$.
Similarly, we define the left mutation $\lfrs$ as follows:
\begin{align*}
\lms_j&:=\begin{cases}
                                          \A_j &(j\neq i, i+1) \\
                                          \A_i & (j=i+1)        \\
                                          \lperp\langle\A_1, \A_2,\dots, \A_{i-1}\rangle\cap\langle\A_{i},\A_{i+2},
                                          \A_{i+3}, \dots,\A_m\rangle^\perp & (j=i)
                            \end{cases} \\
(\mathbb{L}_i F_\bullet)_\idx&:= \begin{cases}
                                           F_\idx & (\tau(\idx)\neq i+1)\\
                                           L_{\A_i^\perp}\circ F_\idx & (\tau(\idx)=i+1)
                                           \end{cases}
\end{align*}
\end{definition}

To prove the propositions below (Proposition \ref{mutfs}, \ref{braid rel}), we need the following:

\begin{lemma}[\cite{bonrep}]\label{bk}
$R_{\lperp \A_{i+1}}|_{\A^\perp_{i+1}}:\A^\perp_{i+1} \to \lperp \A_{i+1}$ is an isomorphism and 
the quasi-inverse is given by
$L_{\A^\perp_{i+1}}|_{\lperp \A_{i+1}}.$
\end{lemma} 
\begin{proof}
See the proof of \cite[Lemma 1.9]{bonrep}. 
\end{proof}


\begin{proposition}\label{mutfs}
  \sloppy $\rfrs$ and $\lfrs$ are framed semiorthogonal decompositions of type $(\is, s_i\circ\tau)$
with the frame $\fra{X}{\is}.$
\end{proposition}
\begin{proof}
We only prove that $\rfrs$ is a framed semiorthogonal decomposition.
The proof for the left mutation is similar.
By definition, $\{\rms_j\}_{1\le j \le m}$ is a semiorthogonal decomposition of $D^b(X).$
Thus it is sufficient to show that
$(\R_{i+1} F_\bullet)_{\idx}$   
is an isomorphism between $\gr_\idx D^b(X)$ and $\rms_{s_i \circ \tau(\idx)}$.
We only consider the case $\idx=\tau^{-1}(i)$ since the case $\idx \neq \tau^{-1}(i)$ is obvious. 
For $A_i\in\A_i$, we have the following exact triangle:
\[R_{\lperp \A_{i+1}}(A_i) \to A_i \to L_{\A_{i+1}}(A_i) \to R_{\lperp\A_{i+1}}(A_i)[1]. \]
This implies 
\[R_{\lperp \A_{i+1}}(A_i) \in (\langle \A_i, \A_{i+1} \rangle \cap \lperp\A_{i+1})
   =(\R_{i+1} \A_\bullet)_{i+1}.\]  
From Lemma \ref{bk} and $\A_i \subset \A^\perp_{i+1},$
it follows that $R_{\lperp\A_{i+1}}|_{\A_i}$ is fully faithful.

It remains to show that 
$R_{\lperp\A_{i+1}}|_{\A_i}: \A_i \to (\R_{i+1} \A_\bullet)_{i+1}$ is essentially surjective.
Since $\A_i$ is saturated, $\im R_{\lperp\A_{i+1}}|_{\A_i}\subset \rms_{i+1}$ is admissible.
Hence it is sufficient to show that
\[\rms_{i+1} \cap {}^\perp (\im R_{\lperp\A_{i+1}}|_{\A_i})\cong0.\] 
Choose $A \in \rms_{i+1} \cap {}^\perp (\im R_{\lperp\A_{i+1}}|_{\A_i}).$ 
Then, for all $B \in \A_i,$ we have 
\[\Hom_{D^b(X)}(A, B)=\Hom_{\lperp \A_{i+1}}(A, R_{\lperp\A_{i+1}}B)=0.\] 
Thus we have $A \in \lperp \A_i \cap \rms_{i+1} \cong 0,$
which completes the proof.
\end{proof}


We construct an action of the braid group $\mathrm{Br}_m$ on the set of framed semiorthogonal decompositions.
Set
\begin{align*}
\sigma_i \frs{\is} &:=\rfrs, \\
\sigma^{-1}_i \frs{\is} &:=\lfrs.
\end{align*}


\begin{proposition}\label{braid rel}
\
    \begin{enumerate}
      \item $\sigma_i^{-1} \sigma_i=\sigma_i\sigma^{-1}_i=\id.$
      \item $\sigma_i\sigma_{i+1}\sigma_i=\sigma_{i+1}\sigma_i\sigma_{i+1}.$ 
      \item $\sigma_i\sigma_j=\sigma_j\sigma_i \quad(|i-j|\ge 2).$
    \end{enumerate}
In other words, $\sigma_i, \sigma^{-1}_i$ satisfy the braid relations.
\end{proposition}
\begin{proof} 

(1) By definition, we have
\[\big{(}\mathbb{L}_i(\R_{i+1} F_\bullet)_\bullet\big{)}_\idx=
  \begin{cases}
    F_\idx \quad &(\tau(\idx) \neq i) \\
    L_{\A^\perp_{i+1}} \circ R_{\lperp \A_{i+1}} \circ F_\idx \quad &(\tau(\idx)=i).
  \end{cases} 
\]
From Lemma \ref{bk} and $\A_i \subset \A^\perp_{i+1},$ 
we see that 
\[L_{\A^\perp_{i+1}} \circ R_{\lperp \A_{i+1}} \circ F_{\tau^{-1}(i)} \cong F_{\tau^{-1}(i)}.\]
This implies $\sigma_i^{-1} \sigma_i=\id.$
Similarly, we can prove $\sigma_i\sigma^{-1}_i=\id.$ \\
(2) Set $\A_{i+2}':= \im R_{\lperp \A_{i+2}}|_{\A_{i+1}}.$
By definition, we have 
\[
\left( \R_{i+1} \left( \R_{i+2} \left( \R_{i+1} F_\bullet \right)_\bullet \right)_\bullet \right)_\idx=
\begin{cases}
F_\idx \quad & (\tau(\idx) \neq i,i+1) \\
R_{\lperp \A_{i+2}} \circ F_\idx \quad & (\tau(\idx)=i+1)\\
R_{\lperp \A_{i+2}} \circ R_{\lperp \A_{i+1}} \circ F_\idx \quad & (\tau(\idx)=i), 
\end{cases}
\]

\[
\left( \R_{i+2} \left( \R_{i+1} \left( \R_{i+2} F_\bullet \right)_\bullet \right)_\bullet \right)_\idx=
\begin{cases}
F_\idx \quad & (\tau(\idx) \neq i,i+1) \\
R_{\lperp \A_{i+2}} \circ F_\idx \quad & (\tau(\idx)=i+1)\\
R_{\lperp \A_{i+2}'} \circ R_{\lperp \A_{i+2}} \circ F_\idx \quad & (\tau(\idx)=i). 
\end{cases}
\]
Hence it is sufficient to show that
\[R_{\lperp \A_{i+2}} \circ R_{\lperp \A_{i+1}}
   \cong R_{\lperp \A_{i+2}'} \circ R_{\lperp \A_{i+2}}.\]
Choose $A \in D^b(X).$ Then we obtain the following sequences of exact triangles:
\[
\xymatrix@C=3.5pt@R=26pt
{0 \ar[rr] && R_{\lperp \A_{i+2}} \circ R_{\lperp \A_{i+1}} A \ar[rr] \ar[ld]
              && R_{\lperp\A_{i+1}} A \ar[rr] \ar[ld] 
              && A \ar[ld] \\
                & R_{\lperp \A_{i+2}} \circ R_{\lperp \A_{i+1}} A \ar@{.>}[lu]
              && L_{\A_{i+2}} \circ R_{\lperp \A_{i+1}} A \ar@{.>}[lu]
              && L_{\A_{i+1}} A \ar@{.>}[lu]}
\]

\[
\xymatrix@C=3.5pt@R=26pt
{0 \ar[rr] && R_{\lperp \A_{i+2}'} \circ R_{\lperp \A_{i+2}} A \ar[rr] \ar[ld]
              && R_{\lperp\A_{i+2}} A \ar[rr] \ar[ld] 
              && A \ar[ld] \\
                & R_{\lperp \A_{i+2}'} \circ R_{\lperp \A_{i+2}} A \ar@{.>}[lu]
              && L_{\A_{i+2}'} \circ R_{\lperp \A_{i+2}} A \ar@{.>}[lu]
              && L_{\A_{i+2}} A \ar@{.>}[lu]}
\]
From these sequences of exact triangles, we see that
$R_{\lperp \A_{i+2}} \circ R_{\lperp \A_{i+1}}$
is the projection to $\lperp \langle \A_{i+1}, \A_{i+2} \rangle$ and 
$R_{\lperp \A_{i+2}'} \circ R_{\lperp \A_{i+2}}$
is the projection to $\lperp \langle \A_{i+2}, \A_{i+2}' \rangle.$
Since 
\[\langle \A_{i+1}, \A_{i+2} \rangle=\langle \A_{i+2}, \A_{i+2}' \rangle\]
as subcategories of $D^b(X)$,
it follows that
\[R_{\lperp \A_{i+2}} \circ R_{\lperp \A_{i+1}}
   \cong R_{\lperp \A_{i+2}'} \circ R_{\lperp \A_{i+2}}.\] \\
(3) Obvious from definition. 
\end{proof}

\begin{remark}
More precisely, the braid group $\mathrm{Br}_m$ acts on the set of equivalence classes of framed semiorthogonal decompositions with the frame $\fra{X}{\is}$.
Here two framed semiorthogonal decompositions $\frs{\is}$ and
$(\{ \A'_i \}_{1\le i \le m}, \{F'_\idx\}_{\idx \in \is})$ with the frame $\fra{X}{\is}$ are equivalent if and only if 
$\widetilde{F}_\idx \cong \widetilde{F}'_\idx$ for all $\idx \in \is.$
\end{remark}



For $\sigma \in \mathrm{Br}_m$, we can define the framed semiorthogonal decomposition
$\sigma \frs{\is}$ of type $(\is, \bar{\sigma} \circ \tau)$
with the frame $\fra{X}{\is}$ as the composition of mutations.

\begin{theorem}\label{compatible to braid}
  $\mathfrak{B} \big( \sigma \frs{\is} \big) = \mathbb{M}_\sigma \big( \mathfrak{B} \frs{\is} \big).$
\end{theorem}
\begin{proof}
By the braid relations, it is sufficient to show the case $\sigma=\sigma_i$ for some $i$. 
We only show that 
\begin{align*}
\phi_{(\R_{i+1} F_\bullet)_\idx}&=R_{i+1} \circ f_\idx, \\
\phi_{(\R_{i+1} F_\bullet)^*_\idx}&= f^*_\idx \circ R^*_{i+1}.
\end{align*}
for $\idx=\tau^{-1}(i)$ since the rest of the proof is evident.
The first equality is shown by 
\[R_{i+1}=\id_{V_X}-\phi_{\widetilde{L}_{\A_{i+1}}}=\phi_{\widetilde{R}_{\lperp \A_{i+1}}}.\]
We will show the second equality.
Note that, for an admissible subcategory $\A,$ we have
\[(\widetilde{R}_\A)^*=(i_\A\circ R_\A)^* \cong \widetilde{L}_\A.\] 
Since
$(\R_{i+1} F_\bullet)_\idx \cong  \widetilde{R}_{\lperp\A_{i+1}} \circ i_{\A_i} \circ F_\idx$,
we have 
\[(\R_{i+1} F_\bullet)^*_\idx \cong F^{-1}_\idx \circ L_{\A_i} \circ \widetilde{L}_{\lperp\A_{i+1}}.\]
Hence it follows that 
\[\phi_{(\R_{i+1} F_\bullet)^*_\idx} 
            = f^*_\idx \circ \phi_{\widetilde{L}_{\lperp \A_{i+1}}}.\]
Since 
$\widetilde{L}_{\lperp \A_{i+1}}  
   \cong (\widetilde{R}_{\lperp \A_{i+1}})^* 
   \cong  S^{-1}_X \circ \widetilde{R}^!_{\lperp \A_{i+1}} \circ S_X,$
we have
\[\phi_{\widetilde{L}_{\lperp \A_{i+1}}} 
 =T_X \circ \phi_{\widetilde{R}^!_{\lperp \A_{i+1}}} \circ T_X^{-1}.\]
By simple computation, it follows  
\begin{align*}
      \phi_{\widetilde{R}^!_{\lperp \A_{i+1}}}
  &=\id_{V_X}-\phi_{\widetilde{L}^!_{\A_{i+1}}}   \\
  &=\id_{V_X}-\phi_{\widetilde{R}_{\A_{i+1}}}    \\
  &= L_{i+1},
\end{align*}
where in the first line we use (\ref{additive}).
Thus we have
\[\phi_{\widetilde{L}_{\lperp \A_{i+1}}}=T_X \circ L_{i+1} \circ T_X^{-1}=R^*_{i+1},\]
which proves the claim.
\end{proof}


\subsection{Pairings on Hochschild homology and B-mutation systems}
 
In this section, we introduce a slightly modified version of the generalized Mukai pairing and show that
this pairing is compatible with the monodromy $T_X$ and the Stokes data constructed from a framed semiorthogonal decomposition.
\begin{definition}[see also \cite{calmuk1}, \cite{pollef}, \cite{ramrel}, \cite{rammuk}, \cite{shkhir}]
We define a Pairing on $\HH{X}$ by 
\[\bil{\alpha}{\beta}_X:=\frac{1}{(\tat)^{\di_X}}\int_{X}e^{\pi \i \rho_{\!{}_X} }W(\alpha) \cup \beta, \]
where $W(\alpha):=\i^{p+q}\alpha$ for $\alpha \in H^q(X,\Omega^p_X).$
\end{definition}

Since
\[ e^{\pi {\tt i} \mu} e^{-\pi {\tt i} \rho_{\!{}_X}}=e^{\pi {\tt i} \rho_{\!{}_X}} e^{\pi {\tt i} \mu}
  ={\tt i}^{-d_X} e^{\pi {\tt i} \rho_{\!{}_X}} W, \]
we see that $\bil{\cdot}{\cdot}_X=[\cdot, \cdot)_X$.
To compute $T_X$, we show the following lemma. 
\begin{lemma}\label{tensor}
For $E \in D^b(X),$ we have $\phi_{-\otimes E}=-\cup\Ch(E)$.
\end{lemma}
\begin{proof}
Let $\pi_1$ (resp. \!\!$\pi_2$) be the projection from $X \times X$ to the first (resp. \!\!second) factor.

We recall that $-\otimes E$ is a Fourier-Mukai functor with a Fourier-Mukai kernel $\Delta_*E.$ 
We compute $\phi_{-\otimes E}(\alpha).$
By definition, we have 
\[\phi_{-\otimes E}(\alpha)=
  \pi_{2*}\big(\pi^*_1(\alpha) \cup \Ch(\Delta_*E)\sqrt{\Td_{X \times X}}\big).\]
Using the Riemann-Roch theorem and the projection formula, we have 
\begin{align*}
     \Ch(\Delta_*E)\sqrt{\Td_{X \times X}}
   &=\Delta_*(\Ch(E)\Td_X) \cup \sqrt{\Td_{X \times X}}^{-1} \\
   &=\Delta_*\big(\Ch(\E) \Delta^*(\sqrt{\Td_{X \times X}})\big) \cup \sqrt{\Td_{X \times X}}^{-1} \\
   &=\Delta_*(\Ch(E)).
\end{align*}
Thus we have
\begin{align*}
   \phi_{-\otimes E}(\alpha)
   &=\pi_{2*}\big(\pi^*_1(\alpha) \cup \Delta_*(\Ch(E))\big) \\
   &=\pi_{2*} \Delta_*\big(\Delta^*\pi^*_1(\alpha) \cup \Ch(E)\big) \\
   &=\alpha \cup \Ch(E). 
\end{align*} 
\end{proof}

As a special case of $E=\omega^{-1}_X[-d_X]$, we have
$T_X= (-1)^{\di_X} (-1)^{\deg} e^{\tat \rho_{\!{}_X}}.$

To show some statements about $\bil{\cdot}{\cdot}_X$,
we use the following properties of the morphism $W$:
\begin{itemize}
\item $W$ is a ring homomorphism. 
\item For $\alpha, \beta \in \HH{X}$, we have $\int_X(-1)^{\di_X}W(\alpha)\cup \beta = \int_X W(\beta) \cup \alpha$.
\item $W\big(c_1(X)\big)=-c_1(X), W\big(\Ch(E^\vee)\big)=\Ch(E),
          W(\sqrt{\Td_X})=\sqrt{\Td_X}\cup e^{-\pi {\tt i} c_1(X)}$. 
\item For a map $f: X\to Y$, we have $W\circ f^*=f^*\circ W, \ W\circ f_*=(-1)^{d_Y-d_X}f_*\circ W$. 
\end{itemize}

\begin{proposition}\label{compmon}
$\bil{\cdot}{\cdot}_X$ is compatible with $T_X.$
\end{proposition}
\begin{proof} 
We need to show $\bil{T_X\alpha}{\beta}_X=(-1)^{\deg\alpha}\bil{\beta}{\alpha}_X,$ i.e., 
\[\frac{1}{(\tat)^{\di_X}}\int_X\e^{\pi{\tt i}\rho_{\!{}_X}}
   W\big( (-1)^{\deg\alpha-\di_X} e^{\tat \rho_{\!{}_X}} \alpha \big) \cup \beta=
   \frac{(-1)^{\deg\alpha}} {(\tat)^{\di_X}} \int_X \e^{\pi {\tt i} \rho_{\!{}_X}} W(\beta) \cup \alpha \]
for homogeneous elements $\alpha, \beta \in \HH{X}$.
Using the above properties of $W$, this follows from simple computation. 
\end{proof}


The next proposition gives a characterization of $\phi_{\E^*}, \phi_{\E^!}$ in terms of the pairing $\bil{\cdot}{\cdot}_X$.

\begin{proposition}[cf. {\cite[Theorem 8]{calmuk1}}]\label{adjoint}
For $\E \in D^b (X \times Y), \alpha \in \HH{Y}, \beta \in \HH{X},$ we have 
  \begin{enumerate}
    \item $\bil{\alpha}{\phi_{\E}(\beta)}_Y=\bil{\phi_{\E^*}(\alpha)}{\beta}_X,$
    \item $\bil{\phi_{\E}(\beta)}{\alpha}_Y=\bil{\beta}{\phi_{\E^!}(\alpha)}_X.$
  \end{enumerate}
\end{proposition}
\begin{proof}
Set $\widetilde{\pi}_1:=\pi_2 \circ \sigma_{XY}$ and $\widetilde{\pi}_2:=\pi_1 \circ \sigma_{XY}.$
We first show 
\begin{align}\label{1.1}
 (-1)^{\di_Y} W\big(\nu(\E^*)\big) \cup \widetilde{\pi}^*_2 e^{\pi{\tt i}c_1(X)}
   =\sigma^*_{XY}\big(\nu(\E) \cup \pi^*_2 e^{\pi{\tt i}c_1(Y)}\big).
\end{align}
By definition, we have
$\Ch(\E^*)=\sigma^*_{XY}\big(\Ch(\E^\vee) \cup (-1)^{d_Y}\pi^*_2 e^{-2\pi{\tt i}c_1(Y)}\big)$,
which implies
\[W\big(\Ch(\E^*)\big)=\sigma^*_{XY}\big(\Ch(\E) \cup (-1)^{d_Y}\pi^*_2 e^{2\pi{\tt i}c_1(Y)}\big).\]
Hence we have 
\[W\big(\nu(\E^*)\big)
    =W\big(\Ch(\E^*)\big) \cup W(\sqrt{\Td_{Y\times X}})
    =\sigma^*_{XY}\big(\Ch(\E) \cup (-1)^{d_Y} \pi^*_2 e^{2\pi{\tt i}c_1(Y)}
      \cup W(\sqrt{\Td_{X\times Y}})\big),\]
where we use $\sigma^*_{XY} \sqrt{\Td_{X \times Y}}=\sqrt{\Td_{Y \times X}}.$
Using $W(\sqrt{\Td_{X\times Y}})=\sqrt{\Td_{X\times Y}}\cup  e^{-\pi {\tt i} c_1(X\times Y)}$, 
it follows that
\begin{align*}
W\big(\nu(\E^*)\big)&=
   \sigma^*_{XY}\big(\Ch(\E) \cup (-1)^{d_Y}\pi^*_2 e^{\pi{\tt i}c_1(Y)}
   \cup \sqrt{\Td_{X\times Y}}\cup \pi^*_1 e^{-\pi{\tt i}c_1(X)}\big) \\
   &=\sigma^*_{XY}\big(\Ch(\E) \cup \pi^*_2 e^{\pi{\tt i}c_1(Y)}
   \cup \sqrt{\Td_{X\times Y}}\big) \cup (-1)^{d_Y} \widetilde{\pi}^*_2 e^{-\pi{\tt i}c_1(X)},
\end{align*}
which implies the equality (\ref{1.1}).

We next show (1).
By definition, we have 
\begin{align*}
 \bil{\phi_{\E^*}(\alpha)}{\beta}_X
 &=\frac{1}{(\tat)^{\di_X}} \int_X e^{\pi{\tt i}\rho_{\!{}_X}} 
   W\big(\widetilde{\pi}_{2*}(\widetilde{\pi}^*_1\alpha \cup \nu(\E^*))\big)\cup \beta \\
 &=\frac{(-1)^{\di_Y}}{(\tat)^{\di_X}} \int_X \widetilde{\pi}_{2*}\big(\widetilde{\pi}^*_1
    W(\alpha) \cup W(\nu(\E^*))\big)\cup e^{\pi{\tt i}\rho_{\!{}_X}}\beta \\
 &=\frac{(-1)^{\di_Y}}{(\tat)^{\di_{X \times Y}}} \int_{Y \times X} \widetilde{\pi}^*_1
        W(\alpha) \cup W\big(\nu(\E^*)\big) \cup \widetilde{\pi}_2^*(e^{\pi{\tt i}\rho_{\!{}_X}}\beta),
\end{align*}
where in the third line we use the projection formula.
Using the equality (\ref{1.1}) and the projection formula, we have
\begin{align*}
 \bil{\phi_{\E^*}(\alpha)}{\beta}_X
 &=\frac{1}{(\tat)^{\di_{X \times Y}}} \int_{Y \times X}\sigma^*_{XY}
     \big(\pi^*_2 W(\alpha) \cup \nu(\E) \cup \pi^*_2 e^{\pi{\tt i}c_1(Y)} \cup \pi^*_1\beta\big) \\
 &=\frac{1}{(\tat)^{\di_{X \times Y}}} \int_{X \times Y}
     \pi^*_2\big(e^{\pi{\tt i} \rho_{\!{}_Y}} W(\alpha)\big) \cup \pi^*_1 \beta \cup \nu(\E) \\
 &=\frac{1}{(\tat)^{\di_Y}} \int_Y e^{\pi{\tt i}\rho_{\!{}_Y}}
        W(\alpha) \cup \pi_{2*}\big(\pi^*_1 \beta \cup \nu(\E)\big). 
\end{align*}

This proves (1).

Finally, we show (2). 
Since $\Phi_{\E^!} \cong S_X \circ \Phi_{\E^*} \circ S^{-1}_Y,$ 
we see that $\phi_{\E^!}=T_X^{-1} \circ \phi_{\E^*} \circ T_Y.$
Hence the statement follows from (1) and Proposition \ref{compmon}.  
\end{proof}



Let $\A$ be an admissible subcategory of $D^b(X).$
We define a pairing $\bil{\cdot}{\cdot}_\A$ on $\HH{\A}$
by the restriction of $\bil{\cdot}{\cdot}_X$ to $\HH{\A}.$

\begin{lemma}\label{25}
The pairing $\bil{\cdot}{\cdot}_\A$ is compatible with $T_\A.$
\end{lemma}
\begin{proof}
By Lemma \ref{serre}, we have $S_\A \cong R_\A \circ S_X \circ i_\A$.
Hence we have $S^{-1}_\A \cong S^*_\A \cong L_\A \circ S^{-1}_X \circ i_\A$ and 
$T_\A=\phi_{L_\A} \circ T_X|_{\HH{\A}}$.
Since $L^!_\A \cong i_\A$,
the statement follows from Proposition \ref{compmon} and (2) of Proposition \ref{adjoint}.
\end{proof}

\begin{lemma}\label{26}
Let $\A$ and $\B$ are admissible subcategories of $D^b(X)$.
Suppose that $\B \subseteq \A^\perp$.
Then $\bil{\alpha}{\beta}_X=0$ for $\alpha \in \HH{\A}$ and $\beta\in \HH{\B}$.
\end{lemma}
\begin{proof}
Note that $\widetilde{L}^!_\A=(i_\A\circ L_\A)^!\cong \widetilde{R}_\A$.
Hence we have 
  \[\bil{\alpha}{\beta}_X
   =\bil{\phi_{\widetilde{L}_\A} (\alpha)}{\beta}_X
   =\bil{\alpha}{\phi_{\widetilde{R}_{\A}}(\beta)}_X.\]
Since $\B \subset \A^\perp$ we see $\widetilde{R}_{\A}|_{\B}\cong 0,$
which implies the lemma.
\end{proof}

By applying this lemma to $\A$ and $\A^\perp$, 
we see that $\bil{\cdot}{\cdot}_\A$ is non-degenerate.
In general, for a Fourier-Mukai functor $F$, the morphism $\phi_F$ does not preserve the pairings.
\begin{lemma}\label{27}
Let $F: \A \to \B $ be a Fourier-Mukai functor.
Suppose that $F$ is fully faithful.
Then $\phi_F$ preserves the pairings.
\end{lemma}
\begin{proof}
Note that $F^*$ is a Fourier-Mukai functor.
Since $F$ is fully faithful, it follows $F^*\circ F \cong \id_\A$ (e.g., \cite[Corollary 1.22]{huybook}).
Thus we have 
\[\bil{\phi_F(\alpha)}{\phi_F(\beta)}_\B
   =\bil{\phi_{F^*\circ F}(\alpha)}{\beta}_\A
   =\bil{\alpha}{\beta}_\A. \qedhere\]
\end{proof}
Let $(\{\A_i\}_{1\le i \le m}, \{F_\idx\}_{\idx\in\is})$
be a framed semiorthogonal decomposition of type $(\is, \tau)$ with a frame $\fra{X}{\is}$ 
of $D^b(X).$\\
Combining Lemmas \ref{25}, \ref{26}, and \ref{27}, we conclude
\begin{theorem}
The pairing $\bil{\cdot}{\cdot}_X$ is compatible with the Stokes data
$\mathfrak{B} \big( \frs{\is} \big),$
and hence gives a mutation system.
\end{theorem}

\begin{definition}
This mutation system
is called a B-mutation system.
\end{definition}

To simplify notation, the B-mutation system is also denoted by
$\mathfrak{B} \big( \frs{\is} \big)$.


\section{Dubrovin type conjectures}
Let $X$ be a Fano manifold.
Fix $-\is_X$-generic $\theta_\circ \in \R.$
In \S \ref{sectionamut}, under Assumption \ref{assumption of exponential type}, we construct the A-mutation system 
$\amut{X}{\theta_\circ}$ with the splitting data $(\tau_{\theta_\circ}, \lua f_{\theta_\circ})$.
On the other hand, in \S \ref{sectionbmut}, we construct the B-mutation system
$\mathfrak{B} \big( \frs{-\is_X} \big)$
with the splitting data $(\tau_{\theta_\circ}, \lub f)$
from a framed semiorthogonal decomposition of type $(-\is_X,\tau_{\theta_\circ})$
with a frame $\fra{X}{-\is_X}$.

Let
\[\hat{\Gamma}_X=\prod_{i=1}^{d_X} \Gamma(1+\delta_i) \in \coho{X}\]
be the Gamma class of $X$, 
where $\delta_1, \delta_2, \dots, \delta_{d_X}\in H^\bullet(X ; \Z)$ are the Chern roots of $TX$ and
$\Gamma(z)$ is the Gamma function.
We define an isomorphism
\[\Gamma: \HH{X} \to \coho{X} \]
by 
$\Gamma(\alpha):=(\alpha \cup \hat{\Gamma}_X)/\sqrt{\Td_X},$
where we identify $H^\bullet(X)$ with $\HH{X}$ via the Hodge decomposition.

\begin{lemma}
For $\alpha,\beta \in \HH{X},$
we have $\bil{\Gamma(\alpha)}{\Gamma(\beta)}_X=[\alpha, \beta)_X.$
\end{lemma}
\begin{proof}
Note that $\bil{\alpha}{\beta}_X=[\alpha,\beta)_X$.
Hence it is sufficient to show that $\bil{\Gamma(\alpha)}{\Gamma(\beta)}_X=\bil{\alpha}{\beta}_X$.
By the identity 
\[e^{\pi{\tt i}z}\Gamma(1-z)\Gamma(1+z)=\frac{\tat z}{1-e^{-\tat z}},\]
we see that 
\[e^{\pi{\tt i}\rho_{\!{}_X}}W(\hat{\Gamma}_X)\cup\hat{\Gamma}_X=\Td_X.\]
From this and $W(\sqrt{\Td_X})=e^{-\pi{\tt i}\rho_{\!{}_X}}\sqrt{\Td_X},$ we obtain 
\[e^{\pi{\tt i}\rho_{\!{}_X}}W(\Gamma(\alpha))\cup \Gamma(\beta)=e^{\pi{\tt i}\rho_{\!{}_X}}W(\alpha)\cup\beta.\]
Hence we have $\bil{\Gamma(\alpha)}{\Gamma(\beta)}_X=\bil{\alpha}{\beta}_X.$
\end{proof}

\begin{definition}\label{dubconj}
We say that $X$ satisfies Dubrovin type conjecture
if the following conditions hold:
\begin{itemize}
\item  The meromorphic connection $\mathscr{M}_X$ is of 
          exponential type $($Assumption $\ref{assumption of exponential type})$.
\item There exists a framed semiorthogonal decomposition of type
$(-\is_X,\tau_{\theta_\circ})$ with a frame\\ $\fra{X}{-\is_X}$ such that $\Gamma$ gives an isomorphism of mutation systems, i.e., $$\ima \lua f_{\theta_\circ, \idx}=\ima (\Gamma \circ \lub f_\idx)$$
for all $\idx\in-\is_X$.
\end{itemize}
\end{definition}

\begin{remark}
For another $-\is_X$-generic $\theta'_\circ$, we can construct another A-mutation system.
By Theorem $\ref{mutation gives stokesfactor},$
these mutation systems are related by the braid group action.
From this and Theorem $\ref{compatible to braid},$
we see that $X$ satisfies Dubrovin type conjecture for $\theta_\circ$
if and only if  $X$ satisfies Dubrovin type conjecture for $\theta'_\circ$.
\end{remark}

\begin{remark}
\
\begin{itemize}
\item If $($even part of $)$ the quantum cohomology of $X$ is semisimple and $D^b(X)$ has a full exceptional collection, 
then Dubrovin type conjecture defined above follows from Gamma conjecture II $($\cite[Conjecture 4.6.1]{ggi}$).$
\item Gamma conjecture II is closely related to Dubrovin's $($original \!$)$ conjecture \cite{dubgeo}
$($see \cite[\S 4.6]{ggi} for a more detailed explanation$)$
\end{itemize}
\end{remark}


\section{Properties of quantum connections of Fano manifolds}
\subsection{Symmetry of fundamental solutions of quantum connections}\label{subsecsol}

Let $X$ be a Fano manifold. 
We denote by $r_X$ the Fano index of $X,$ that is,  
\[r_X:= \max \left\{ r \in \Z_{>0} \middle| \frac{c_1(X)}{r} \in H^2(X ; \Z)\right\}.\]
Recall that $\mathrm{Pic}(X) \cong H^2(X;\Z)$ since $X$ is Fano. 
Let $\oo(k)$ be the line bundle which satisfies $c_1(\oo(k))=kc_1(X)/r_X.$
Set $H_X:=c_1(\oo(1)).$ 
Let $\C_q$ be the complex plane with the coordinate $q$.
Set $c_1(X)*_q:=c_1(X)*_{c_1(X)\log q}.$
Note that
$c_1(X)*_q \in \mathrm{End}(H^\bullet(X))\otimes\C[q^{r_X}]$
by the divisor axiom.
Recall that the set of eigenvalues of the operator $c_1(X)*_0$ is denoted by $\is_X.$ 
We denote by $E(\idx)$ the generalized eigenspace of $c_1(X)*_0$
associated with the eigenvalue $\idx \in \is_X.$ 
Set $\log \omega_k:=-(\tat k)/r_X $ and $ \omega_k:=e^{\log \omega_k}$ for $k \in \Z.$
We recall the following equation (See, e.g., \cite[\S 2.2]{ggi}):
\begin{align}\label{gauge}
q^\mu(c_1(X)*_q)q^{-\mu}= q (c_1(X)*_0). 
\end{align}
By substituting $\omega_k$ for $q$, we have the following lemma:

\begin{lemma}\label{symexp}
If $ \idx \in C_X,$ then $\omega_k \idx \in C_X.$
Moreover, we have $E(\idx)\cong E(\omega_k \idx)$
and this isomorphism is given by $\alpha \to \omega^{-\mu}_k \alpha.$
\end{lemma}
\begin{proof}
Since $c_1(X)*_q\in \mathrm{End}(H^\bullet(X))\otimes\C[q^{r_X}]$,
it follows $c_1(X)*_{\omega_k}=c_1(X)*_0$. 
Hence the lemma follows from the equation (\ref{gauge}). 
\end{proof}

We introduce a one parameter deformation of the quantum connection
$(\mathcal{H}_X, \nabla)$ of $X$, which is the restriction of 
the usual Dubrovin connection (\cite{dubgeo2},\cite{dubgeo}) to the $c_1(X)$ direction 
(see also \cite[Chapter 8]{coxkatz}, \cite[\S 3.2.1]{kkp2} and references therein).
We consider the trivial bundle 
\[\widetilde{\mathcal{H}}_X:= 
(H^\bullet(X) \times \C_z \times \C_q \to \C_z \times \C_q)\]
with the meromorphic connection 
\[\widetilde{\nabla}:=d - \left(\frac{1}{z}(c_1(X)*_q)-\mu\right)\frac{dz}{z}
              +\left(\frac{c_1(X)*_q}{z}\right)\frac{dq}{q}.\]
By definition,
the restriction of $(\widetilde{\mathcal{H}}_X, \widetilde{\nabla})$ to $q=1$
is the quantum connection $(\mathcal{H}_X, \nabla)$.

Set 
\[\widetilde{S}(z,q):=
   q^{-\mu}S\left(\frac{z}{q}\right) \left(\frac{z}{q}\right)^{-\mu} \left(\frac{z}{q}\right)^{\rho_{\!{}_X}} .\]
This is a holomorphic map defined on the universal cover $\widetilde{\C}^*_z \times  \widetilde{\C}^*_q.$

\begin{lemma}\label{defsol}
$\widetilde{S}$ gives a fundamental solution of the meromorphic connection
$(\widetilde{\mathcal{H}}_X, \widetilde{\nabla})$, that is, $\widetilde{\nabla} \widetilde{S}=0.$
\end{lemma}
\begin{proof}
Set $S'(z):=S(z)z^{-\mu}z^{\rho_{\!{}_X}}.$
This lemma is proved by the following computation:
\begin{align*}
\widetilde{\nabla}_{\frac{\partial}{\partial z}} \widetilde{S} 
             &=q^{-\mu} \frac{1}{q} \left(\frac{dS'}{dz}\right)\left(\frac{z}{q}\right)
               +q^{-\mu} \frac{1}{q} \left(\frac{q}{z}\right) \mu S'\left(\frac{z}{q}\right)
               -q^{-\mu} \frac{1}{q} \left(\frac{q}{z}\right)^2 (c_1(X)*_0) S'\left(\frac{z}{q}\right) \\
             &=q^{-\mu} \frac{1}{q}(\nabla_{\frac{d}{dz}}S')\left(\frac{z}{q}\right) \\
             &=0,  \\
\widetilde{\nabla}_{\frac{\partial}{\partial q}} \widetilde{S}
             &=-\mu q^{-\mu} \frac{1}{q} S'\left(\frac{z}{q}\right)
               -q^{-\mu}\left(\frac{z}{q^2}\right) 
                 \left(\frac{dS'}{dz}\right)\left(\frac{z}{q}\right)
               +q^{-\mu}\frac{1}{z}(c_1(X)*_0) S'\left(\frac{z}{q}\right) \\
             &=-q^{-\mu}\left(\frac{z}{q^2}\right) (\nabla_{\frac{d}{dz}}S')
                 \left(\frac{z}{q}\right) \\
             &=0,
\end{align*}
where we use the equation (\ref{gauge}).
\end{proof}
Recall that $T(z)=z^\mu S(z) z^{-\mu}.$
\begin{proposition}\label{symsol}
$T(z/\omega_k)=T(z).$
\end{proposition}
\begin{proof}
Since $\widetilde{\nabla} \widetilde{S}=0$ and $c_1(X)*_{\omega_k}=c_1(X)*_0,$
we see that
$\widetilde{S}(z, \omega_k)=z^{-\mu} T(z/\omega_k) z^{\rho_{\!{}_X}} \omega^{-\rho_{\!{}_X}}_k$
is a fundamental solution of $(\mathcal{H}_X, \nabla).$
Hence $z^{-\mu} T(z/\omega_k) z^{\rho_{\!{}_X}}$ is also a fundamental solution.
It is obvious that $T(z/\omega_k)= \id$ at $z=\infty.$
This implies $T(z/\omega_k)=T(z)$ by uniqueness (Proposition \ref{fundamental solution}).
\end{proof}


\subsection{Gamma conjecture I}
In this section, we recall Gamma conjecture I and derive some assertions.
First, we recall Property $\oo$.
Set \[\mathrm{T}_X:=\max\left\{|\idx| \middle| \idx \in \is_X \right\}.\]

\begin{definition}[{\cite[Definition 3.1.1]{ggi}}]\label{propertyo} 
We say that $X$ satisfies Property $\oo$ if the following conditions hold:
\begin{itemize}
\item $\mathrm{T}_X \in \is_X.$ \\
\item If $\idx \in \is_X$ and $|\idx|=\mathrm{T}_X,$ then $\idx=\omega_k \mathrm{T}_X$
         for some $k.$ \\
\item The multiplicity of $\mathrm{T}_X$ is one, i.e., $\dim E(\mathrm{T}_X)=1.$
\end{itemize}
\end{definition}

\begin{remark}
Galkin-Golyshev-Iritani \cite{ggi} only considered the even part cohomology
$H^{\mathrm{ev}}(X):=\bigoplus_{k \in 2\Z} H^k(X).$
But the next lemma implies our definition of Property $\oo$ is equivalent to the original one.
Note that the set of eigenvalues of  $(c_1(X)*_0)|_{H^{\mathrm{ev}}(X)}$ is equal to $\is_X$
$($see, e.g., the proof of \cite[Proposition 7.1]{gi}$)$. 
Hence we easily see that our Property $\mathcal{O}$ implies the original one.
\end{remark}

\begin{lemma}\label{6.6}
Assume that $H^\mathrm{ev}(X)$ satisfies the same conditions as Definition \ref{propertyo}.
Then $X$ satisfies Property $\oo$. 
\end{lemma}
\begin{proof}
The proof is the same as \cite[Proposition 1.2]{herupd}.
Set $H^\text{odd}(X):=\bigoplus_{k-1\in 2\Z} H^k(X).$
Let $\alpha$ be an element of $H^\text{odd}(X)\cap E(\idx) \setminus \{0\}.$
Take $\beta \in H^{\text{odd}}(X)$ such that $(\alpha, \beta)_X=1.$
Then we have $\alpha *_0 \beta \in H^\text{ev}(X) \cap E(\idx) \setminus\{0\},$
and $(\alpha *_0\beta)^2=0.$
Thus we have $\dim \big( H^{\text{ev}}(X) \cap E(\idx) \big) \ge 2,$
which implies the lemma.
\end{proof}

We define the subspace 
\[\mathscr{A}_{(\idx, \theta)} \subset \left\{s: \widetilde{\C}^*_z \to H^\bullet(X) \middle|\nabla s=0\right\}\] of flat sections as follows (cf. \cite[\S 3.3]{ggi}). 
A section $s$ is an element of $\mathscr{A}_{(\idx, \theta)}$ if and only if
there exists a constant $C$ and a non-negative integer $N$ such that 
$\|e^{\idx/z} s(z)\| \le C|z|^{-N}$ for $\log z \in \widetilde{\C}^*_z$ with $\Im\log z=\theta, |z|<1$.
Here the norm of $e^{\idx/z}s$ is given by a fixed hermitian metric on $\coho{X}$.
$\mathscr{A}_{(\omega_k \mathrm{T}_X, -2\pi k/r_X)}$ is denoted briefly by $\mathscr{A}_k.$ 
Recall that the isomorphism
\[\Phi: H^\bullet(X) \to \left\{s: \widetilde{\C}^*_z \to H^\bullet(X) \middle|\nabla s=0\right\}\]
is defined by $\Phi(\alpha)(z)=(2\pi)^{-\di_X/2}S(z)z^{-\mu}z^{\rho_{\!{}_X}}\alpha.$

\begin{definition}[{\cite[Conjecture 3.4.3]{ggi}}]
We say that $X$ satisfies Gamma conjecture I
if $X$ satisfies Property $\oo$ and $ \mathscr{A}_0=\C\Phi(\hat{\Gamma}_X)$.
\end{definition}

We have the following lemma:

\begin{lemma}\label{syma}
$\mathscr{A}_{(\idx, \theta)} \cong \mathscr{A}_{(\omega_k\idx, \theta-2\pi k/r_X)}$.
This isomorphism is given by $\alpha \to \Ch(\oo_X(k))\cup \alpha$ via the isomorphism $\Phi$.
\end{lemma}
\begin{proof}
For $s \in \mathscr{A}_{(\idx, \theta)},$
choose $\alpha \in H^\bullet(X)$ which satisfies $s=\Phi(\alpha)$. 
By Lemma \ref{defsol}, 
we have $\omega^{-\mu}_k s(z/\omega_k) \in \mathscr{A}_{(\omega_k\idx, \theta-2\pi k/r_X)}.$ 
By Proposition \ref{symsol},
we have
$ \omega^{-\mu}_k s(z/\omega_k)=\Phi(\omega_k^{-\rho_{\!{}_X}} \cup \alpha).$
The lemma is proved by $\omega_k^{-\rho_{\!{}_X}}=\Ch(\oo_X(k))$.
\end{proof}

From this lemma together with Gamma conjecture I, we can calculate $\mathscr{A}_k$.
\begin{proposition}\label{symv}
If $X$ satisfies Gamma conjecture I, then we have 
\[\mathscr{A}_k=\C \Phi(\hat{\Gamma}_X \Ch(\oo_X(k))).\]
\end{proposition}
\begin{proof}
Obvious from Lemma \ref{syma}.
\end{proof}


\subsection{Vanishing cycles}

In this section, we fix $(-\is_X)$-generic $\theta_\circ$
and assume $\mathscr{M}_X$ is of exponential type (Assumption \ref{assumption of exponential type}).
Let $\amut{X}{\theta_\circ}$
be the A-mutation system with the splitting data $(\tau_{\theta_\circ}, \lua f_{\theta_\circ}).$
\begin{definition}
Let $\A \subset D^b(X)$ be an admissible subcategory.
We say that $\A$ is a vanishing cycle at $(\idx, \theta_\circ)$ 
if $\ima \lua f_{\theta_\circ, -\idx}=\Gamma(\HH{\A}).$
\end{definition}

Recall that $E \in D^b(X) $ is called an exceptional object if 
\[\Hom(E, E[k])=
\begin{cases}
\C \quad &(k=0) \\
0 \quad &(k\neq 0).
\end{cases}
\] 
Let $\langle E \rangle$ be the smallest full strict triangulated subcategory of $D^b(X)$ which contains $E$.
We note that $\langle E \rangle$ is an admissible subcategory of $D^b(X)$ if $E$ is an exceptional object.
(see, e.g., \cite[Lemma 1.58]{huybook}).  
 
\begin{lemma}\label{exceptional} 
Let $E \in D^b(X)$ be an exceptional object. 
Then $\Gamma \big( \HH{\langle E \rangle} \big)=\C \hat{\Gamma}_X\Ch(E).$
\end{lemma}
\begin{proof}
We consider $E$ as an object of $D^b(\{\mathrm{pt}\} \times X).$
Let $\Phi_E: D^b(\{\mathrm{pt}\}) \to D^b(X)$ be the Fourier-Mukai functor with the Fourier-Mukai kernel $E.$
Then we have $\im\Phi_E \subset \langle E \rangle$ and this inclusion induces an isomorphism.
Hence we see that $\phi_E$ induces an isomorphism
$\HH{D^b(\{\mathrm{pt}\})} \cong \HH{\langle E \rangle}.$ 
By definition, we have $\HH{\{\mathrm{pt}\}} \cong \C$,
which implies $\im \phi_{E}=\C\nu(E)$, where $\nu(E)$ is the Mukai vector.
The statement follows from $\Gamma \big( \nu(E) \big)=\hat{\Gamma}_X\Ch(E)$.   
\end{proof}

We note that $\oo(k)$ is an exceptional object since $X$ is Fano.
Moreover, we easily see that the collection of objects $\big( \oo_X(0), \oo_X(1), \dots, \oo_X(r_X-1) \big)$
is exceptional, i.e., $\langle \oo_X(k) \rangle \subset \langle \oo_X(l) \rangle^\perp$ for $k<l.$

\begin{proposition}\label{ok is van cyc}
Suppose that $-\pi/2 \le \theta_\circ \le \pi/2.$
If $X$ satisfies Gamma conjecture I, 
then $\langle \oo_X(k) \rangle$ is a vanishing cycle at
$(\omega_k \mathrm{T}_X, \theta_\circ-2\pi k/r_X).$ 
\end{proposition}
\begin{proof}
By Lemma \ref{criterion c}, we have
\[\ima \lua f_{\theta_\circ-2\pi k/r_X, -\omega_k \mathrm{T}_X}
  \subset \mathscr{A}_k=\C \Phi \big( \hat{\Gamma}_X \Ch (\oo_X(k)) \big).\]
By Remark \ref{jigen} and Lemma \ref{symexp},
we have $ \dim \ima \lua f_{\theta_\circ-2\pi k/r_X, -\omega_k \mathrm{T}_X}=1.$ 
Thus we have
\[\ima \lua f_{\theta_\circ-2\pi k/r_X, -\omega_k \mathrm{T}_X}=
   \C\Phi\big(\hat{\Gamma}_X \Ch(\oo_X(k))\big).\]
Combined with Lemma \ref{exceptional}, we have the statement.
\end{proof}

In the proof of Lemma \ref{van cyc at center}, which is used in the proof of Theorem \ref{propertyd},
we use the following notation: 
\begin{notation}
Let $(V, \bil{\cdot}{\cdot})$ be a finite dimensional vector space $V$ over a field $\bm{k}$
with a pairing $\bil{\cdot}{\cdot}$ and $W$ be a subspace of $V.$
Set
\[ \lperp W:=\left\{ v \in V \middle|\bil{v}{w} = 0 \ \text{for all}\  w \in W \right\},\ 
   W^\perp:=\left\{ v \in V \middle|\bil{w}{v} = 0 \ \text{for all}\  w \in W \right\}.\]
\end{notation}

\begin{lemma}\label{van cyc at center}
Suppose that $\theta_\bullet=\{\theta_\idx\}_{\idx\in-\is_X}$ is ordered.
Take $\idx_\circ\in-\is_X$.
Let $\{\A_{(-\idx, \theta_\idx)}\}_{\idx\in-\is_X\setminus \{\idx_\circ\}}$ be a family of vanishing cycles.
Assume that $\A_{(-\idx,\theta_\idx)}\subset \A^\perp_{(-\idxx,\theta_\idxx)}$
for $\idx<_{\theta_\bullet+\pi/2}\!\idxx$.
Set
\[\A_{(-\idx_\circ,\theta_{\idx_\circ})}:=
  \left(\bigcap_{\idx<_{\theta_\bullet+\pi/2} \idx_\circ}\lperp\A_{(-\idx,\theta_\idx)}\right)
  \bigcap
  \left(\bigcap_{\idx_\circ<_{\theta_\bullet+\pi/2} \idx}\A^\perp_{(-\idx,\theta_\idx)}\right).
\]
Then $\A_{(-\idx_\circ,\theta_{\idx_0})}$ is a vanishing cycle at $(-\idx_\circ, \theta_{\idx_\circ})$.
\end{lemma}
\begin{proof}
By construction, $\{\mathcal{A}_{(-\idx,\theta_\idx)}\}_{\idx\in-\is_X}$
gives a framed semiorthogonal decomposition \[\frs{-\is_X}\]
of type $(-\is_X, \tau_{\theta_\bullet})$
with a frame $\{\mathcal{A}_{(-\idx,\theta_\idx)}\}_{\idx\in-\is_X}$ by requiring 
$F_\idx=\id$ for all $\idx$. 
Hence we can construct a B-mutation system 
\[ \mathfrak{B} \big( \frs{-\is_X} \big)\]
with the splitting data $(\tau_{\theta_\bullet}, \lub f).$
By semiorthogonality of the pairing $\bil{\cdot}{\cdot}_X$, we have

\begin{align}\label{3.8}
\im \lub f_{\idx_\circ}=\left(\bigcap_{\idx<_{\theta_\bullet+\pi/2} \idx_\circ}\lperp(\im \lub f_\idx)\right)
  \bigcap
  \left(\bigcap_{\idx_\circ<_{\theta_\bullet+\pi/2} \idx}{(\im \lub f_\idx)}^\perp\right).
\end{align}

On the other hand, $\im \lua f_{\theta_\bullet, \idx_\circ}$ of the splitting data
$(\tau_{\theta_\bullet}, \lua f_{\theta_\bullet})$ of
$\mathfrak{A}_{\theta_\bullet} \big( \mathrm{RH} (\mathscr{M}_X, \mathscr{Q}_X)\big)$ 
has a similar expression as the equation (\ref{3.8}), i.e., 
\begin{align} 
\im \lua f_{\theta_\bullet, \idx_\circ}=
\left(\bigcap_{\idx<_{\theta_\bullet+\pi/2} \idx_\circ}\lperp(\im \lua f_{\theta_\bullet, \idx}) \right)
\bigcap \left(\bigcap_{\idx_\circ<_{\theta_\bullet+\pi/2} \idx}{(\im \lua f_{\theta_\bullet, \idx})}^\perp\right).
\end{align}
By the definition of vanishing cycles,
we have $\im ( \Gamma \circ \lub f_\idx )=\im \lua f_{\theta_\bullet,\idx}$ for $\idx\neq\idx_\circ.$
Thus we have $\im (\Gamma \circ \lub f_{\idx_\circ})=\im \lua f_{\theta_\bullet,\idx_\circ}$,
which proves the lemma.
\end{proof}


\section{Examples}\label{Example}
In this section, we prove that complete intersection Fano manifolds in projective spaces with Fano index greater than one satisfy Dubrovin type conjectures. 
The proof is similar to the proof of Gamma conjectures for
projective spaces (\cite[\S 5]{ggi}).    
\subsection{Calculations of cohomology}
 
Let $F$ be an ample vector bundle on a Fano manifold $Y$, that is,
the tautological line bundle of the projective space bundle $\mathbb{P}(F)$ is ample.
By \cite[Proposition 1.8]{mukpol}, $F$ is generated  by global sections.
Hence $F$ is convex, that is, $H^1(\mathbb{P}^1,f^*F)=0$
for all non-constant holomorphic maps $f:\mathbb{P}^1 \to Y.$
We assume $1 \le \mathrm{rk}F \le d_Y.$
Let $s$ be a global section of $F$ and set $X:=s^{-1}(0).$
By a Bertini type theorem (\cite[Theorem1.10]{mukpol}), we can choose general $s$ so that $X$ is reduced and smooth.
We assume $X$ is a Fano manifold.
We denote by $i:X \to Y$ the inclusion and set $\amb:=\im i^*.$
Since $F$ is convex, 
$\amb$ is closed under the quantum cup product
$*_\tau$ for $\tau \in \amb$ (see, e.g., \cite{coaqua}, \cite[\S 2.4]{iriqua}, \cite[\S 3.4]{imm}).
The next lemma is well known  (see, e.g., \cite[\S 2.2]{manqua}) and
we only use the case $Y=\mathbb{P}^{d_Y}$, but we give a detailed proof for completeness.


\begin{lemma}
\
\begin{itemize}
\item $\coho{X}=\amb \oplus \amb^\perp.$
\item $\amb^\perp \subset H^{d_X}(X).$
\end{itemize}
Here $\amb^\perp$ is the orthogonal subspace of $\amb$ with respect to the Poincar\'e pairing.
\end{lemma}
\begin{proof}
By definition, we have 
\[\amb \cap \amb^\perp=\left\{i^*\alpha \middle| \int_X i^*\alpha \cup i^*\beta =0 \quad \text{for all}
   \ \beta \in \coho{Y} \right\}. \]
Using the projection formula and $i_*i^*=(\tat)^{d_Y-d_X}e(F)\cup$, we obtain 
\[\int_X i^*\alpha \cup i^*\beta=\int_Y e(F)\cup\alpha\cup\beta,\]
where $e(F)\in H^\bullet(Y ; \Z)$ is the Euler class of $F$.
Thus we have 
\[\amb \cap \amb^\perp=i^*\mathrm{ker}(e(F)\cup-).\]
By Sommese's theorem (see, e.g., \cite[Theorem 7.1.1]{lazbook2}),
it follows that $i^*:H^k(Y) \to H^k(X)$ is an isomorphism for $k<d_X.$
Choose an ample line bundle $L$ on $Y.$
Then
\[(i^*c_1(L)\cup-)^{d_X-k}:H^k(X) \to H^{2d_X-k}(X)\]
is an isomorphism for $k<d_X$ 
by the hard Lefschetz theorem.
Consider the following commutative diagram: 

\[
\xymatrix{
H^k(Y) \ar[r]^{i^*} \ar[d]_{(c_1(L)\cup-)^{d_X-k}} & H^k(X) \ar[d]^{(i^*c_1(L)\cup-)^{d_X-k}} \\
H^{2d_X-k}(Y) \ar[r]^{i^*} & H^{2d_X-k}(X)  
}\]

Here the top and right arrows are isomorphisms for $k<d_X.$
This diagram implies that
\[i^*:H^k(Y) \to H^k(X)\] is surjective for $k\neq d_X.$
Hence we obtain  $\amb^\perp \subset H^{d_X}(X).$
By the hard Lefschetz theorem for ample bundles (e.g., \cite[Theorem 7.1.10]{lazbook2}), it follows that 
\[e(F)\cup-:H^{d_X}(Y) \to H^{2d_Y-d_X}(Y)\]
 is injective.
Thus $H^{d_X}_{\text{amb}}(X) \cap \amb^\perp=\{0\},$ which proves the lemma.
\end{proof}
 
We consider the quantum connection $\mathscr{M}_X$ as a $\mathscr{O}_{\C, 0}(*\{0\})\langle \partial_z\rangle$-module.

\begin{lemma}\label{decomp}
The decomposition $\coho{X}=\amb \oplus \amb^\perp$ induces the decomposition
\[\coho{X}\otimes \mathscr{O}_{\C, 0}(*\{0\}) =\amb\otimes \mathscr{O}_{\C, 0}(*\{0\}) \oplus \amb^\perp\otimes \mathscr{O}_{\C, 0}(*\{0\})\]
as a $ \mathscr{O}_{\C, 0}(*\{0\}) \langle\partial_z\rangle$-module.
Moreover the $\mathscr{O}_{\C, 0}(*\{0\}) \langle\partial_z\rangle$-module
$\amb^\perp\otimes \mathscr{O}_{\C, 0}(*\{0\})$ is of exponential type.
\end{lemma}
\begin{proof}
Recall that $\amb$ is closed under the quantum product $*_0$.
For $\alpha \in \amb^\perp$ and $\beta \in \amb,$ we have 
\[(c_1(X)*_0\alpha, \beta)_X=(\alpha, c_1(X)*_0 \beta)_X=0,\]
where we use $c_1(X)=i^*(c_1(Y)-c_1(F)) \in \amb.$
This implies $c_1(X)*_0\alpha \in \amb^\perp.$
Hence $c_1(X)*_0$ and $\mu$ are compatible with the decomposition 
$\coho{X}=\amb\oplus\amb^\perp$, 
which implies the first statement.
We will show the second statement.
Since $\mu=0$ on $\amb^\perp \subset H^{d_X}(X),$
the differential $\partial _z$ acts on $\amb^\perp \subset H^{d_X}(X)$ as $d/dz-(c_1(X)*_0)/z^2.$
Let $c_1(X)*_0=S+N$ be the Jordan decomposition, where $S$ is the semisimple part and $N$ is the nilpotent part.
Without loss of generality, we can assume $N$ has only one Jordan block.
Take a basis $e_1, e_2, \dots e_k$ of $\amb^\perp$
such that $N(e_i)=e_{i+1}$ for $1 \le i \le k-1$ and $N(e_k)=0$.
We define a grading operator $\mathrm{Gr}$ by $ \mathrm{Gr}(e_i):=i \cdot e_i.$
Then we easily see that
\[z^{\mathrm{Gr}} \left( \frac{d}{dz}-\frac{c_1(X)*_0}{z^2} \right) z^{-\mathrm{Gr}}
   =\frac{d}{dz}-\frac{\mathrm{Gr}}{z}-\frac{N}{z}-\frac{S}{z^2}\]
and the right hand side is obviously of exponential type.
\end{proof}

If the Fano index of $X$ is greater than one,
we can calculate the action of $c_1(X)*_0$ on $\amb^\perp$ as follows:

\begin{lemma}[{\cite[\S 7]{gi}}]\label{primeigen}
If $r_X \ge 2,$ then $c_1(X)*_0\alpha=0$ for $\alpha \in \amb^\perp.$
\end{lemma}
\begin{proof}
Since $\alpha \in \amb^\perp \subset H^{d_X}(X)$ and $r_X \ge 2$,
we have \[c_1(X)*_0\alpha \in H^{d_X+2}(X) \oplus \bigoplus_{k\le d_X-2}H^k(X).\]
But $c_1(X)*_0 \alpha \in \amb^\perp \subset H^{d_X}(X).$
Hence we have $c_1(X)*_0\alpha=0.$
\end{proof}


\subsection{Dubrovin type conjectures for complete intersections in projective spaces}
In this section, we assume that $Y$ is the projective space $\mathbb{P}^{d_Y}$ of dimension $d_Y$ 
and $F=\oo(d_1)\oplus \cdots \oplus \oo(d_k).$
Moreover, we assume
$d_i \ge 2, d_X=d_Y-k \ge 3,$ and $r_X=d_Y+1-d_1-d_2-\cdots-d_k \ge2.$   
 
\begin{proposition}\label{exptype}
The quantum connection $\mathscr{M}_X$ is of exponential type.
\end{proposition}
\begin{proof}
By Lemma \ref{decomp}, it is sufficient to show that
$(\amb\otimes \mathscr{O}_{\C, 0}(*\{0\}), \nabla^{\text{amb}})$ is of exponential type,
where $\nabla^{\text{amb}}$ is the restriction of $\nabla$ to $\amb\otimes \mathscr{O}_{\C, 0}(*\{0\})$.
Let $\log q \in \C$.
Set $q:=e^{\log q}, w:=qz,$ and $\tau:=c_1(X)\log q.$
Note that $c_1(X)*_\tau$ defines an endomorphism of $\amb$ since $\tau \in \amb$. 
Then we can easily check that 
\[q^{-\mu}\nabla^{\text{amb}}_{\frac{d}{dz}}q^{\mu}
  =(w\frac{d}{dw}+\mu-\frac{1}{w}c_1(X)*_\tau)|_{\amb\otimes \mathscr{O}_{\C, 0}(*\{0\})}.\]
By \cite[Corollary 6.14]{reinon},
the right hand side is equipped with a non-commutative Hodge structure for $|q|\ll1$ (see \cite{sabnon} for the definition),
and hence of exponential type.
\end{proof}

As a corollary of Givental's mirror theorem, we calculate the ring structure of $\amb$ with respect to the quantum product. 
\begin{lemma}\label{coh of z}
Set $D_X=d_1^{d_1}d_2^{d_2}\cdots d_k^{d_k}.$
Then \[\amb \cong \C[x]/\langle x^{d_X+1-r_X}(x^{r_X}-D_X) \rangle \] as a ring. 
\end{lemma}
\begin{proof}
By \cite[Corollary 9.3]{givequ},
$H_X \in H^2(X)$ satisfies the relation $H^{d_X+1}_X-D_X H^{d_1+d_2+\cdots+d_k-k}_X=0.$
Hence we have the ring morphism 
\[\C[x]/\langle x^{d_X+1-r_X}(x^{r_X}-D_X)\rangle \to \amb\]
which sends $x$ to $H_X.$
We easily see that these rings have the same dimension.
Since \[H_X^{l}=\overbrace{H_X \cup H_X \cup \cdots \cup H_X}^l+(\text{lower degree term}),\]
the ambient cohomology $\amb$ is generated by $H_X$.
Hence the above morphism is an isomorphism.
\end{proof}

\begin{remark}\label{ind1}
Similarly, by using \cite[Corollary 10.9]{givequ}, we have
\[\amb \cong \C[x]/\langle (x+D'_X)^{d_X}(x^{r_X}-D_X+D'_X) \rangle \]
for the case $r_X=1,$
where $D'_X:=d_1!d_2!\cdots d_k!.$
\end{remark}

\begin{corollary}
$\mathrm{T}_X=r_X D_X^{1/r_X}, C_X=\left\{\mathrm{T}_X\omega_k \middle| k\in \Z\right\}\cup\{0\}.$
Moreover, $X$ satisfies Property $\oo.$
\end{corollary}
\begin{proof}
This statement easily follows from Lemmas \ref{primeigen} and \ref{coh of z}.
\end{proof}

\begin{proposition}\label{GIforZ}
$X$ satisfies Gamma conjecture.
\end{proposition}
\begin{proof}
By \cite[Theorem 5.0.1]{ggi}, projective spaces satisfy Gamma conjectures I.
By \cite[Corollary 3.9, Theorem 8.3]{gi},
it follows that $X$ satisfies Gamma conjecture I by induction on $k.$ 
\end{proof}


\begin{theorem}\label{propertyd}
$X$ satisfies Dubrovin type conjecture.
\end{theorem}
\begin{proof}
Choose a sufficiently small positive real number $\theta_\circ \in \R$.
Set $\idx_\circ:=0$ and $\idx_k:=-\mathrm{T}_X\omega_k \ (k=0,1,\dots, r_X-1)$.
We define a ordered tuple of real numbers $\theta_\bullet$ by 
\[ \theta_\idx:=
  \begin{cases}
    \theta_\circ \quad &(\idx=\idx_\circ)\\
    -2\pi k/r_X+\theta_\circ \quad &(\idx=\idx_k).
  \end{cases}
\]
Then the order $<_{\theta_\bullet+\pi/2}$ on $-\is_X$ is given by
$\idx_\circ<_{\theta_\bullet+\pi/2}\idx_0<_{\theta_\bullet+\pi/2}\idx_1<_{\theta_\bullet+\pi/2}\cdots <_{\theta_\bullet+\pi/2}\idx_{r_X-1}$.
Note that $\theta_\circ$ and $\theta_\bullet$ are $(-\is_X)$-generic. 
By Corollary \ref{orderedmutation},
we can construct a mutation system
$\amut{X}{\theta_\bullet}$ with the splitting data $(\tau_{\theta_\bullet}, \lua f_{\theta_\bullet})$. 

Since $X$ satisfies Gamma conjecture I, using Proposition \ref{ok is van cyc}, we see that $\A_{-\idx_k}:=\langle\oo(k)\rangle$ is a vanishing cycle at $(-\idx_k, \theta_{\idx_k})$.
It is easy to see that $\{\A_{-\idx_k}\}_{{\idx_k} \in -\is_X \setminus \{\idx_\circ\}}$
satisfies the assumption of Lemma \ref{van cyc at center},
which implies 
\[\A_{-\idx_\circ}:=\langle\oo(0),\oo(1),\dots, \oo(r_X-1)\rangle^\perp\]
is a vanishing cycle at $(-\idx_\circ, \theta_{\idx_\circ})$. 
We define a framed semiorthogonal decomposition \[\frs{-\is_X}\]
of type $(-\is_X, \tau_{\theta_\bullet})$
with the frame $\fra{X}{-\is_X}:=\{ \mathcal{A}_{-\idx}\}_{\idx \in \is_X}$
by requiring $F_\idx=\id_{\mathcal{A}_{-\idx}}.$ 
We consider the corresponding B-mutation system
\[\mathfrak{B} \big( \frs{-\is_X} \big) \]
with the splitting data $(\tau_{\theta_\bullet}, \lub f)$.
Then, by the definition of vanishing cycle categories,
we have $\ima (\Gamma \circ \lub f_\idx)=\ima \lua f_{\theta_\bullet, \idx}$ for all $\idx \in -\is_X$. 

Let
$\amut{X}{\theta_\circ}$
be the A-mutation system with the splitting data $(\tau_{\theta_\circ}, \lua f_{\theta_\circ})$.
By Theorem \ref{compatible to braid}, it is sufficient to show that 
\[(\tau_{\theta_\bullet}, \lua f_{\theta_\bullet})=
   (\bar{\sigma} \circ \tau_{\theta_\circ}, \sigma \lua f_{\theta_\circ})\]
for some $\sigma \in \mathrm{Br}_{r_X+1}.$
To show this statement, we introduce the following ordered tuple of real numbers
$\{\phi_\bullet\}_{\idx\in-\is_X}$:
\[\phi_\idx:=
  \begin{cases}
    \theta_\idx & (\theta_\idx>\theta_\circ-\pi) \\
    \theta_\circ-\pi & (\theta_\idx \le \theta_\circ-\pi).
  \end{cases}
\]
We can easily see that $\phi_\bullet$ is $-\is_X$-generic. 
Let $\amut{X}{\phi_\bullet}$ be the corresponding mutation system with the splitting data
$(\tau_{\phi_\bullet}, \lua f_{\phi_\bullet}).$
Note that $V_{\mathscr{M}_X}$ is identified with $\coho{X}$ via the isomorphism $\Phi$
(see Definition \ref{3.5}). 
We define elements of the symmetric group $s, s'$ by
$\tau_{\phi_\bullet}\circ\tau_{\theta_\circ}^{-1}, \tau_{\theta_\bullet}\circ\tau_{\phi_\bullet}^{-1}$
respectively.
We claim that 
\[(\tau_{\phi_\bullet}, \lua f_{\phi_\bullet})=
  \big( s \circ \tau_{\theta_\circ}, (s)_R \lua f_{\theta_\circ} \big), \quad
  (\tau_{\theta_\bullet}, \lua f_{\theta_\bullet})=
  \big( s' \circ \tau_{\phi_\bullet}, (s')_R \lua f_{\phi_\bullet} \big).\]
We first show
$(\tau_{\phi_\bullet}, \lua f_{\phi_\bullet})=
\big(s \circ \tau_{\theta_\circ}, (s)_R \lua f_{\theta_\circ} \big).$
By Lemma \ref{braidact}, we have 
\[\big((s)_R \lua f_{\theta_\circ}\big)_\idx=
\left( \prod^\leftarrow_{i \in I_{\tau_{\theta_\circ}(\idx)}(s)}R_i \right)
\circ \lua f_{\theta_\circ, \idx}.\]
On the other hand, by Corollary \ref{factorcomputation}, we have
\[\lua f_{\phi_\idx, \idx}=
\left( \prod^\leftarrow_{i \in I_{\tau_{\theta_\circ}(\idx)}(s_{\theta_\circ, \phi_\idx})}R_i \right)
\circ \lua f_{\theta_\circ, \idx}.\]
Hence it is sufficient to show that 
\[\left\{\idxx\in -\is_X\middle|
\begin{array}{l}
    \idx <_{\theta_\circ+\pi/2}\idxx \\
   \idxx <_{\phi_\bullet+\pi/2}\idx
\end{array}
\right\}
=
\left\{\idxx\in-\is_X\middle|
\begin{array}{l}
    \idx <_{\theta_\circ+\pi/2}\idxx \\
   \idxx <_{\phi_\idx+\pi/2}\idx
\end{array}
\right\}
\]
for all $\idx\in-\is_X$.
By simple consideration, we see that both sets are equal to 
\[\begin{cases}
  \emptyset & (\idx=\idx_\circ) \\
  \left\{\idx_l\middle|\idx<_{\theta_\circ+\pi/2}\idx_l, l<k\right\}
    & (\idx=\idx_k, \theta_{\idx_k}>\theta_\circ-\pi) \\
  \left\{\idx_l\middle|\idx<_{\theta_\circ+\pi/2}\idx_l, \right\}\cup \left\{\idx_\circ \right\}
    & (\idx=\idx_k,\theta_{\idx_k} \le \theta_\circ-\pi).
  \end{cases}
\]
We next show
$(\tau_{\theta_\bullet}, \lua f_{\theta_\bullet})=
\big( s' \circ \tau_{\phi_\bullet}, (s')_R \lua f_{\phi_\bullet} \big).$
It is sufficient to show that
\[\left\{\idxx\in -\is_X\middle|
\begin{array}{l}
    \idx <_{\phi_\bullet+\pi/2}\idxx \\
   \idxx <_{\theta_\bullet+\pi/2}\idx
\end{array}
\right\}
=
\left\{\idxx\in-\is_X\middle|
\begin{array}{l}
    \idx <_{\phi_\idx+\pi/2}\idxx \\
   \idxx <_{\theta_\idx+\pi/2}\idx
\end{array}
\right\}
\]
for all $\idx\in-\is_X$ as ordered sets,
where order of the left hand side (resp. right hand side) is defined by
$<_{\phi_\bullet+\pi/2}$ (resp. $<_{\phi_\idx+\pi/2}$).
We see that both sets are equal to 
\[\begin{cases}
  \emptyset & (\idx=\idx_\circ) \\
  \emptyset & (\idx=\idx_k, \theta_{\idx_k}>\theta_\circ-\pi) \\
  \left\{ \idx_l\middle|\idx<_{\theta_\circ-\pi/2}\idx_l, l<k \right\} &
  (\idx=\idx_k, \theta_{\idx_k} \le \theta_\circ-\pi).
  \end{cases}
\]
Since $\phi_{\idx_k}=\theta_\circ-\pi$ for $\idx_k$ with $\theta_{\idx_k}<\theta_\circ-\pi$, we see that $<_{\phi_\bullet+\pi/2}$ and $<_{\phi_\idx+\pi/2}$ define the same order on this set,
which completes the proof.
\end{proof}
\begin{remark}
Similarly, we can prove that $X$ satisfies Dubrovin type conjecture if $r_X=1$ and $d_X$ is odd
$($see also Lemma $\ref{6.6}$ and Remark $\ref{ind1})$.
\end{remark}

 \subsection*{Acknowledgement}
          Both of the authors would like to thank Hiroshi Iritani 
          for discussions and encouragement in many occasions, 
          Masahiro Futaki and Yuuki Shiraishi for informal seminar.
          The first author would like to thank Tatsuki Kuwagaki and Kohei Yahiro for discussions.
          The second author would like to express his gratitude for 
          his supervisor Takuro Mochizuki. 
          He is also grateful Claus Hertling and Claude Sabbah for their advise, 
          encouragement, and kindness. 

          The first author was supported by JSPS 
          Grant-in-Aid for Scientific Research numbers 23224002 and 15H02054.
          The second author was supported by Grant-in-Aid for JSPS Research Fellow
          number 16J02453 and the Kyoto Top Global University Project (KTGU).

\bibliographystyle{plain}


\begin{thebibliography}{10}

\bibitem{bjr}
W.~Balser, W.~B. Jurkat, and D.~A. Lutz.
\newblock Birkhoff invariants and {S}tokes' multipliers for meromorphic linear
  differential equations.
\newblock {\em J. Math. Anal. Appl.}, 71(1):48--94, 1979.

\bibitem{baysem}
A.~Bayer and Yu.~I. Manin.
\newblock ({S}emi)simple exercises in quantum cohomology.
\newblock In {\em The {F}ano {C}onference}, pages 143--173. Univ. Torino,
  Turin, 2004.

\bibitem{bonsym}
A.~I. Bondal.
\newblock A symplectic groupoid of triangular bilinear forms and the braid
  group.
\newblock {\em Izv. Ross. Akad. Nauk Ser. Mat.}, 68(4):19--74, 2004.

\bibitem{bonrep}
A.~I. Bondal and M.~M. Kapranov.
\newblock Representable functors, {S}erre functors, and reconstructions.
\newblock {\em Izv. Akad. Nauk SSSR Ser. Mat.}, 53(6):1183--1205, 1337, 1989.

\bibitem{boulie2}
N.~Bourbaki.
\newblock {\em Lie groups and {L}ie algebras. {C}hapters 4--6}.
\newblock Elements of Mathematics (Berlin). Springer-Verlag, Berlin, 2002.
\newblock Translated from the 1968 French original by Andrew Pressley.

\bibitem{calmuk2}
A.~C{\u a}ld{\u a}raru.
\newblock The {M}ukai pairing. {II}. {T}he {H}ochschild-{K}ostant-{R}osenberg
  isomorphism.
\newblock {\em Adv. Math.}, 194(1):34--66, 2005.

\bibitem{calmuk1}
A.~C{\u a}ld{\u a}raru and S.~Willerton.
\newblock The {M}ukai pairing. {I}. {A} categorical approach.
\newblock {\em New York J. Math.}, 16:61--98, 2010.

\bibitem{stenon}
A.~Canonaco and P.~Stellari.
\newblock Non-uniqueness of {F}ourier-{M}ukai kernels.
\newblock {\em Math. Z.}, 272(1-2):577--588, 2012.

\bibitem{chaqua3}
P.~E. Chaput, L.~Manivel, and N.~Perrin.
\newblock Quantum cohomology of minuscule homogeneous spaces {III}.
  {S}emi-simplicity and consequences.
\newblock {\em Canad. J. Math.}, 62(6):1246--1263, 2010.

\bibitem{cioont}
G.~Ciolli.
\newblock On the quantum cohomology of some {F}ano threefolds and a conjecture
  of {D}ubrovin.
\newblock {\em Internat. J. Math.}, 16(8):823--839, 2005.

\bibitem{coaqua}
{T}. {Coates}.
\newblock {The Quantum Lefschetz Principle for Vector Bundles as a Map Between
  Givental Cones}.
\newblock {arXiv:1405.2893}, May 2014.

\bibitem{morsto}
J.~A. Cruz~Morales and M.~van~der Put.
\newblock Stokes matrices for the quantum differential equations of some {F}ano
  varieties.
\newblock {\em Eur. J. Math.}, 1(1):138--153, 2015.

\bibitem{vanona}
L.~{de Thanhoffer de Volcsey} and M.~{Van den Bergh}.
\newblock {On an analogue of the Markov equation for exceptional collections of
  length 4}.
\newblock {arXiv:1607.04246}, July 2016.

\bibitem{dubgeo2}
B.~Dubrovin.
\newblock Geometry of {$2$}{D} topological field theories.
\newblock In {\em Integrable systems and quantum groups ({M}ontecatini {T}erme,
  1993)}, volume 1620 of {\em Lecture Notes in Math.}, pages 120--348.
  Springer, Berlin, 1996.

\bibitem{dubgeo}
B.~Dubrovin.
\newblock Geometry and analytic theory of {F}robenius manifolds.
\newblock In {\em Proceedings of the {I}nternational {C}ongress of
  {M}athematicians, {V}ol. {II} ({B}erlin, 1998)}, number Extra Vol. II, pages
  315--326, 1998.

\bibitem{dubpai}
B.~Dubrovin.
\newblock Painlev\'e transcendents in two-dimensional topological field theory.
\newblock In {\em The {P}ainlev\'e property}, CRM Ser. Math. Phys., pages
  287--412. Springer, New York, 1999.

\bibitem{ggi}
S.~Galkin, V.~Golyshev, and H.~Iritani.
\newblock Gamma classes and quantum cohomology of {F}ano manifolds: gamma
  conjectures.
\newblock {\em Duke Math. J.}, 165(11):2005--2077, 2016.

\bibitem{gi}
S.~{Galkin} and H.~{Iritani}.
\newblock {Gamma conjecture via mirror symmetry}.
\newblock {arXiv:1508.00719}, August 2015.

\bibitem{galdub}
S.~Galkin, A.~Mellit, and M.~Smirnov.
\newblock Dubrovin's conjecture for {$\rm IG(2,6)$}.
\newblock {\em Int. Math. Res. Not. IMRN}, (18):8847--8859, 2015.

\bibitem{givequ}
A.~Givental.
\newblock Equivariant {G}romov-{W}itten invariants.
\newblock {\em Internat. Math. Res. Notices}, (13):613--663, 1996.

\bibitem{guzsto}
D.~Guzzetti.
\newblock Stokes matrices and monodromy of the quantum cohomology of projective
  spaces.
\newblock {\em Comm. Math. Phys.}, 207(2):341--383, 1999.

\bibitem{herupd}
C.~Hertling, Yu.~I. Manin, and C.~Teleman.
\newblock An update on semisimple quantum cohomology and {$F$}-manifolds.
\newblock {\em Tr. Mat. Inst. Steklova}, 264(Mnogomernaya Algebraicheskaya
  Geometriya):69--76, 2009.

\bibitem{herexa}
C.~Hertling and C.~Sabbah.
\newblock Examples of non-commutative {H}odge structures.
\newblock {\em J. Inst. Math. Jussieu}, 10(3):635--674, 2011.

\bibitem{hkr}
G.~Hochschild, B.~Kostant, and A.~Rosenberg.
\newblock Differential forms on regular affine algebras.
\newblock {\em Trans. Amer. Math. Soc.}, 102:383--408, 1962.

\bibitem{huybook}
D.~Huybrechts.
\newblock {\em Fourier-{M}ukai transforms in algebraic geometry}.
\newblock Oxford Mathematical Monographs. The Clarendon Press, Oxford
  University Press, Oxford, 2006.

\bibitem{iriint}
H.~Iritani.
\newblock An integral structure in quantum cohomology and mirror symmetry for
  toric orbifolds.
\newblock {\em Adv. Math.}, 222(3):1016--1079, 2009.

\bibitem{iriqua}
H.~Iritani.
\newblock Quantum cohomology and periods.
\newblock {\em Ann. Inst. Fourier (Grenoble)}, 61(7):2909--2958, 2011.

\bibitem{imm}
H.~Iritani, E.~Mann, and T.~Mignon.
\newblock Quantum {S}erre theorem as a duality between quantum {$D$}-modules.
\newblock {\em Int. Math. Res. Not. IMRN}, (9):2828--2888, 2016.

\bibitem{iwasto}
K.~Iwaki and A.~Takahashi.
\newblock Stokes matrices for the quantum cohomologies of a class of orbifold
  projective lines.
\newblock {\em J. Math. Phys.}, 54(10):101701, 18, 2013.

\bibitem{kkp1}
L.~Katzarkov, M.~Kontsevich, and T.~Pantev.
\newblock Hodge theoretic aspects of mirror symmetry.
\newblock In {\em From {H}odge theory to integrability and {TQFT}
  tt*-geometry}, volume~78 of {\em Proc. Sympos. Pure Math.}, pages 87--174.
  Amer. Math. Soc., Providence, RI, 2008.

\bibitem{kkp2}
L.~Katzarkov, M.~Kontsevich, and T.~Pantev.
\newblock Bogomolov-{T}ian-{T}odorov theorems for {L}andau-{G}inzburg models.
\newblock {\em J. Differential Geom.}, 105(1):55--117, 2017.

\bibitem{kuzhoc}
A.~{Kuznetsov}.
\newblock {Hochschild homology and semiorthogonal decompositions}.
\newblock {arXiv:0904.4330}, April 2009.

\bibitem{kuzbas}
A.~{Kuznetsov}.
\newblock Base change for semiorthogonal decompositions.
\newblock {\em Compos. Math.}, 147(3):852--876, 2011.

\bibitem{kuzexc}
A.~Kuznetsov.
\newblock Exceptional collections in surface-like categories.
\newblock {\em Mat. Sb.}, 208(9):116--147, 2017.

\bibitem{lazbook2}
R.~Lazarsfeld.
\newblock {\em Positivity in algebraic geometry. {II}}, volume~49 of {\em
  Ergebnisse der Mathematik und ihrer Grenzgebiete. 3. Folge. A Series of
  Modern Surveys in Mathematics [Results in Mathematics and Related Areas. 3rd
  Series. A Series of Modern Surveys in Mathematics]}.
\newblock Springer-Verlag, Berlin, 2004.
\newblock Positivity for Vector Bundles, and Multiplier Ideals.

\bibitem{libche}
A.~Libgober.
\newblock Chern classes and the periods of mirrors.
\newblock {\em Math. Res. Lett.}, 6(2):141--149, 1999.

\bibitem{lunlef}
V.~A. Lunts.
\newblock Lefschetz fixed point theorems for {F}ourier-{M}ukai functors and
  {DG} algebras.
\newblock {\em J. Algebra}, 356:230--256, 2012.

\bibitem{mallac}
B.~Malgrange.
\newblock La classification des connexions irr\'eguli\`eres \`a une variable.
\newblock In {\em Mathematics and physics ({P}aris, 1979/1982)}, volume~37 of
  {\em Progr. Math.}, pages 381--399. Birkh\"auser Boston, Boston, MA, 1983.

\bibitem{manqua}
E.~Mann and T.~Mignon.
\newblock Quantum {$\mathscr{D}$}-modules for toric nef complete intersections.
\newblock {\em Internat. J. Math.}, 28(6):1750047, 49, 2017.

\bibitem{tabfro}
M.~Marcolli and G.~Tabuada.
\newblock From exceptional collections to motivic decompositions via
  noncommutative motives.
\newblock {\em J. Reine Angew. Math.}, 701:153--167, 2015.

\bibitem{matgen}
H.~Matsumoto.
\newblock G\'en\'erateurs et relations des groupes de {W}eyl g\'en\'eralis\'es.
\newblock {\em C. R. Acad. Sci. Paris}, 258:3419--3422, 1964.

\bibitem{mukpol}
S.~Mukai.
\newblock Polarized {$K3$} surfaces of genus {$18$} and {$20$}.
\newblock In {\em Complex projective geometry ({T}rieste, 1989/{B}ergen,
  1989)}, volume 179 of {\em London Math. Soc. Lecture Note Ser.}, pages
  264--276. Cambridge Univ. Press, Cambridge, 1992.

\bibitem{noghel}
D.~Yu. Nogin.
\newblock Helices on some {F}ano threefolds: constructivity of semiorthogonal
  bases of {$K_0$}.
\newblock {\em Ann. Sci. \'Ecole Norm. Sup. (4)}, 27(2):129--172, 1994.

\bibitem{pollef}
A.~Polishchuk.
\newblock Lefschetz type formulas for dg-categories.
\newblock {\em Selecta Math. (N.S.)}, 20(3):885--928, 2014.

\bibitem{polche}
A.~Polishchuk and A.~Vaintrob.
\newblock Chern characters and {H}irzebruch-{R}iemann-{R}och formula for matrix
  factorizations.
\newblock {\em Duke Math. J.}, 161(10):1863--1926, 2012.

\bibitem{ramrel}
A.~C. Ramadoss.
\newblock The relative {R}iemann-{R}och theorem from {H}ochschild homology.
\newblock {\em New York J. Math.}, 14:643--717, 2008.

\bibitem{rammuk}
A.~C. Ramadoss.
\newblock The {M}ukai pairing and integral transforms in {H}ochschild homology.
\newblock {\em Mosc. Math. J.}, 10(3):629--645, 662--663, 2010.

\bibitem{reinon}
T.~Reichelt and C.~Sevenheck.
\newblock Non-affine {L}andau-{G}inzburg models and intersection cohomology.
\newblock {\em Ann. Sci. \'Ec. Norm. Sup\'er. (4)}, 50(3):665--753, 2017.

\bibitem{sabiso}
C.~Sabbah.
\newblock {\em Isomonodromic deformations and {F}robenius manifolds}.
\newblock Universitext. Springer-Verlag London, Ltd., London; EDP Sciences, Les
  Ulis, french edition, 2007.
\newblock An introduction.

\bibitem{sabnon}
C.~Sabbah.
\newblock Non-commutative {H}odge structures.
\newblock {\em Ann. Inst. Fourier (Grenoble)}, 61(7):2681--2717, 2011.

\bibitem{sabint}
C.~Sabbah.
\newblock {\em Introduction to {S}tokes structures}, volume 2060 of {\em
  Lecture Notes in Mathematics}.
\newblock Springer, Heidelberg, 2013.

\bibitem{shkhir}
D.~Shklyarov.
\newblock Hirzebruch-{R}iemann-{R}och-type formula for {DG} algebras.
\newblock {\em Proc. Lond. Math. Soc. (3)}, 106(1):1--32, 2013.

\bibitem{swahoc}
R.~G. Swan.
\newblock Hochschild cohomology of quasiprojective schemes.
\newblock {\em J. Pure Appl. Algebra}, 110(1):57--80, 1996.

\bibitem{uedsto2}
K.~Ueda.
\newblock Stokes matrices for the quantum cohomologies of {G}rassmannians.
\newblock {\em Int. Math. Res. Not.}, (34):2075--2086, 2005.

\bibitem{uedsto1}
K.~{Ueda}.
\newblock {Stokes Matrix for the Quantum Cohomology of Cubic Surfaces}.
\newblock {arXiv:math/0505350}, May 2005.

\bibitem{yekcon}
A.~Yekutieli.
\newblock The continuous hochschild cochain complex of a scheme.
\newblock {\em Canad. J. Math.}, 54(6):1319--1337, 2002.

\end{thebibliography}
\end{document}